\documentclass[10pt]{article}

\usepackage[letterpaper, left=1in, top=1in, right=1in, bottom=1in, verbose, ignoremp]{geometry}

\usepackage{url}\RequirePackage[colorlinks,citecolor=blue, linkcolor=blue,urlcolor = blue]{hyperref}
\usepackage{latexsym,amssymb,amsmath,amsfonts,graphicx,color,fancyvrb,amsthm,enumerate,subfigure,mathrsfs}
\usepackage[authoryear,round]{natbib}
\usepackage[dvipsnames]{xcolor}
\usepackage{xy}\xyoption{all} \xyoption{poly} \xyoption{knot}
\usepackage{float}
\usepackage{bm}
\usepackage{bbm}
\usepackage{multicol,multirow}
\usepackage{array}
\usepackage{relsize}
\usepackage{chngcntr}
\usepackage{etoolbox}
\usepackage{caption}
\usepackage{tikz}
\usepackage{tikz-cd}
\usetikzlibrary{patterns}
\usepackage{amsopn}
\usepackage{calc}
\usepackage{algorithm}
\usepackage[noend]{algpseudocode}
\usepackage{hyperref}

\usepackage{textcomp}

\newtheorem{theorem}{Theorem}[]
\newtheorem*{theorem*}{Theorem}
\newtheorem{corollary}[theorem]{Corollary}
\newtheorem{lemma}[theorem]{Lemma}
\newtheorem{proposition}[theorem]{Proposition}

\newtheorem*{claim*}{Claim}
\theoremstyle{definition}
\newtheorem{definition}[theorem]{Definition}
\newtheorem*{definition*}{Definition}
\theoremstyle{AppDefinition}

\theoremstyle{AppClaim}

\theoremstyle{remark}
\newtheorem{remark}[theorem]{Remark}

\newtheorem{example}[theorem]{Example}
\newtheorem*{example*}{Example}

\def\beginmat{ \left( \begin{array} }
\def\endmat{ \end{array} \right) }

\makeatletter
\newcommand*{\op}{%
  \DOTSB
  \mathop{\vphantom{\bigoplus}\mathpalette\matt@op\relax}%
  \slimits@
}
\newcommand\matt@op[2]{%
  \vcenter{\m@th\hbox{\resizebox{\widthof{$#1\bigoplus$}}{!}{$\boxplus$}}}%
}
\makeatother

\renewcommand{\v}{{\bf{v}}}

\newcommand{\x}{{\bf{x}}}
\newcommand{\y}{{\bf{y}}}

\newcommand{\one}{{\bf{1}}}

\def\R{{\mathbb R}}

\newcommand{\dtr}{d_{\mathrm{tr}}}

\newcommand{\T}{\intercal}

\newcommand{\argmin}{\text{argmin}}
\makeatletter
\def\@biblabel#1{}
\makeatother

\makeatletter
\patchcmd{\NAT@citex}
  {\@citea\NAT@hyper@{%
     \NAT@nmfmt{\NAT@nm}%
     \hyper@natlinkbreak{\NAT@aysep\NAT@spacechar}{\@citeb\@extra@b@citeb}%
     \NAT@date}}
  {\@citea\NAT@nmfmt{\NAT@nm}%
   \NAT@aysep\NAT@spacechar\NAT@hyper@{\NAT@date}}{}{}

\patchcmd{\NAT@citex}
  {\@citea\NAT@hyper@{%
     \NAT@nmfmt{\NAT@nm}%
     \hyper@natlinkbreak{\NAT@spacechar\NAT@@open\if*#1*\else#1\NAT@spacechar\fi}%
       {\@citeb\@extra@b@citeb}%
     \NAT@date}}
  {\@citea\NAT@nmfmt{\NAT@nm}%
   \NAT@spacechar\NAT@@open\if*#1*\else#1\NAT@spacechar\fi\NAT@hyper@{\NAT@date}}
  {}{}

\makeatother

\begin{document}
\def\spacingset#1{\renewcommand{\baselinestretch}%
{#1}\small\normalsize} \spacingset{1}
%\spacingset{1.45} % DON'T change the spacing!
\begin{flushleft}
{\Large{\textbf{Tropical Optimal Transport and Wasserstein Distances}}}
\newline
\\
Wonjun Lee$^{1}$, Wuchen Li$^{2}$, Bo Lin$^{3}$, and Anthea Monod$^{4,\dagger}$
\\
\bigskip
\bf{1} Department of Mathematics, University of California, Los Angeles, CA, USA
\\
\bf{2} Department of Mathematics, University of South Carolina, Columbia, SC, USA
\\
\bf{3} School of Mathematics, Georgia Institute of Technology, Atlanta, GA, USA
\\
\bf{4} Department of Mathematics, Imperial College London, UK 
\\
\bigskip
$\dagger$ Corresponding e-mail: a.monod@imperial.ac.uk
\end{flushleft}

%%%%%%%%%%%%%%%%%%%%%%%%%%%%%%%%%%%%%%%%%%%%%%%%%%%

\section*{Abstract}
We study the problem of optimal transport in tropical geometry and define the Wasserstein-$p$ distances in the continuous metric measure space setting of the tropical projective torus.  We specify the tropical metric---a combinatorial metric that has been used to study of the tropical geometric space of phylogenetic trees---as the ground metric and study the cases of $p=1,2$ in detail.  The case of $p=1$ gives an efficient computation of the infinitely-many geodesics on the tropical projective torus, while the case of $p=2$ gives a form for Fr\'{e}chet means and a general inner product structure.  Our results also provide theoretical foundations for geometric insight a statistical framework in a tropical geometric setting.  We construct explicit algorithms for the computation of the tropical Wasserstein-1 and 2 distances and prove their convergence.  Our results provide the first study of the Wasserstein distances and optimal transport in tropical geometry.  Several numerical examples are provided. 

%From the probabilistic perspective, the tropical Wasserstein-1 and 2 distances are natural tools for statistics and probability.    
%The structure laid in our work allows for further rigorous studies in geometry, probability, and statistics on the full sample space of all phylogenetic trees.  

\paragraph{Keywords:} Optimal transport; Tropical geometry; Tropical metric; Tropical projective torus; Wasserstein distances

%%%%%%%%%%%%%%%%%%%%%%%%%%%%%%%%%%%%%%%%%%%%%%%%%%%

\section{Introduction}
\label{sec:intro}

In algebraic geometry, the geometry of zero sets of systems of polynomials---known as algebraic varieties---are studied using commutative algebra.  {\em Tropical geometry} is a variant of this field where the polynomials are defined by the tropical algebra: the tropical sum of two elements is their maximum and the tropical product is their usual sum.  Mathematical objects such as functions and curves evaluated under the tropical algebra are piecewise linear structures, and tropical varieties are polyhedral complexes.  Tropical geometry is an important tool for the study of classical algebraic varieties due to many theoretical coincidences between the two settings.  In addition, tropical geometry possesses the advantage of computational tractability and efficiency, as well as connections to other applied sciences.  For example, it has been used in optimization theory \citep{richter2005first}, dynamic programming in computer science \citep{maclagan2015introduction}, as well as in economics and game theory \citep{lin2017two}.  An application of tropical geometry that has gained much interest is the tropical geometric representation of the space of phylogenetic trees.  In particular, there has very recently been active work in using tropical geometry as a data analytic tool for sets of phylogenetic trees \citep{lin2018tropical, yoshida2019tropical, 10.1093/bioinformatics/btaa564, tang2020tropical}.  In this paper, we study the tropical projective torus, which is the ambient space of phylogenetic trees, and build upon it to provide a set of tools for statistical, probabilistic, and geometric studies using optimal transport theory.

Optimal transport theory arises from a question posed in economics, and specifically, in the allocation of resources.  It deals with optimizing transport modes when geographically displacing resources.  Its mathematical formulation was established in the 18th century and has been well-studied since, resulting in strong connections and mutual implications between the domains of dynamical systems and geometry.  It has also provided important results in applications and computational fields, such as computer science.  An important concept arising from optimal transport is the {\em Wasserstein distances}, which are metrics on probability distributions.  Intuitively, they measure the effort required to recover the probability mass of one distribution in terms of an efficient reconfiguration of the other.  As such, Wasserstein distances broaden the scope of optimal transport theory to probability theory.  Additionally, they have been exploited to move further beyond these realms to solve concrete problems in inferential statistics, such as in \cite{panaretos2020invitation}.  Establishing Wasserstein distances in tropical geometric settings thus provides a framework for a vast body of existing results in these related fields to be applicable to the important problem of statistical inference and data analysis in applied tropical geometric settings by providing a setting for the study of probability measures and distributions.  Additionally, it provides an alternative mechanism to study geometric aspects of tropical objects and spaces.

Connecting algebraic theory to optimal transport theory is a new direction of research with very recent contributions involving algebraic geometry and algebraic topology.  In \cite{otvariety}, the Wasserstein distance between a probability distribution and an algebraic variety is minimized via transportation polytopes.  In topological data analysis, where algebraic topology is leveraged to reduce the dimensionality of complex data spaces and extract shape features within the data, optimal transport theory has improved computational efficiency \citep{lacombe2018large} and also has been used to study geometric aspects of algebraic topological invariants \citep{divol2019understanding}.  A prior transportation problem (distinct from the optimal transport setting) has been previously considered in tropical geometry by \cite{richter2005first}.  Our work in this paper presents the first connection between tropical geometry and optimal transport theory.  Specifically, we consider an infinite metric measure space in a continuous tropical geometric setting endowed with a combinatorial ground metric.  Numerical computations of optimal transport with various ground metrics has been recently studied in the continuous setting and shown to be efficient \citep{benamou2016numerical, Li2018}.  Additionally, studying the optimal transport problem provides a computational framework for the probability density space, which also encodes the geometry of sample space \citep{Lafferty, OV, otto, villani2008optimal}.  In solving the optimal transport problem, we thus define tropical Wasserstein distances and provide algorithms for our proposed tropical Wasserstein distances.  Collectively, these results offer tools for probabilistic, statistical, and geometric inference in a tropical geometric setting, which then may be translated to other applications where tropical geometry plays an important computational and interpretive role.

The remainder of this paper is organized as follows.  Section \ref{sec:trop} gives an overview of tropical geometry and the {\em tropical projective torus} as our ground space of interest.  We present and review properties of the tropical metric, which endows this space with a metric structure; we also give some variational forms for the tropical metric.  Section \ref{sec:opt} overviews the problem of optimal transport and the role of the Wasserstein distances in this framework.  We then define the tropical Wasserstein-$p$ distance, with the tropical metric as the ground metric and the tropical projective torus as the ground space; we also give variational forms of the tropical Wasserstein distance.  We study the specific cases of $p=1$ and $2$: the $p=1$ case gives a method for computing all infinitely many tropical geodesics, while in the case of $p=2$, the Wasserstein metric is amenable to statistical analysis by providing an inner product structure on probability measures on the tropical projective torus.  
%Building on the $p=2$ case, we also discuss geometric aspects and gradient flows arising from differential equations on the probability space supported on palm tree space under the Tropical Wasserstein-2 distance. 
Section \ref{sec:computation} gives algorithms to explicitly compute the tropical Wasserstein-$p$ distances, while Section \ref{sec:num_exp} presents the results of several numerical experiments implementing our proposed algorithms.  We close the paper with a discussion in Section \ref{sec:discussion} on future research stemming from the work presented in this paper.

%%%%%%%%%%%%%%%%%%%%%%%%%%%%%%%%%%%%%%%%%%%%%%%%%%%

\section{Tropical Geometry, the Tropical Projective Torus,\\ and the Tropical Metric}
\label{sec:trop}

In this section, we give the basics of tropical geometry that are relevant for our work.  We then present the tropical projective torus as our ground space of interest, and the tropical metric as the ground metric on this space.  We also give alternative versions of the metric in terms of variational forms.  This is the metric with respect to which we will define the tropical optimal transport problem and the tropical Wasserstein-$p$ distances.

\subsection{Essentials of Tropical Geometry: Tropical Algebra}

Tropical geometry may be seen as a subdiscipline of algebraic geometry.  In the latter, the zero sets of systems of polynomial equations are studied using algebraic methods; in the former, these polynomials are defined via the {\em tropical semiring}, $(\R \cup \{-\infty \}, \boxplus, \odot)$ where addition between two elements is given by their max and multiplication is given by their sum:
\begin{align*}
a \boxplus b & := \max(a,b),\\
a \odot b & := a + b.
\end{align*}
Notice that tropical subtraction is not defined, therefore resulting in a semiring, rather than a ring.  Both operations of the semiring are commutative and associative; multiplication distributes over addition.  {\em Tropicalization} refers to interpreting classical arithmetic operations with their tropical counterparts.  Using these operations, lines, polynomials, and other more general mathematical constructions can be built, which will result in ``skeletal" piecewise linear structures.

\subsection{The Tropical Projective Torus}
\label{subsec:trop_proj}

Tropical geometry naturally gives rise to polyhedral structures.  The interplay between algebraic geometry and polyhedral geometry results in new interpretations of important concepts which form the building blocks for the study of tropical geometry.  

An important example is the reinterpretation of a fundamental object in computational algebraic geometry---the Gr\"{o}bner basis.  A Gr\"{o}bner basis is a particular generating set of an ideal in a polynomial ring over a field; computing Gr\"{o}bner bases is one of the main approaches in solving systems of polynomials, which is a central problem in algebraic geometry.  Reinterpreting Gr\"{o}bner bases using valuations (functions over fields that give a notion of its size) gives rise to {\em Gr\"{o}bner complexes}.  Gr\"{o}bner complexes lead to universal Gr\"{o}bner bases, which are analogs to tropical bases; see \cite{maclagan2015introduction} for full details of this construction.  The Gr\"{o}bner complex is thus a fundamental object in tropical geometry; it is a polyhedral complex constructed for a homogeneous ideal in the polynomial ring $K[x_0, x_1, \ldots, x_n]$ over a field $K$.  The ambient space of a Gr\"{o}bner complex is the {\em tropical projective torus}, denoted by $\R^{n+1}/\R\one$.  In this paper, we consider the tropical projective torus as our ground space of interest.  

%where $N$ is the fixed number of leaves in a tree, and $n+1 := \binom{N}{2}$ \citep{maclagan2015introduction, lin2018tropical, doi:10.1137/16M1079841}.  

The tropical projective torus is the quotient space that identifies vectors differing from each other by tropical scalar multiplication (or classical addition).  It is generated by the following equivalence relation $\sim$ on $\R^{n+1}$:
\[x\sim y \Leftrightarrow x_{1} - y_{1} = x_{2} - y_{2} = \cdots = x_{n+1} - y_{n+1}. \]
Mathematically, $\R^{n+1}/\R\one$ is constructed in the same manner as the complex torus: take a lattice $\Lambda \in \mathbb{C}^{n+1}$ as a real vector space, then the complex torus is $\mathbb{C}^{n+1}/\Lambda$. For $x\in \R^{n+1}$, let $\bar{x}$ be its image in $\R^{n+1}/\R\one$.  The tropical projective torus identifies with $\R^n$ by taking representatives of the equivalence classes whose last coordinate is zero:
\begin{equation}
\label{eq:embedding}
\bar{x} \mapsto (x_{1} - x_{n+1}, \, x_{2} - x_{n+1}, \ldots,\, x_{n} - x_{n+1}).
\end{equation}
We denote an element in $\R^{n+1}$ by $x$, an element in $\R^{n+1}/\R\one$ by $\bar x$, and an element in $\R^n$ by $\x = (x_1 - x_{n+1},\, \ldots,\, x_n - x_{n+1})$---which is the image of $\bar{x}$ in $\R^n$.

\paragraph{The Space of Phylogenetic Trees.}

One important practical example that arises in the tropical projective torus is the space of phylogenetic trees, $\mathcal{T}_N$ (where $N$ is the fixed number of leaves in a tree).  \cite{Speyer2004} identify an equivalence between the space of all phylogenetic trees and a tropical geometric space via a homeomorphism \citep{maclagan2015introduction, lin2018tropical, doi:10.1137/16M1079841}.  The space of phylogenetic trees is contained within the tropical projective torus.  In other words, the tropical projective torus is also the ambient space of phylogenetic trees.  Although the space of phylogenetic trees is a proper subset of the tropical projective torus, it possesses a very complex structure that is not yet well understood.  In particular, it is connected and possesses a polyhedral structure, but is not convex \citep{lin2018tropical, doi:10.1137/16M1079841}.  Additionally, trees are defined by a specific combinatorial condition, which makes the precise characterization of the space of phylogenetic trees and establishing its boundary within the tropical projective torus difficult.  The dimension of tree space also is lower than the tropical projective torus: its dimension grows linearly in the number of leaves in a tree, while for the tropical projective torus, the dimension grows quadratically.

\subsection{The Tropical Metric}
\label{subsec:trop_metric}

The tropical projective torus $\R^{n+1}/\R\one$ becomes a metric space when endowed with a {\em generalized Hilbert projective metric} function \citep{COHEN2004395, AKIAN20113261}, which is a combinatorial metric that is tropical in nature.  It has been referred to as the {\em tropical metric} in recent literature \citep{doi:10.1137/16M1079841, lin2018tropical}.  Our work here is based on the ambient tree space given by the tropical projective torus endowed with the tropical metric.

\begin{definition}
\label{def:trop_metric}
For a point $x \in \R^{n+1}$, denote its coordinates by $x_1, x_2, \ldots, x_{n+1}$ and its representation in the tropical projective torus $\R^{n+1}/\R\one$ by $\bar{x}$.  The {\em tropical metric} on $\R^{n+1}/\R\one$ is given by
\begin{align*}
d_{\mathrm{tr}}(\bar{x}, \bar{y}) & := \max_{1\leq i\leq n+1}(x_{i} - y_{i}) - \min_{1 \leq i\leq n+1}(x_{i} - y_{i}).
%& = \max_{1\le i\le j\le n}{\left| (x_{i} - y_{i}) - (y_{j} - x_{j}) \right|}.
\end{align*}
When considering the representatives of the equivalence classes as in (\ref{eq:embedding}), the tropical metric translates to the following between $\R^{n+1}/\R\one$ and $\R^n$: for $\bar{x}, \bar{y} \in \R^{n+1}/\R\one$ and $\x, \y \in \R^n$, 
$$
\dtr(\bar{x}, \bar{y}) := \max\Big\{\max_{1\leq i < j \leq n}\big|(\x_i - \y_i)-(\x_j - \y_j)\big|,\, \max_{1\leq i \leq n}|\x_i - \y_i|\Big\} := \dtr(\x, \y).
$$
\end{definition}

Figure \ref{fig:iso} illustrates the relationship where $\R^{n+1}$ identifies with $\R^{n+1}/\R\one$ by the equivalence relation $\sim$; $\R^{n+1}/\R\one$ then embeds into $\R^{n}$.  The metric $\dtr$ is defined on $\R^{n+1}/\R\one$ and has a representation in $\R^{n}$; it is an isometry from $\R^{n+1}$ to $\R^{n+1}/\one$ to $\R^{n}$.  Again, recall the notation that an element in $\R^{n+1}$ is denoted by $x$, an element in $\R^{n+1}/\R\one$ is denoted by $\bar x$, and an element in $\R^n$ is denoted by $\x = (x_1 - x_{n+1}, \ldots, x_n - x_{n+1})$.

\begin{lemma}
On $\R^{n+1}/\R\one$, we have the following alternate expression for the tropical metric:
	\begin{equation*}
		d_{\mathrm{tr}}(\bar{x}, \bar{y}) = \max_{1\le i\le j\le n+1}{\left| (x_{i} - y_{i}) - (x_{j} - y_{j}) \right|}.
	\end{equation*}
\end{lemma}

%For $n\in \N$, we consider the quotient space $\R^{n+1}/\R\one$ defined by the equivalence relation $\sim$ on $\R^{n+1}$:

%\[x\sim y \Leftrightarrow x_{1} - y_{1} = x_{2} - y_{2} = \cdots = x_{n+1} - y_{n+1}. \]

\begin{proposition}{\cite[Proposition 17]{lin2018tropical}}
	$d_{\mathrm{tr}}(\cdot,\cdot)$ is a well-defined metric function on $\R^{n+1}/\R\one$.
\end{proposition}

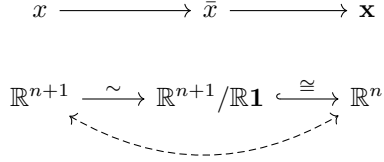
\begin{figure}
$$
\begin{tikzcd}
x \arrow[r] & \bar{x} \arrow[r] & \x \\
\R^{n+1} \arrow[r, "\sim"] \arrow[urrd, bend right, leftrightarrow, dashed] & \R^{n+1}/\R\one \arrow[r, hook, "\cong"] & \R^{n}
\end{tikzcd}
$$
\caption{Diagram illustrating embedding of and relationships between Euclidean spaces and the tropical projective torus.  The dashed arrow represents the isometry of the tropical metric between all three spaces.}
\label{fig:iso}
\end{figure}

\subsection{Variational Forms of the Tropical Metric}
\label{subsec:variational}

%For the quotient space $\R^{n+1}/\R\one$, for convenience we may regard its points as in $\R^{n}$: we can always take representatives of the equivalence classes whose last coordinate is zero. Now we would like to find a variational form of $d_{\mathrm{tr}}(\cdot,\cdot)$.

It turns out that the tropical metric may be considered in terms of unknown functions and corresponding differential equations, which provides an alternative formulation for the tropical metric in terms of a variational form.  Variational forms are useful in computational studies, since numerically, it is often easier to find solutions to variational problems rather than differential equations.  As we will see further on, this turns out to be an important advantage in explicit computations of the tropical Wasserstein distances and associated results.

\paragraph{Notation.}
We use the $+$ and $-$ superscript notation as follows: 
\begin{align*}
(\cdot)^+ &:= \max(\cdot, 0),\\
(\cdot)^- &:= \min(\cdot, 0).
\end{align*}

\begin{proposition}\label{prop:vf1}
	For $\bar{x}, \bar{y}\in \R^{n+1}/\R\one$, we have
	
	\begin{equation}\label{dist}
	d_{\mathrm{tr}}(\bar{x}, \bar{y}) = \left(
	        \begin{aligned}
	            \text{\rm minimize } \quad & \int_{0}^{1} L_{\mathrm{tr}}\big(\mathbf{v}(t) \big) dt, \\
	            \text{\rm subject to:} \quad & \frac{\textrm{d}{\bf z}}{\textrm{d}t}=\v(t),\,\, {\bf z}(0) = \x,\,\, {\bf z}(1) = \y
	        \end{aligned}
	\right), 
	\end{equation}
	where ${\bf v,z}: [0,1] \to \R^{n}$ and we define the {\em tropical Lagrangian} $L_{\mathrm{tr}}(\cdot)$ as the tropical norm for $\mathbf{a} \in \R^{n}$ as follows:
	\begin{equation}\label{L}
	\begin{split}
	L_{\mathrm{tr}}({\bf a}) = \| {\bf a} \|_{\mathrm{tr}} & = \max\Big(\max_{1\le i\le n}(\mathbf{a}_{i}),0\Big) - \min\Big(\min_{1\le i\le n}(\mathbf{a}_{i}),0\Big)\\
	& = {\max_{1\le i\le n}(\mathbf{a}_{i})}^+ - {\min_{1\le i\le n}(\mathbf{a}_{i})}^- \quad \forall~\mathbf{a} \in \R^{n}.
	\end{split}
	\end{equation}	
\end{proposition}

\begin{proof}
	Let $D = \{\x_{i}-\y_{i}\mid 1\le i\le n\} \cup \{0\}$. By Definition \ref{def:trop_metric}, $d_{\mathrm{tr}}(\bar{x}, \bar{y}) = \max(D) - \min(D)$. Hence
$$
d_{\mathrm{tr}}(\bar{x}, \bar{y}) = 
	d_{\mathrm{tr}}(\bar{y}, \bar{x}) =
	{\max_{1\le i\le n}(\x_{i}-\y_{i})}^+ - {\min_{1\le i\le n}(\x_{i}-\y_{i})}^-.
$$

%	\[d_{\mathrm{tr}}(\bar{x}, \bar{y}) = 
%	d_{\mathrm{tr}}(\bar{y}, \bar{x}) =
%	\max\Big(\max_{1\le i\le n}(\x_{i}-\y_{i}),0\Big) - \min\Big(\min_{1\le i\le n}(\x_{i}-\y_{i}),0\Big). \]

	First, let ${\bf z}(t) = t\cdot \y + (1-t) \cdot \x$, then $\v(t)$ is the constant vector $\y - \x$, and the integral $\int_{0}^{1}{\| \mathbf{v}(t) \|_{\mathrm{tr}} dt}$ becomes $L_{\mathrm{tr}}(\y - \x) = d_{\mathrm{tr}}(\bar{x},\bar{y})$. Second, in order to show that 
	
	\[\int_{0}^{1}{\| \mathbf{v}(t) \|_{\mathrm{tr}} dt} \ge d_{\mathrm{tr}}(\bar{x},\bar{y}), \]
it suffices to show that the integral is always no less than any of $|\y_{i} - \x_{i}|$ and $\left|\left(\y_{i} - \x_{i}\right) - \left(\y_{j} - \x_{j}\right) \right|$ where $1\le i,j\le n$.
	
	For $1\le i\le n$, by definition of $L_{\mathrm{tr}}$ we have
	
	\[\|\mathbf{v}(t)\|_{\mathrm{tr}} \ge |\v(t)_{i}-0| = |\v(t)_{i}|. \]
	
	Now consider the function $f_{i}: [0,1] \to \R$ given by $f_{i}(t) = \mathbf{z}(t)_{i}$. Then $\displaystyle \v(t)_{i} = \frac{df_{i}}{dt}(t)$, which gives
	\begin{equation}\label{eq:ftc}
		\int_{0}^{1}{\v(t)_{i} dt} = f_{i}(1) - f_{i}(0) = \y_{i} - \x_{i}
	\end{equation}
and
	\[\int_{0}^{1}{\| \mathbf{v}(t) \|_{\mathrm{tr}} dt} \ge \int_{0}^{1}{|\v(t)_{i}| dt} \ge \left|\int_{0}^{1}{\v(t)_{i} dt} \right| = \left|\y_{i} - \x_{i}\right|. \]
	
	Similarly, for any $1\le i,j\le n$, by definition of $L_{\mathrm{tr}}$, we have
	
	\[\| \mathbf{v}(t) \|_{\mathrm{tr}} \ge |\v(t)_{i} - \v(t)_{j}|.\]
	
	By (\ref{eq:ftc}), we get
	
	\begin{align*}
	\int_{0}^{1}{\| \mathbf{v}(t) \|_{\mathrm{tr}} dt} \ge \int_{0}^{1}{|\v(t)_{i} - \v(t)_{j}| dt} & \ge \left|\int_{0}^{1}{\left(\v(t)_{i} - \v(t)_{j}\right) dt}\right|\\
	& = \left|\left(\y_{i} - \x_{i}\right) - \left(\y_{j} - \x_{j}\right)\right|. 
	\end{align*}
\end{proof}

\begin{example}
When $n=2$, 
    \[ 
    L_{\mathrm{tr}}({\bf a}) =
    \begin{cases}
        a_{1}, & \text{ if } a_{1}\ge a_{2}\ge 0; \\
        a_{2}, & \text{ if } a_{2}\ge a_{1}\ge 0; \\
        -a_{1}, & \text{ if } 0 \ge a_{2}\ge a_{1}; \\
        -a_{2}, & \text{ if } 0 \ge a_{1}\ge a_{2}; \\
        a_{1}-a_{2}, & \text{ if } a_{1}\ge 0 \ge a_{2}; \\   
        a_{2}-a_{1}, & \text{ if } a_{2}\ge 0 \ge a_{1}. 
    \end{cases}
    \]
\end{example}

The above variational form (\ref{dist}) of $d_{\mathrm{tr}}(\cdot,\cdot)$ may be further generalized as follows.

\begin{corollary}\label{prop:vf2}
	For $\bar{x}, \bar{y}\in \R^{n+1}/\R\one$, let $L_{\mathrm{tr}}$ be the same as in Proposition \ref{prop:vf1}. For $p>1$, we have
		
		\begin{equation}\label{dist2}
		d_{\mathrm{tr}}(\bar{x}, \bar{y}) = \left(
		        \begin{aligned}
		            \text{\rm minimize } \quad & \left(\int_{0}^{1}{L_{\mathrm{tr}}\big( \mathbf{v}(t) \big)^{p}}dt\right)^{\frac{1}{p}} \, \\
		            \text{\rm subject to:} \quad & \frac{\textrm{d}{\bf z}}{\textrm{d}t}=\v(t),\,\, {\bf z}(0) = \x,\,\, {\bf z}(1) = \y
		        \end{aligned}
		\right).
		\end{equation}
\end{corollary}

\begin{proof}
	When ${\bf z}(t) = t\cdot \y +(1-t)\cdot \x$, $\v(t)$ is still the constant $\y- \x$ and the equality still holds. In addition, by the H\"{o}lder inequality, 
	
	\[\left(\int_{0}^{1}{\big\| \mathbf{v}(t) \big\|_{\mathrm{tr}}^{p} dt}\right)^{\frac{1}{p}} \ge \int_{0}^{1}{\big| \| \mathbf{v}(t) \|_{\mathrm{tr}} \big| dt}. \]
	
	Hence for any ${\bf z}:[0,1]\to \R^{n}$ and ${\bf v}(t) = \frac{d{\bf z}}{dt}$,
	
	\[\left(\int_{0}^{1}{\big\| \mathbf{v}(t) \big\|_{\mathrm{tr}}^{p} dt}\right)^{\frac{1}{p}} \ge d_{\mathrm{tr}}(\x,\y),\]
as in Definition \ref{def:trop_metric}.
\end{proof}

%%%%%%%%%%%%%%%%%%%%%%%%%%%%%%%%%%%%%%%%%%%%%%%%%%%

\section{Optimal Transport and the Tropical Wasserstein-$p$ Distances}
\label{sec:opt}

We now give a brief background on and a description of the problem of optimal transport; we also formally present the setting of the optimal transport problem specific to our work.

The question underlying the theory of optimal transport can be posed in a very basic and intuitive manner as follows: What is the most efficient way to move a given pile of dirt from one location to another?  The total volume of the dirt must remain intact, but the shape and form of the pile may change during transportation and arrive at its location in a differently shaped pile.  This problem has been recast mathematically in various formulations with various assumptions.  There is a vast literature of historical as well as technical aspects and perspectives on the optimal transport problem; see for example \cite{villani2003topics, villani2008optimal, ambrosio2013user} for detailed discussions.  

\subsection{Optimal Transport and Probability}

Adapting the intuitive description of the optimal transport problem above to a more mathematically formal setting, we may view the pile of dirt as a probability measure to be transported over a space---or alternatively, one probability distribution to be transformed into another---which gives us a probabilistic and statistical perspective on the problem.

A key factor in solving the optimal transport problem is the {\em cost function}, which gives the cost of moving the pile of dirt, or the transporting the probability measure.  Mathematically, this is generally a function of two variables---an origin or ``start" location and destination or ``end" location---which maps to the positive real line to give the cost, and may take into account any number of factors.  In the simplest case, however, when the cost of moving the pile of dirt from its origin to destination is nothing more than the distance between the origin and destination, the solution to the optimal transport problem yields the {\em Wasserstein distance} (for a fixed dimension).  Intuitively, the Wasserstein distance gives the minimum cost of transforming one probability distribution into another.  This minimum cost is simply the ``amount of dirt" to be transported, multiplied by the mean distance it must be moved.  In the case of probability distributions that contain a total mass of 1, the minimum cost is therefore simply the mean distance it must be moved.  More precisely, the Wasserstein distance is a distance function for probability distributions defined on a given metric space, referred to as the ground space and the associated metric is referred to as the ground metric; these concepts are formalized further on in Definition \ref{def:wass}.  The Wasserstein distance is thus a useful tool for comparing distributions.

\paragraph{Specific Setting.}

In our work, the ground space is the tropical projective torus and the ground metric is the tropical metric.  We consider the set of all probability measures on the tropical projective torus, which exist and are well-defined \citep{lin2018tropical}, as a space.  This work defines and constructs Wasserstein distances as a metric on these probability measures associated with the tropical projective torus.  Figure \ref{fig:Wass_fig} provides a conceptual illustration of the relationship between the ground space, equipped with a ground metric, and the Wasserstein space of probability measures over the ground space, equipped with the Wasserstein distance. 

\begin{figure}
\begin{center}
\includegraphics[scale=0.5]{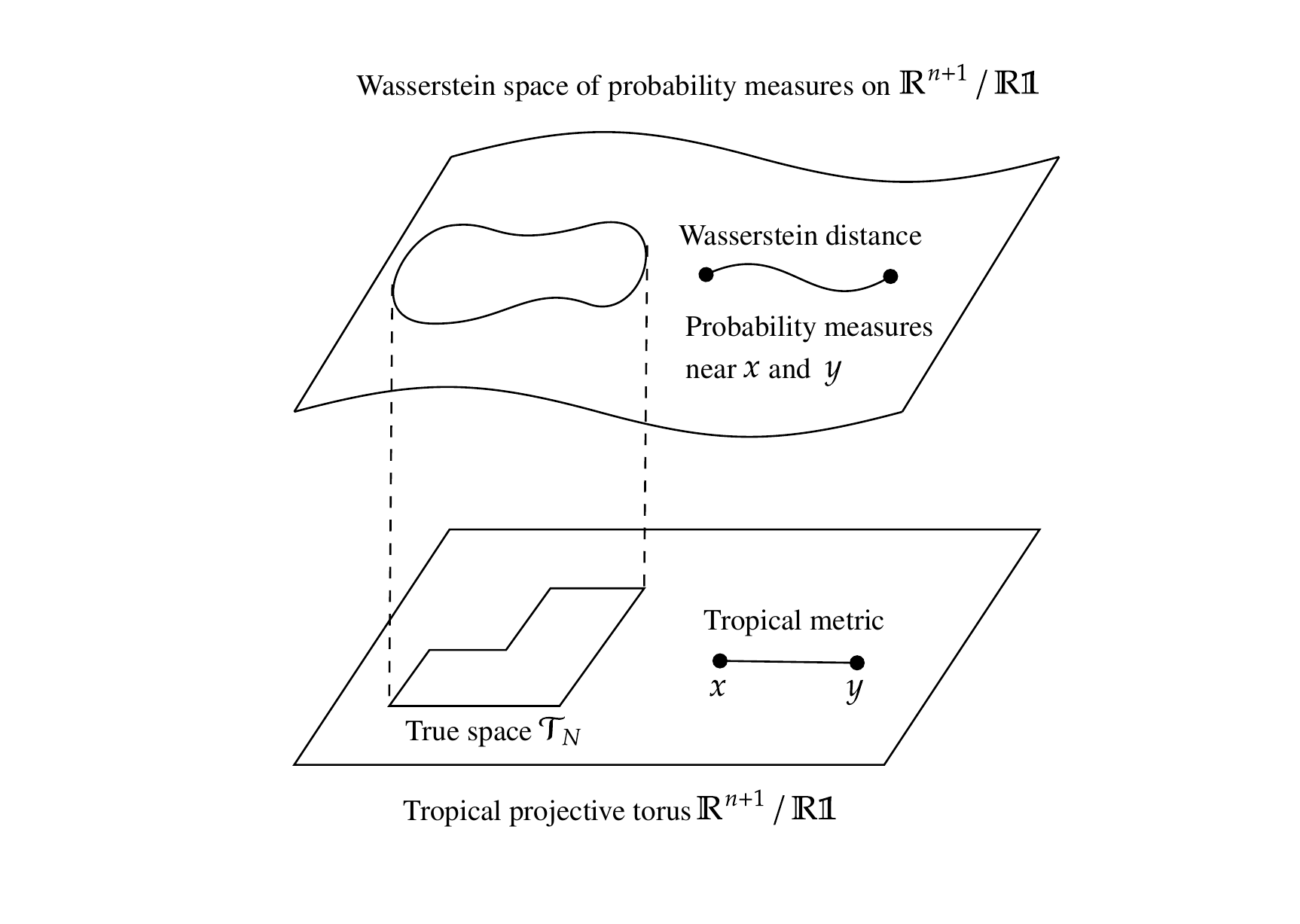}
\caption{Illustrative figure of the relationship between the ground space and the Wasserstein space of probability measures.  Here, the plane below depicts the tropical projective torus $\R^{n+1}/\R\one$ is the ground space; it is equipped with the tropical metric.  This space admits well-defined probability measures \citep{lin2018tropical}.  Collecting these probability measures as a separate space yields the space of probability measures on $\R^{n+1}/\R\one$; in this figure, it is depicted in the manifold above.  This space can be equipped with a particular metric---the Wasserstein distance.  The Wasserstein distance is therefore defined on the space of probability measures on $\R^{n+1}/\R\one$; it measures distances between probability measures on the tropical projective torus.  In this illustrative figure, we also show the space of phylogenetic trees with $N$ leaves, $\mathcal{T}_N$, as a figurative proper non-convex subset of the tropical projective torus $\R^{n+1}/\R\one$.  The probability measures associated with this specific subset of $\R^{n+1}/\R\one$ are depicted in the Wasserstein space of probability measures above, which is also non-convex (see Remark \ref{remark:convex}).}
\label{fig:Wass_fig}
\end{center}
\end{figure}

\paragraph{Wasserstein Distances as Metrics Between Probability Distributions.}

Although other metrics for probability distributions exist in the literature on mathematical statistics, the Wasserstein distance possesses desirable computational and intuitive properties.  To illustrate a few such properties, let us consider random variables $X, Y$ defined on $\R^d$ distributed as $X \sim P$ and $Y \sim Q$ with densities $p$ and $q$, respectively.  Three commonly-used measures for distances between $P$ and $Q$ are total variation, $\frac{1}{2}\int |p-q|$; Hellinger, $\sqrt{\int (\sqrt{p} - \sqrt{q})^2}$; and $L_2$, $\int (p-q)^2$.
%\begin{enumerate}[(i)]
%\item total variation: $\displaystyle \frac{1}{2}\int |p-q|$;
%\item Hellinger: $\displaystyle \sqrt{\int (\sqrt{p} - \sqrt{q})^2}$;
%\item $L_2$: $\displaystyle \int (p-q)^2$.
%\end{enumerate}

When comparing one discrete versus one continuous distribution, these distances yield results that are not very informative.  Let $P$ be uniform on $[0,1]$, and let $Q$ be uniform on $\{0,\, 1/n,\, 2/n,\, \ldots,\, 1\}$.  The total variation distance between these distributions is 1, which is the total size of each of the two sets, and the largest that any distance can be, while the Wasserstein distance is $1/n$.

These distances also do not take into account the underlying geometry of the space on which the distributions are defined.  Consider the three densities $p_1$, $p_2$ and $p_3$ shown in Figure \ref{fig:density_ex}.  We have
$$
\int |p_1 - p_2| = \int |p_1 - p_3| = \int |p_2 - p_3|,
$$
and similar results for the Hellinger and $L_2$ distances, however, intuitively, we would like to think of $p_1$ and $p_2$ being more similar and and hence closer to each other than to $p_3$.  The Wasserstein distance is able to make this distinction.

\begin{figure}
\begin{center}
\includegraphics[scale=0.75]{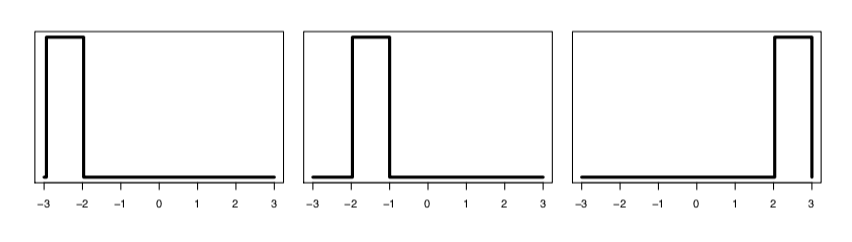}
\caption{Three example densities $p_1$, $p_2$, $p_3$.  This figure appears in \cite{wasserman_opt}.  The total variation, Hellinger, and $L_2$ distances between these three densities are the same, while the Wasserstein distance between $p_1$ and $p_2$ is smaller than that between either $p_1$ or $p_2$ and $p_3$.}
\label{fig:density_ex}
\end{center}
\end{figure}

In computing a distance between distributions, we arrive at some measure of their similarity or dissimilarity, but the total variation, Hellinger, $L_2$, and other distances do not provide any information on how or why the distributions are qualitatively different.  Perhaps the most helpful property of the Wasserstein distance is that, in addition to a measure of distance between the distributions, we also obtain a map that describes how $P$ morphs into $Q$.  This map is known as a {\em transport plan}.

In addition to the illustrative examples discussed above, there are other desirable computational and statistical properties of the Wasserstein distance, such as stability to small perturbations and a well-behaved and intuitive Wasserstein Fr\'{e}chet mean.  Further details and more complete discussions on statistical aspects of the Wasserstein distance can be found in \cite{doi:10.1146/annurev-statistics-030718-104938, wasserman_opt}.  

Aside from statistical aspects, there also exist other analytic advantages of the Wasserstein distances, depending on the context.  For instance, the Wasserstein distances' intimate connection to optimal transport problems inherently make them natural tools in these and other settings with foundations in partial differential equations.  %Additionally, a particularly important property not discussed in this paper which makes Wasserstein distances particularly attractive computation-wise is their {\em duality property}; see, for example, \cite{villani2008optimal} for further details.

\paragraph{Wasserstein Distances and Phylogenetic Trees.}

Wasserstein distances have been previously studied in the context of phylogenetic trees.  A single tree itself may be treated as a metric space, for instance, by considering genetic distances which measure distances between pairs of sequences on a single tree; the metric here is defined {\em within} the tree itself.  When considering a single tree, the context is a finite metric space.  Wasserstein distances have been defined and studied in these contexts, such as in \cite{doi:10.1111/j.1467-9868.2011.01018.x}, where probability distributions giving rise to individual trees are compared.  \cite{kloeckner_2015} studies geometric properties of measures on equidistant trees (i.e., rooted trees with equal branch lengths from the root to all leaves) using Wasserstein distances.  For finite spaces, \cite{doi:10.1111/rssb.12236} conduct statistical inference studies for empirical Wasserstein metrics computed from datasets.  Very recently, \cite{NIPS2019_9396} studied the sliced formulation of optimal transport---developed to alleviate computational and statistical drawbacks of optimal transport theory---on tree metrics.  \cite{sato2020fast} furthermore propose an extremely fast algorithm that solves the optimal transport problem to compute Wasserstein distances on a tree with one million nodes in less than one second.  The setting of these works all differ from the study of Wasserstein distances on the {\em space} of phylogenetic trees.

In the context of tree spaces, other probability-based distances between trees have also been proposed \citep{10.1093/sysbio/syx080, garba2020information}.  These are related, but are nevertheless strictly different from the notion of distances between probability measures over tree space.  The contributions of these works are classical measures between probability distributions on genetic sequences that make up trees, which then induce probabilistic distances between trees, including Hellinger distances and Kullback--Leibler divergences.  Kullback--Leibler divergences measure the difference in terms of information gain between models of statistical inference \citep{Kullback:1951aa}.  Outside the scope of interest of this paper, other tree spaces have also been proposed that are not probability-based; an example of a combinatorial construction based on posets that turns out to be related to tree-reconstruction using Markov processes is the edge-product space \citep{MOULTON2004710, GILL2008158}. 

\subsection{Formalizing the Optimal Transport Problem and Defining the Wasserstein-$p$ Distances.}

\cite{monge1781memoire} is largely recognized to have provided the first mathematical formalization of the optimal transport problem described above, while the subsequent probabilistic reinterpretation by \cite{kantorovich1942translocation} lead to a fundamental computational breakthrough that seeded the development of linear optimization.  As such, the statement of the mathematical optimal transport problem is often referred to as the Monge--Kantorovich transport problem and presented in the setting of measure theory.  We now give an overview of this presentation.

\begin{definition}
Let $\Omega$ and $\Omega'$ be separable metric spaces that are Radon spaces (that is, any probability measure on each space is a Radon measure).  Let $c: \Omega \times \Omega' \rightarrow [0, \infty]$ be a Borel-measurable cost function.  For $\rho^0 \in \mathscr{P}(\Omega)$ and $\rho^1 \in \mathscr{P}(\Omega')$ where $\mathscr{P}(\cdot)$ denotes the collection of probability measures on the respective spaces, the {\em Monge--Kantorovich transport problem} is to find a probability measure $\pi$ on $\Omega \times \Omega'$ such that
$$
\inf \Bigg\{ \int_{\Omega \times \Omega'} c(x, y) \mathrm{d}\pi(x,y) \,\, \bigg| \,\, \pi \in \Pi(\rho^0, \rho^1) \Bigg\}
$$
is achieved.  Here, $\Pi(\rho^0, \rho^1)$ denotes the collection of all probability measures on $\Omega \times \Omega'$ with marginal measures $\rho^0$ on $\Omega$ and $\rho^1$ on $\Omega'$.  

%By duality, this problem may be equivalently expressed as
%$$
%\sup \Bigg( \int_M \varphi(x)\mathrm{d}\mu(x) + \int_{M'} \psi(y)\mathrm{d}\nu(y) \Bigg)
%$$
%over all pairs of bounded continuous functions $\varphi: M \rightarrow \R$ and $\psi: M' \rightarrow \R$ such that $\varphi(x) + \psi(y) \leq c(x,y)$.  
\end{definition}

When the cost function is lower semi-continuous, and given that $\Omega$ and $\Omega'$ are Radon spaces, $\Pi(\rho^0, \rho^1)$ is tight, and therefore a solution to the Monge--Kantorovich transport problem always exists under these conditions \citep[e.g.,][]{ambrosio2008gradient}.  From this formulation, the Wasserstein-$p$ distance may be defined as follows.

\begin{definition}
\label{def:wass}
Let $(\Omega,d)$ be a separable metric Radon space.  Let $p \geq 1$ and $\mathscr{P}_p(\Omega)$ be the collection of all probability measures $\mu$ on $\Omega$ such that $\mu$ has finite $p$th moment for some $\x_0 \in \Omega$; i.e., $\displaystyle \int_{\Omega} d(\x, \x_0)^p \mathrm{d}\mu(\x) < +\infty$.  The {\em Wasserstein-$p$ distance} between probability measures $\rho^0, \rho^1 \in \mathscr{P}_p(\Omega)$ is given by
\begin{align*}
W_p & : \mathscr{P}_p(\Omega) \times \mathscr{P}_p(\Omega) \rightarrow [0, +\infty)\\
W_p(\rho^0, \rho^1) & := \Bigg( \inf_{\pi \in \Pi(\rho^0, \rho^1)} \int_{\Omega \times \Omega} d(\x,\y)^p \mathrm{d}\pi(\x,\y) \Bigg)^{1/p},
\end{align*}
where, as before, $\Pi(\rho^0, \rho^1)$ is the collection of all probability measures on $\Omega \times \Omega$ with marginal measures $\rho^0$ and $\rho^1$ on the respective copies of $\Omega$.  Equivalently, we have
$$
W_p(\rho^0, \rho^1)^p = \inf \Big\{ \mathbb{E}\big[d(X, Y)^p \big] \Big\},
$$
where $\mathbb{E}[\cdot]$ denotes the expectation, and the infimum is taken over all joint distributions of random variables $X$ and $Y$ with respective marginals $\rho^0$ and $\rho^1$.  The metric $d$ is referred to as the {\em ground metric}; the function $\pi$ is known as the {\em transport plan}.
\end{definition}

The transport plan $\pi(\x,\y)$ is a function that describes a way to move the measure $\rho^0$ into $\rho^1$, and between locations $\x$ and $\y$; transport plans are not unique.  Since the total mass moved out of a region around $x$ must be equal to $\rho^0(\x)\mathrm{d}\x$ and the total mass moved into a region around $\x$ must be $\rho^1(\x)\mathrm{d}\x$, we have the following restrictions on a transport plan:
\begin{align*}
\int_{\mathbb{R}^n} \pi(\x,\x')d\x' & = \rho^0(\x);\\
\int_{\mathbb{R}^n} \pi(\x,\x')d\x & = \rho^1(\x').
\end{align*}
In other words, $\pi$ is a joint probability distribution with marginals $\rho^0$ and $\rho^1$.  The total infinitesimal mass which moves from $\x$ to $\y$, therefore, is $\pi(\x,\y) \mathrm{d}\x \mathrm{d}\y$ and the cost of moving this amount of mass from $\x$ to $\y$ is $c(\x,\y)\pi(\x,\y)\mathrm{d}\x\mathrm{d}\y$.  The total cost is then
$$
C = \iint c(\x,\y)\pi(\x,\y)d\x d\y = \int c(\x,\y)\mathrm{d}\pi(\x,\y).
$$
The {\em optimal transport plan} is the $\pi$ which achieves the minimal value of $C$:
$$
C^* = \inf_{\pi \in \Pi(\rho^0, \rho^1)} \int c(\x,\y)\mathrm{d}\pi(\x,\y).
$$
If the cost of a move $c(\x,\y)$ is no more than the distance between the two points $d(\x,\y)$, then the optimal cost value $C^*$ is identically the Wasserstein-1 distance, $W_1$. 

\begin{remark}
In the particular case where $p=1$, the Wasserstein-1 distance is also referred to as the {\em Kantorovich--Rubinstein distance}, and the {\em earth mover's distance} (EMD) in the computer science literature.
\end{remark}

\begin{remark}
The Wasserstein distances satisfy all conditions for a formal definition of a metric \citep[e.g.,][]{villani2008optimal}.  If the condition of finite $p$th moment is relaxed, the Wasserstein distances may technically be infinite, and therefore not a metric in the strict sense.
\end{remark}

\begin{remark}
\label{remark:convex}
For any $p \geq 1$, if $(\Omega, d)$ is a complete and separable metric space, then so too is $(\mathscr{P}_p(\Omega), W_p)$ \citep[e.g.,][]{villani2008optimal}.  Other geometric properties between the ground space and its associated Wasserstein distance also hold, including compactness, convexity, as well as non-convexity.  An adaptation of the Brunn--Minkowski theorem \citep{brunn1887ueber, minkowski1896geometrie} relating volumes of compact and convex sets, as well as its generalization to non-convex sets by \cite{lyusternik1935brunn}, for comparative relations between ground and Wasserstein spaces also exists \citep{villani2008optimal}.  The geometric implication of these results is that compact, non-convex subsets of the ground space with respect to the ground metric correspond to non-convex subsets in the Wasserstein space of probability measures (with generalized Ricci curvature bounds) over the ground space with respect to the Wasserstein distance.

In the applicative setting of our work concerning the space of phylogenetic trees as a non-convex subset of the tropical projective torus, the implication is that the corresponding space of probability measures associated with the space of phylogenetic trees is also non-convex with respect to the Wasserstein distances.  (Compactness of tree space can be established by fixing an upper bound on the height of trees.)  This provides a geometric compatibility between the space of phylogenetic trees equipped with the tropical metric and its associated space of probability measures equipped with Wasserstein distances.  See Figure \ref{fig:Wass_fig} for an illustrative description of this relationship.
\end{remark}

\paragraph{A Time-Dependent Cost Function: Formulating a Hamiltonian.}

In formulating the above variational forms of the tropical metric (\ref{dist}) and (\ref{dist2}), the notation with respect to $t$ is not by coincidence and purposely alludes to a dependence upon time.  Within the setting of Wasserstein distances and their relation to the optimal transport problem where the ground metric is itself the cost function, intuitively, a time-dependent ground metric corresponds to a cost function where time is a cost factor.  

Considering time dependence allows for a rich and alternate formulation of the optimal transport problem, which extends to the continuous displacement of measures---precisely the setting of the tropical metric on the tropical projective torus as a continuous metric measure space.  However, there are certain instances where continuous displacement problems turn out to be equivalent to steady-state, time-independent problems with an alternate formulation that favors computational efficiency: this occurs when the Lagrangian $L$ is homogeneous of degree 1 and convex.

\begin{lemma}
\label{lem:convex}
The tropical Lagrangian $L_{\mathrm{tr}}$ defined in (\ref{L}) is convex on $\R^n$.  More specifically, for $\mathbf{a},\mathbf{b}\in \R^{n}$ and $0\le w\le 1$, we have
	\begin{equation}
		(1-w)\|\mathbf{a}\|_{\mathrm{tr}} + w\|\mathbf{b}\|_{\mathrm{tr}} \ge \|(1-w)\mathbf{a}+w\mathbf{b}\|_{\mathrm{tr}}.
	\end{equation}
\end{lemma}

\begin{proof}
	By definition, 

%	\begin{equation*}
%	\|(1-w)\mathbf{a}+w\mathbf{b}\|_{\mathrm{tr}} = \max\Big(\max_{1\le i\le n}{\big((1-w)a_{i}+wb_{i}\big)},0\Big) - \min\Big(\min_{1\le i\le n}{\big((1-w)a_{i}+wb_{i}\big)},0\Big).
%	\end{equation*}

	\begin{align*}
		\|(1-w)\mathbf{a}+w\mathbf{b}\|_{\mathrm{tr}} & = \max\Big(\max_{1\le i\le n}{\big((1-w)a_{i}+wb_{i}\big)},0\Big)\\
		& - \min\Big(\min_{1\le i\le n}{\big((1-w)a_{i}+wb_{i}\big)},0\Big)\\
		& = {\max_{1\le i\le n}{\big((1-w)a_{i}+wb_{i}\big)}}^+ - {\min_{1\le i\le n}{\big((1-w)a_{i}+wb_{i}\big)}}^-.
	\end{align*}

So either there exist $1\le j,k\le n$ such that 
\[\|(1-w)\mathbf{a}+w\mathbf{b}\|_{\mathrm{tr}} = \big((1-w)a_{j}+wb_{j}\big) - \big((1-w)a_{k}+wb_{k}\big),\] 
or there exists $1\le j\le n$ such that 
\[\|(1-w)\mathbf{a}+w\mathbf{b}\|_{\mathrm{tr}} = (1-w)a_{j}+wb_{j}.\]
Note that
\begin{align*}
	\big((1-w)a_{j}+wb_{j}\big) & - \big((1-w)a_{k}+wb_{k}\big) = (1-w)\left(a_{j} - a_{k}\right) + w\left(b_{j} - b_{k}\right) \\
	&\le (1-w)\Big({\max_{1\le i\le n}{(a_{i})}}^+ - {\min_{1\le i\le n}{(a_{i})}}^-\Big) + w\Big({\max_{1\le i\le n}{(b_{i})}}^+ - {\min_{1\le i\le n}{(b_{i})}}^-\Big)\\
	&=(1-w)\| \mathbf{a} \|_{\mathrm{tr}} + w\|\mathbf{b}\|_{\mathrm{tr}}.
\end{align*}
		
%	\begin{align*}
%		\big((1-w)a_{j}+wb_{j}\big) - & \big((1-w)a_{k}+wb_{k}\big) = (1-w)\left(a_{j} - a_{k}\right) + w\left(b_{j} - b_{k}\right) \\
%		\le& \begin{aligned}[t] & (1-w)\left(\max\Big(\max_{1\le i\le n}{(a_{i})},0\Big) - \min\Big(\min_{1\le i\le n}{(a_{i})},0\Big)\right)\\ & + w\left(\max\Big(\max_{1\le i\le n}{(b_{i})},0\Big) - \min\Big(\min_{1\le i\le n}{(b_{i})},0\Big)\right) \end{aligned} \\
%		=&(1-w)\| \mathbf{a} \|_{\mathrm{tr}} + w\|\mathbf{b}\|_{\mathrm{tr}}.
%	\end{align*}

	We also have
%	\begin{align*}
%		(1-w)a_{j}+wb_{j}
%		\le & (1-w)|a_{j}| + w|b_{j}| \\
%		\le & \begin{aligned}[t] & (1-w)\left(\max\Big(\max_{1\le i\le n}{(a_{i})},0\Big) - \min\Big(\min_{1\le i\le n}{(a_{i})},0\Big)\right)\\ & + w\left(\max\Big(\max_{1\le i\le n}{(b_{i})},0\Big) - \min\Big(\min_{1\le i\le n}{(b_{i})},0\Big)\right) \end{aligned} \\
%		= &(1-w)\|\mathbf{a}\|_{\mathrm{tr}} + w\|\mathbf{b}\|_{\mathrm{tr}}.
%	\end{align*}
	
		\begin{align*}
		(1-w)a_{j} & +wb_{j} \le (1-w)|a_{j}| + w|b_{j}| \\
		& \le (1-w)\Big({\max_{1\le i\le n}{(a_{i})}}^+ - {\min_{1\le i\le n}{(a_{i})}}^-\Big) + w\Big({\max_{1\le i\le n}{(b_{i})}}^+ - {\min_{1\le i\le n}{(b_{i})}}^-\Big) \\
		& =(1-w)\|\mathbf{a}\|_{\mathrm{tr}} + w\|\mathbf{b}\|_{\mathrm{tr}}.
	\end{align*}

	Hence Lemma \ref{lem:convex} holds in either case. 
\end{proof}

\begin{remark}
Note that convexity of $L_{\mathrm{tr}}$ also implies convexity of $\frac{1}{p}L_{\mathrm{tr}}^{p}$.
\end{remark}

The convexity of the tropical Lagrangian $L_{\mathrm{tr}}$ then allows for the formulation of the {\em Hamiltonian} \citep[Example 7.5]{villani2008optimal} for $\mathbf{b} \in \R^{n}$ as follows:
\begin{equation}\label{eq:hp}
\begin{split}
H(\mathbf{b})=&\sup_{\mathbf{a}\in \R^{n}} \left\{\mathbf{a}^{\T} \mathbf{b}- \frac{1}{p} \| \mathbf{a} \|_{\mathrm{tr}}^p \right\}\\
=&\sup_{\mathbf{a}\in \R^{n}} \left\{ \sum_{i=1}^{n}{b_{i} a_{i}}-\frac{1}{p}\Big({\max_{1\le i\le n}(a_{i})}^+ - {\min_{1\le i\le n}(a_{i})}^-\Big)^p\right\}.
\end{split}
\end{equation}

%\begin{equation}\label{eq:hp}
%\begin{split}
%H(\mathbf{b})=&\sup_{\mathbf{a}\in \R^{n}} \left\{\mathbf{a}^{\T} \mathbf{b}- \frac{1}{p} \| \mathbf{a} \|_{\mathrm{tr}}^p \right\}\\
%=&\sup_{\mathbf{a}\in \R^{n}} \left\{ \sum_{i=1}^{n}{b_{i} a_{i}}-\frac{1}{p}\left(\max\Big(\max_{1\le i\le n}(a_{i}),0\Big) - \min\Big(\min_{1\le i\le n}(a_{i}),0\Big)\right)^p\right\}.
%\end{split}
%\end{equation}

We now explicitly compute the value of the Hamiltonian (\ref{eq:hp}), which will provide concise formulations with regard to the tropical Wasserstein-$p$ distances.  For convenience, and identifying $\R^{n+1}/\R\one$ with $\R^n$, for $\mathbf{b} \in \R^{n}$ we define 
\begin{equation}\label{ms}
\zeta(\mathbf{b}) := \max_{S\subset \{1,2,\ldots,n\}}{\left|\sum_{i\in S}{b_{i}}\right|}.
\end{equation}
In other words, $\zeta(\mathbf{b})$ is the absolute value of the sum of either all positive $b_{i}$ or all negative $b_{i}$. In particular, $\zeta(\mathbf{b})=0$ if and only if $\mathbf{b}=\mathbf{0}$.

\begin{example}\label{remark13}
    When $n=2$, we have $\mathbf{b} = (b_1, b_2)$ and
    \[
    \zeta(\mathbf{b}) = \begin{cases}
    b_{1}+b_{2}, & \text{ if } b_{1}\ge 0, b_{2}\ge 0; \\
    -b_{1}-b_{2}, & \text{ if } b_{1}\le 0, b_{2}\le 0; \\
    b_{1}, & \text{ if } b_{1}\ge -b_{2}\ge 0; \\
    b_{2}, & \text{ if } b_{2}\ge -b_{1}\ge 0; \\
    -b_{1}, & \text{ if } -b_{1}\ge b_{2}\ge 0; \\
    -b_{2}, & \text{ if } -b_{2}\ge b_{1}\ge 0.
    \end{cases}
    \]
\end{example}

\begin{proposition}\label{prop:hp}
	The value of $H(\mathbf{b})$ is:
	\begin{enumerate}[(i)]
		\item $0$ when $\mathbf{b}=\mathbf{0}$, or $\zeta(\mathbf{b})\le 1$ and $p=1$;
		\item $\infty$ when $\mathbf{b}\ne \mathbf{0}$ and $p<1$, or $\zeta(\mathbf{b})>1$ and $p=1$;
		\item $\displaystyle \frac{p-1}{p}\zeta(\mathbf{b})^{\frac{p}{p-1}}$ when $\mathbf{b}\ne \mathbf{0}$ and $p>1$.
	\end{enumerate}
\end{proposition}

\begin{proof}

\begin{enumerate}[(i)]

\item When $\mathbf{b}=\mathbf{0}$, $\sum_{i=1}^{n}{b_{i} a_{i}}$ is always zero, and $L_{\mathrm{tr}}(\mathbf{a})\ge 0$, so $H(\mathbf{b})\le 0$.  However, when $\mathbf{a}=\mathbf{0}$, the right-hand side of (\ref{eq:hp}) is zero, so $H(\mathbf{0})=0$. When $\zeta(\mathbf{b})\le 1$ and $p=1$, let
	\[u := \max_{1\le i\le n}(a_{i},0) \ge 0 \mbox{~~~and~~~} v := \min_{1\le i\le n}(a_{i},0) \le 0.\]	
Then we have
	\begin{align*}
		\sum_{i=1}^{n}{b_{i} a_{i}} &= \sum_{b_{i}>0}{b_{i} a_{i}} + \sum_{b_{i}<0}{b_{i} a_{i}} \\
		&\le \sum_{b_{i}>0}{b_{i} u} + \sum_{b_{i}<0}{b_{i} v} \\
		&\le \zeta(\mathbf{b})u + \zeta(\mathbf{b})(-v) \\
		&= \zeta(\mathbf{b})(u-v)\le u-v.
	\end{align*}
Hence $H(\mathbf{b})\le 0$, and equality holds when $\mathbf{a}=\mathbf{0}$. So $H(\mathbf{b})=0$.\\

\item	Now we may assume that $\mathbf{b}\ne \mathbf{0}$ and thus $\zeta(\mathbf{b})>0$.  We may choose nonempty $S\subset \{1,2,\ldots,n\}$ such that
	\[\zeta(\mathbf{b}) = \left| \sum_{j\in S}{b_{j}} \right|. \]
	For any $N>0$ and each $1\le i\le n$, we let
	
	\[ a_{i} = \begin{cases}
	\displaystyle \frac{b_{i}}{|b_{i}|}\cdot N, &\text{ if } i\in S; \\
	0, &\text{ if } i\notin S.
	\end{cases} \] 
	Then $\sum_{i=1}^{n}{b_{i} a_{i}} = \zeta(\mathbf{b})\cdot N$ and the set $\{a_{i}\mid 1\le i\le n\}\cup \{0\}$ is either $\{0, N\}$ or $\{0, -N\}$, so 
	$L_{\mathrm{tr}}(\mathbf{a})$ is $N - 0$ or $0 - (-N)$, which is $N$.  Since $\zeta(\mathbf{b})>0$, when $p<1$, or $\zeta(\mathbf{b})>1$ and $p=1$, we have
	\[\lim\limits_{N\to \infty}{\left(\zeta(\mathbf{b})N - \frac{1}{p}N^{p}\right)} = \infty. \]
	So $H(\mathbf{b})=\infty$. \\

\item We denote $u,v$ as in (i) above. Then

	\[H(\mathbf{b}) \le \zeta(\mathbf{b})(u-v) - \frac{1}{p}(u-v)^{p}. \]
	
	Let $s:=u-v\ge 0$.  We need to find the maximum of $\zeta(\mathbf{b})s-\frac{1}{p}s^{p}$ when $s\ge 0$. The derivative of this function of $s$ is
	
	\[\zeta(\mathbf{b}) - s^{p-1}. \]
	
	Hence the function is increasing when $0\le s\le \zeta(\mathbf{b})^{\frac{1}{p-1}}$, and it is decreasing when $s\ge \zeta(\mathbf{b})^{\frac{1}{p-1}}$. So the maximum is attained when $s=\zeta(\mathbf{b})^{\frac{1}{p-1}}$, thus
	
	\[H(\mathbf{b})\le \zeta(\mathbf{b})\cdot \zeta(\mathbf{b})^{\frac{1}{p-1}} - \frac{1}{p}\zeta(\mathbf{b})^{\frac{p}{p-1}} = \frac{p-1}{p}\zeta(\mathbf{b})^{\frac{p}{p-1}}. \]
	
	Finally, as in (ii), we may choose nonempty $S\subset \{1,2,\ldots,n\}$ such that
		\[\zeta(\mathbf{b}) = \left| \sum_{j\in S}{b_{j}} \right|, \]
	and the equality holds when
	\begin{equation}\label{Hab}
	a_{i} = 
	\begin{cases}
	\displaystyle \frac{b_{i}}{|b_{i}|}\cdot \zeta(\mathbf{b})^{\frac{1}{p-1}}, &\text{ if } i\in S; \\
	0, &\text{ if } i\notin S.
	\end{cases}
	\end{equation}
	
\end{enumerate}
\end{proof}

For notational convenience, %when considering the expression \eqref{Hab}, 
we also define $\eta \colon \mathbb{R}^n\rightarrow \mathbb{R}^n$, where $\eta(\mathbf{b})=(\eta(\mathbf{b})_i)_{i=1}^n$, with $$\eta(\mathbf{b})_i:=a_i=	\begin{cases}
	\displaystyle \frac{b_{i}}{|b_{i}|}\cdot \zeta(\mathbf{b})^{\frac{1}{p-1}}, &\text{ if } i\in S; \\
	0, &\text{ if } i\notin S.
	\end{cases}$$ 
That is, $\eta(\mathbf{b})_i$ is defined by (\ref{Hab}).

\paragraph{The Tropical Wasserstein-$p$ Distances.}
We consider the tropical projective torus as a probability space \citep{lin2018tropical} with finite $p$th moment as follows:
$$
\mathscr{P}_p(\R^{n}) = \Big\{\rho\in L^1(\mathbb{R}^{n})~\colon \int_{\mathbb{R}^{n}}\rho(\x)^p d\x=1,~\rho\geq 0\Big\}.
$$
Within the optimal transport framework discussed above and as in Definition \ref{def:wass}, the {tropical Wasserstein-$p$ distance} is given as follows:
	\begin{align}\label{lp}
	\tilde{W}^{\mathrm{tr}}_p & : \mathscr{P}_p(\mathbb{R}^n) \times \mathscr{P}_p(\mathbb{R}^n) \rightarrow [0, +\infty) \nonumber\\
	\tilde{W}^{\mathrm{tr}}_{p}(\rho^0, \rho^1)^{p} & := \inf_{\pi\in \Pi(\rho^0, \rho^1)} \int_{\mathbb{R}^n\, \times \, \mathbb{R}^{n}} d_{\mathrm{tr}}(\x,\y)^p\mathrm{d}\pi(\x,\y),
	\end{align}
	where the infimum is taken over the set of all possible joint distributions (transport plans) $\pi$ with marginals $\rho^0$ and $\rho^1$, $\Pi(\rho^0, \rho^1)$.  Here, the distance $\tilde{W}^{\mathrm{tr}}_p$ depends the choice of $p$ in the linear programming formulation \eqref{lp}.  The following alternative gives an equivalent definition of the {tropical Wasserstein-$p$ distances}.  

\begin{definition}[Tropical Wasserstein-$p$ distance]
The {\em tropical Wasserstein-$p$ distance} is given by 
\begin{subequations}\label{UOTT}
\begin{equation}\label{UOT}
W^{\mathrm{tr}}_p(\rho^0,\rho^1)^p=\inf_{\mathbf{v},\rho}\int_{0}^1\int_{\mathbb{R}^n} \big\| \mathbf{v}(t,\x) \big\|_{\mathrm{tr}}^p\,\rho(t,\x)d\x dt
\end{equation} 
such that the following dynamical constraint or {\em continuity equations} hold:
\begin{equation}\label{UC}
\begin{split}
 \partial_t\rho(t,\x)+\nabla\cdot\big(\rho(t,\x)\mathbf{v}(t,\x) \big)& =0,\\
 \rho(0,\x) & =\rho^0(\x),\\
 \rho(1,\x) & =\rho^1(\x).
\end{split}
\end{equation}
\end{subequations}
Here $\|\cdot\|_{\mathrm{tr}}$ is the tropical norm, $\rho^0$, $\rho^1\in \mathscr{P}_p(\mathbb{R}^n)$, $\nabla$, $\nabla\cdot$ are gradient and divergence operators in $\mathbb{R}^n$, and the infimum is taken over all continuous density functions $\rho\colon [0,1]\times \mathbb{R}^n\rightarrow \mathbb{R}$, and Borel vector fields $\mathbf{v}\colon [0,1]\times\mathbb{R}^n \rightarrow \mathbb{R}^n$. 
\end{definition}
Here, the formulation \eqref{UOTT} given by the pairs \eqref{UOT} and \eqref{UC} is known as the {\em Benamou--Brenier formula}, given by \cite{BB}.
As discussed in Chapter 8 of \cite{villani2003topics}, when $c$ satisfies suitable conditions, the linear programming formulation $\tilde{W}_p^{\textrm{tr}}$ is equivalent to the dynamical formulation $W_p^{\textrm{tr}}$.  In this work, we focus on the dynamical formulation \eqref{UOTT} with $p=1,2$ for their concrete implications on computations of the tropical projective torus. 

\subsection{The Tropical Wasserstein-1 Distance}

We first study the case $p=1$.  In this case, it turns out that the tropical Wasserstein-1 distance $W_1^{\mathrm{tr}}$ may be recast as the following minimization problem.

%	\begin{equation}\label{eq:EMD}
%	W^{\mathrm{tr}}_1(\rho^0, \rho^1) = \left(
%	        \begin{aligned}
%	            \text{\rm minimize } \quad & \int_{\mathbb{R}^n} { L_{\mathrm{tr}}\big( \mathbf{m}(\x) \big) \mathrm{d}\x}, \\
%	            \text{\rm subject to } \quad & \nabla \cdot \mathbf{m}(\x) + \rho^1(\x) - \rho^0(\x) = 0 \\
%	            					     & \mathbf{m}(\x) \cdot \mathbf{n}(\x) = 0 \quad \text{\rm for all} \quad
%						     		\begin{cases}
%									\x \in \partial\Omega,\\
%									\mathbf{n}(\x) \text{\rm ~normal to } \partial\Omega
%								\end{cases}
%	        \end{aligned}
%	\right).
%	\end{equation}
%In this setting, the optimization variable is the function $\mathbf{m}: \Omega \rightarrow \R^d$ and is a flux vector; it is constrained to satisfy the {\em zero flux boundary condition}.

\begin{proposition}[Minimal Flux Formulation]
\label{prop:min_flux}
By identifying $\R^{n+1}/\R\one$ with $\R^n$ as discussed in Section \ref{subsec:trop_metric}, the tropical Wasserstein-1 distance satisfies 
\begin{equation}\label{W1}
\begin{split}
W^{\mathrm{tr}}_1(\rho^0,\rho^1)=\inf_{\mathbf{m}} \bigg\{\int_{\mathbb{R}^n} \big\|\mathbf{m}(\x) \big\|_{\mathrm{tr}} d\x \, \colon \rho^1(\x)-\rho^0(\x)+\nabla\cdot \mathbf{m}(\x)=0\bigg\},
\end{split}
\end{equation}
where the infimum is taken over all Borel flux functions $\mathbf{m} \colon \mathbb{R}^n\rightarrow \mathbb{R}^n$. 
\end{proposition}

\begin{proof}
This minimal flux formulation follows the result in optimal transport theory. By Jensen's inequality, the minimizer of \eqref{UOTT} is obtained by a time-independent solution. Denote 
\begin{equation*}
\mathbf{m}(\x) := \int_0^1 \mathbf{v}(t,\x)\rho(t,\x)dt.
\end{equation*}
Then 
\begin{align*}
\int_{0}^1\int_{\mathbb{R}^n} \big\| \mathbf{v}(t,\x) \big\|_{\mathrm{tr}}\,\rho(t,\x)d\x dt
\geq \int_{\mathbb{R}^n} \big\| \mathbf{m}(\x) \big\|_{\mathrm{tr}}d\x
\end{align*}
By choosing $\rho(t,\x)=(1-t)\rho^0(\x)+t\rho^1(\x)$, i.e., $\rho^1(x)-\rho^0(x)+\nabla \cdot m(x)=0$, we derive the minimizer of above minimization problem. 
\end{proof}

Concretely, $\mathbf{m}(\x)$ is the flux vector field that assigns a vector to each point in the measure and determines how much of the mass (measure) should be moved, and in which direction.  %This assignment defines a Lagrangian by minimizing the cost, which depends on the flux vector field.

%%\begin{proposition}[Kantorovich duality problem]
%%\begin{equation*}
%%\begin{split}
%%W_1(\rho^0,\rho^1)=\inf_{m} \Big\{\int_{\mathbb{R}^n}\Phi^1(x)\rho^1(x)-\Phi^0(x)\rho^0(x)dx \colon ms(\nabla\Phi)\leq 1\Big\}.
%%\end{split}
%%\end{equation*}
%%where the supreme is over all continuous function $\Phi^1$, $\Phi^0\colon \mathbb{R}^n\rightarrow \mathbb{R}$, $ms$ is defined in \eqref{ms}. 
%%\end{proposition}

The reformulation of the tropical Wasserstein-1 distance given in Proposition \ref{prop:min_flux}~has enormous computational benefits, compared to that given in Definition \ref{def:wass} \citep{Li2018}.  Notably, the size of the optimization variable is much smaller in solving a discrete approximation; additionally, the structure of the formulation given in Proposition \ref{prop:min_flux} borrows from $L_1$-type minimization problems, which are well-studied and for which there exist fast and simple algorithms (see references in \cite{Li2018}).  We will reap these benefits in formulating explicit algorithms to compute the tropical Wasserstein-$p$ distances for $p=1,2$, as discussed further on in Section \ref{sec:computation}.

\paragraph{Geodesics on the Tropical Projective Torus.}

Geodesics on the tropical projective torus are not unique \citep{doi:10.1137/16M1079841, lin2018tropical}.  In particular, between any two given points in $\R^{n+1}/\R\one$, there are infinitely many geodesics.  The following result gives the explicit connection between geodesics on the tropical projective torus and the minimizer of the tropical Wasserstein-1 distance.

\begin{proposition}[Minimizer of the Tropical Wasserstein-1 distance]
\label{prop:min1}
The minimizer of the tropical Wasserstein-1 distance is given by the following pair:
\begin{equation}\label{pair}
\left\{
\begin{aligned}
&\nabla_{\mathbf{m}} \big\| \mathbf{m}(\x) \big\|_{\mathrm{tr}}=\nabla \Phi(\x) \quad \textrm{if $\mathbf{m}(\x)>0$,}\\
&\rho^1(\x)-\rho^0(\x)+\nabla\cdot \mathbf{m}(\x)=0.
\end{aligned}\right.
\end{equation}
\end{proposition}
\begin{proof}
The minimizer of tropical Wasserstein-1 distance may be derived as follows.  Define a Lagrange multiplier $\Phi\colon \R^n \rightarrow \mathbb{R}$ for the equality constraint of \eqref{W1}, and consider the saddle point problem 
\begin{equation*}
L(\mathbf{m},\Phi)=\int_{\mathbb{R}^n}\|\mathbf{m}(\x)\|_{\textrm{tr}} d\x+\int_{\mathbb{R}^n} \Phi(\x)\big(\nabla\cdot\mathbf{m}(\x)+\rho^1(\x)-\rho^0(\x) \big)d\x.
\end{equation*}
Notice that $L$ is convex in $\mathbf{m}$ and concave in $\Phi$.  Thus, the saddle point $(\mathbf{m}, \Phi)$ satisfies $\delta_{\mathbf{m}}L(\mathbf{m},\Phi)=0$, $\delta_\Phi L(\mathbf{m},\Phi)=0$.  This corresponds to the equation pair \eqref{pair}. 
\end{proof}

\begin{remark}
We notice that the first equation in \eqref{pair} represents the {\em tropical Eikonal equation}
$$\zeta\big(\nabla\Phi(\x) \big)=1.$$
\end{remark}
The tropical Eikonal equation describes the movement of each particle according to the infinitely many geodesics under the tropical metric between $\rho^0$ to $\rho^1$.  This behavior will be explored and demonstrated numerically in experiments further on in Section \ref{sec:num_exp}. 

\begin{proposition}
The set of all infinitely many tropical geodesics is contained in a classical convex polytope.
\end{proposition}

\begin{proof}
For any point $\bar{c}$ on a tropical geodesic connecting $\bar{a}, \bar{b} \in \R^{n+1}/\R\one$, by the definition of geodesics, we have
\begin{equation*}
    \dtr(\bar{c},\bar{a})+\dtr(\bar{c},\bar{b}) = \dtr(\bar{a},\bar{b}).
\end{equation*}
So $\bar{c}$ belongs to a tropical ellipse with foci $\bar{a},\bar{b}$. By Proposition 26 of \cite{LinYoshida}, the set of all points on tropical geodesics is a classical convex polytope.
\end{proof}

\subsection{The Tropical Wasserstein-2 Distance}

We now consider the case where $p=2$. Here we refer to \eqref{Hab} using the notation $\eta(\mathbf{b})$.

\begin{proposition}[Minimizer of the Tropical Wasserstein-2 Distance]
\label{prop:min2}
The minimizer of the tropical Wasserstein-2 distance $(\mathbf{v}(t,\x), \rho(t,\x))$ satisfies 
$$
\mathbf{v}(t,\x)=\eta\big(\nabla\Phi(t,\x) \big),
$$ 
where $\eta\colon \mathbb{R}^n\rightarrow \mathbb{R}^n$ is given by 
\begin{equation*}
	\eta\big(\nabla\Phi(t,\x)\big)_i = 
	\begin{cases}
	\displaystyle \frac{\nabla_{x_i}\Phi(t,\x)}{|\nabla_{x_i}\Phi(t,\x)|}\cdot \zeta\big(\nabla\Phi(t,\x) \big) &\text{ for } i\in S; \\
	0 &\text{ for } i\notin S,
	\end{cases}
	\end{equation*}
where $S$ is as in (\ref{ms}).  Also,  
\begin{equation}\label{minimizer1}
\left \{
\begin{aligned}
&\partial_t\rho(t,\x)+\nabla\cdot \big( \rho(t,\x) \eta \big(\nabla \Phi(t,\x) \big) \big)=0,\\
&\partial_t\Phi(t,\x)+\frac{1}{2}{\zeta}\big(\nabla\Phi(t,\x) \big)^2\leq 0,\\
&\rho(0,\x)=\rho^0(\x),\quad \rho(1,\x)=\rho^1(\x).
\end{aligned}
\right.
\end{equation}
In particular, if $\rho(t,\x)>0$, then 
\begin{equation*}
\partial_t\Phi(t,\x)+\frac{1}{2}{\zeta}\big(\nabla\Phi(t,\x) \big)^2=0. 
\end{equation*}
\end{proposition}

\begin{proof}
The minimizer path for the tropical Wasserstein-2 distance is derived as follows. For $p=2$, denote $\mathbf{m}(t,\x):=\rho(t,\x) v(t,\x)$ where 
\begin{equation*}
F(\mathbf{m},\rho)=\begin{cases}
\displaystyle \frac{\|\mathbf{m}\|_{\textrm{tr}}^2}{2\rho}& \textrm{if $\rho>0$;}\\
0 & \textrm{if $\rho=0$, $\mathbf{m}=0$;}\\
+\infty& \textrm{otherwise.}
\end{cases}
\end{equation*}
Then the variational problem \eqref{UOTT} can be reformulated as 
\begin{equation}\label{UOT1}
\begin{split}
\frac{1}{2}W_2^{\textrm{tr}}(\rho_0,\rho_1)^2=\inf_{\mathbf{m},\rho}\Big\{&\int_{0}^1\int_{\mathbb{R}^n} F\big(\mathbf{m}(t,\x), \rho(t,\x) \big)d\x dt\colon\\
%&\qquad\qquad\partial_t\rho(t,\x)+\nabla\cdot\big( \mathbf{m}(t,\x) \big)= 0,\\
%&\qquad\qquad\rho(0,\x)=\rho_0(\x),~\rho(1,\x)=\rho_1(\x)\Big\}.
&\partial_t\rho(t,\x)+\nabla\cdot\big( \mathbf{m}(t,\x) \big)= 0,\\
&\rho(0,\x)=\rho_0(\x),~\rho(1,\x)=\rho_1(\x)\Big\}.
\end{split}
\end{equation}
Notice that variational problem \eqref{UOT1} is convex in $(\mathbf{m},\mu)$. Again, we denote the Lagrange multiplier $\Phi\colon [0,1]\times \R^n \rightarrow \mathbb{R}$, then we can reformulate \eqref{UOT1} into a 
saddle point problem.
\begin{equation*}
L(\mathbf{m},\rho, \Phi)=\int F(\mathbf{m},\rho)+ \Phi(t,\x) \Big(\partial_t\rho(t,\x)+\nabla\cdot \mathbf{m}(t,\x)\Big) d\x.
\end{equation*}
Thus the saddle point $(\mathbf{m},\rho, \Phi)$ satisfies the system $\delta_{\mathbf{m}} L =0$, $\delta_\rho L \geq 0$, $\delta_\Phi L =0$, i.e.,
\begin{equation*}
\begin{cases}
& \displaystyle \frac{\nabla_{\mathbf{m}}\|\mathbf{m}\|_{\mathrm{tr}}^2}{\rho}=\nabla\Phi\\
& \displaystyle -\frac{\|\mathbf{m}\|_{\mathrm{tr}}^2}{2\rho}-\partial_t\Phi\geq 0. 
\end{cases}
\end{equation*}
Following Proposition \ref{prop:hp}, we obtain the minimizer of the system \eqref{minimizer1}. 
%{\color{red} Similar as Wasserstein-1, we need to check $\nabla_m\|m\|_{tr}^2$, will give operation $ms$ etc. }
\end{proof}

%In what follows, we demonstrate that \eqref{minimizer1} is a geodesic in probability density space endowed with the tropical Wasserstein-2 metric.  We will demonstrate these geodesics in density space numerically.

%%%%%%%%%%%%%%%%%%%%%%%%%%%%%%%%%%%%%%%%%%%%%%%%%%%

\section{Algorithms: Solving the Optimal Transport Problem}
\label{sec:computation}

In this section, we design algorithms for solving the optimal transport problems that give rise to the tropical Wasserstein-$p$ distances and geodesics.  Our approach is mainly based on the {\em G-Prox primal-dual hybrid gradient} (G-Prox PDHG) algorithm \citep{JacobsLegerLiOsher2018_solvinga}, which is a modified version of Chambolle--Pock primal-dual algorithms \citep{pock1, pock2}. 

We now provide a brief overview of the algorithm; see \cite{JacobsLegerLiOsher2018_solvinga, pock1, pock2} for further details.  The classical primal-dual hybrid gradient algorithms convert the following minimization problem
$$ \min_{X} f(KX) + g(X) $$
into the following saddle point problem
$$ \min_X \max_Y \Big\{L(X,Y)=\left< KX, Y\right> + g(X) - f^*(Y)\Big\},$$
where $f$ and $g$ are convex functions with respect to a variable $X$, $f^*$ is a convex dual function of $F$, and $K$ is a continuous linear operator. For each iteration, the algorithm performs gradient descent on the primal variable $X$ and gradient ascent on the dual variable $Y$ as follows: 
\begin{equation}\label{iteration_PD}
\begin{cases}
X^{k+1}=&\arg\min_{X} L(X,Y^k)+\frac{1}{2h}\|X-X^k\|^2 ;\\
Y^{k+1}=&\arg\max_{Y} L(2X^{k+1}-X^k,Y)-\frac{1}{2\tau}\|Y-Y^k\|^2,
\end{cases}
\end{equation}
where suitable norms need to be considered in the update.  

For the tropical Wasserstein-1 and Wasserstein-2 distances, we apply the algorithm in \eqref{iteration_PD} to (\ref{W1}) and (\ref{UOT1}) by setting $Y = \Phi$ and specifying
\begin{align*}
W^{\mathrm{tr}}_1: \qquad &
\begin{aligned}[t]
X & = \mathbf{m},\\
KX & = \nabla \cdot \mathbf{m},\\
g(X) & = \|\mathbf{m}\|_{\mathrm{tr}},\\
f(X) & = \begin{cases} 
    0 & \text{ if } X + \rho^1 - \rho^0 = 0,\\
    \infty & \text{ otherwise};
    \end{cases}
\end{aligned}
\\
W^{\mathrm{tr}}_2: \qquad & 
\begin{aligned}[t]
X & = (\mathbf{m},\rho),\\
KX & = \partial_t \rho + \nabla \cdot \mathbf{m},\\
g(X) & = F(\mathbf{m},\rho),\\
f(X) & = \begin{cases}
    0 & \text{ if } X = 0,\\
    \infty & \text{ otherwise.}
    \end{cases}
\end{aligned}
\end{align*}

%\begin{equation*}
%    \begin{tabular}{c l l l l}
%    $W^{\text{tr}}_1:$, &
%    $X = \mathbf{m}$, &
%    $KX = \nabla \cdot \mathbf{m}$, 
%    &
%        $g(X) = \|\mathbf{m}\|_{\text{tr}}$,
%        &
%    $f(X) = \begin{cases}
%    0 & \text{if } X + \rho^1 - \rho^0 = 0,\\
%    \infty & \text{otherwise};
%    \end{cases}$ \\
%    $W^{\text{tr}}_2:$, &
%    $X = (\mathbf{m},\rho)$, &
%    $KX = \partial_t \rho + \nabla \cdot \mathbf{m}$, 
%    &
%        $g(X) = F(\mathbf{m},\rho)$,&
%    $f(X) = \begin{cases}
%    0 & \text{if } X = 0,\\
%    \infty & \text{otherwise.}
%    \end{cases}$ 
%    \end{tabular}    
%\end{equation*}
% For Tropical Wasserstein-2 distance, 
% %the algorithm is applied in \eqref{iteration_PD} 
% we apply to (\ref{UOT1}) by having 
% \begin{equation*}
%     \begin{tabular}{c c c c}
%     $X = (m,\rho)$ &
%     $KX = \partial_t \rho + \nabla \cdot m$ &
%     $f(X) = \begin{cases}
%     0 & \text{if } X = 0\\
%     \infty & \text{otherwise}
%     \end{cases}$ &
%         $g(X) = F(m,\rho)$.
%     \end{tabular}    
% \end{equation*}
In this paper, we use a version of the G-Prox PDHG algorithm that applies the $H^1$ norm in the dual variable $Y$ update and uses the $L^2$ norm in the primal variable $X$ update. This choice of norms gives us more stable and faster convergence of the algorithm than the standard PDHG algorithm \citep{pock1}.

\subsection{Computing the Tropical Wasserstein-$1$ Distances}

%Expressing the tropical metric in variational forms is extremely beneficial for computations.
We consider here $p=1$. We first present the spatial discretization to compute the general Wasserstein-1 distance. 

Consider a uniform lattice graph $G=(V, E)$ with spacing $\Delta \x$ to discretize the spatial domain, where 
$V$ is the vertex set $V=\{1,2,\ldots, N\},$ and $E$ is the edge set. 
Here $\mathbf{i}=(i_1, \ldots, i_d)\in V$ represents a point in $\mathbb{R}^d$. Consider a discrete probability set supported on all vertices:
\begin{equation*}
\mathcal{P}(G)=\left\{ (q_\mathbf{i})_{\mathbf{i}=1}^N\in \mathbb{R}^{N} \ \Big| \ \sum_{\mathbf{i}=1}^N q_\mathbf{i}=1,~q_\mathbf{i}\geq 0,~\mathbf{i} \in V \right\},
\end{equation*}
where $q_\mathbf{i}$ here represents a probability at node $i$, i.e., $q_\mathbf{i}=\int_{C_\mathbf{i}} \rho(\x)d\x$, and $C_\mathbf{i}$ is a cube centered at $\mathbf{i}$ with length $\Delta \x$. Thus, $\rho^0(\x)$, $\rho^1(\x)$ is approximated by $q^0=(q^0_\mathbf{i} )_{\mathbf{i}=1}^N$ and $q^1= (q^1_\mathbf{i} )_{\mathbf{i}=1}^N$.

We use two steps to compute the Wasserstein-1 distance on $\mathcal{P}(G)$. 
We first define a flux on a lattice. Denote the flux matrix as $\mathbf{m}=(\mathbf{m}_{\mathbf{i}+\frac{1}{2}})_{\mathbf{i}=1}^N\in \mathbb{R}^{N\times d}$, where each component $\mathbf{m}_{\mathbf{i}+\frac{1}{2}}$ is a row vector in $\mathbb{R}^d$, i.e., 
$$
\mathbf{m}_{\mathbf{i}+\frac{1}{2}}=\Big(\mathbf{m}_{\mathbf{i}+\frac{1}{2}e_v} \Big)_{v=1}^d=\Bigg(\int_{C_{\mathbf{i}+\frac{1}{2}e_v}}m^v(\x)d\x \Bigg)_{v=1}^d,
$$
where $e_v=(0,\ldots, \Delta \x,\ldots, 0)^\T$, with $\Delta \x$ at the $v$th column. In other words, if we denote $\mathbf{i}=(i_1, \ldots, i_d)\in \mathbb{R}^d$ and $\mathbf{m}(\x)=(\mathbf{m}^1(\x), \ldots, \mathbf{m}^d(\x))$, then 
\begin{equation*}
\mathbf{m}_{\mathbf{i}+\frac{1}{2}e_v}\approx \mathbf{m}^v\Big(i_1,\ldots,\, i_{v-1},\, i_v+\frac{1}{2}\Delta \x,\, i_{v+1},\ldots, i_d\Big)\Delta \x^d.
\end{equation*}
We consider a zero flux condition: if a point $\mathbf{i}+\frac{1}{2}e_v$ is outside the domain of interest $\Omega$, we let $\mathbf{m}_{\mathbf{i}+\frac{1}{2}e_v}=0$. Based on such a flux $\mathbf{m}$, we define a discrete divergence operator 
$\textrm{div}_G(\mathbf{m}):=(\textrm{div}_G\big(\mathbf{m}_\mathbf{i}))_{\mathbf{i}=1}^N$, where
\begin{equation*}
\textrm{div}_G(\mathbf{m}_\mathbf{i}):=\frac{1}{\Delta \x}\sum_{v=1}^d (\mathbf{m}_{\mathbf{i}+\frac{1}{2}e_v} - \mathbf{m}_{\mathbf{i}-\frac{1}{2}e_v}).
\end{equation*}
We next introduce the discrete cost functional %on a graph
\begin{equation*}
\|\mathbf{m}\|:=\sum_{\mathbf{i}=1}^N\|\mathbf{m}_{\mathbf{i}+\frac{1}{2}}\|_{2}=\sum_{\mathbf{i}=1}^N \sqrt{\sum_{v=1}^d |\mathbf{m}_{\mathbf{i}+\frac{e_v}{2}}|^2}.
\end{equation*}
This gives rise to the following optimization problem in the tropical setting
\begin{equation}\label{W1new}
\begin{aligned}
& \underset{\mathbf{m}}{\text{minimize}}
& &  \|\mathbf{m}\|_{\mathrm{tr}}=\sum_{\mathbf{i}=1}^N \sqrt{\sum_{v=1}^d \| \mathbf{m}_{\mathbf{i}+\frac{e_v}{2}}\|_{\mathrm{tr}}^2} \\
& \text{subject to}
& & \frac{1}{\Delta \x}\sum_{v=1}^d (\mathbf{m}_{\mathbf{i}+\frac{1}{2}e_v} - \mathbf{m}_{\mathbf{i}-\frac{1}{2}e_v})+q_{\mathbf{i}}^1-q_{\mathbf{i}}^0=0,
\end{aligned}
\end{equation}
for $\mathbf{i} =1,\ldots, N; v =1,\ldots, d.$

We solve \eqref{W1new} by studying its saddle point structure. Denoting the Lagrange multiplier of \eqref{W1new} as $\Phi=(\Phi_\mathbf{i})_{\mathbf{i}=1}^N$, we obtain \begin{equation}\label{saddle1}
\min_{\mathbf{m}}\max_{\Phi} \quad L(\mathbf{m}, \Phi):=\min_{\mathbf{m}}\max_{\Phi} \quad \|\mathbf{m}\|_{\mathrm{tr}}+\Phi^\T(\textrm{div}_G(\mathbf{m})+q^1-q^0).
\end{equation}

Saddle point problems such as \eqref{saddle1} are well studied by the first-order primal-dual hybrid gradient (PDHG) algorithm. %As mentioned previously, we implement the G-Prox PDHG (\cite{JacobsLegerLiOsher2018_solvinga}) to solve the problem. Here $G$-Prox represents the proximal gradient operator, which means that we apply the $H^1$ norm in the update of the dual variable. 
Implementing the G-Prox PDHG algorithm gives the following iteration steps:
\begin{equation}\label{iteration}
\begin{cases}
\mathbf{m}^{k+1}=&\arg\min_{\mathbf{m}} L(\mathbf{m},\Phi^k)+\frac{1}{2h}\| \mathbf{m} - \mathbf{m}^k\|^2_{L^2},\\
\Phi^{k+1}=&\arg\max_{\Phi} L(2\mathbf{m}^{k+1}-\mathbf{m}^k,\Phi)-\frac{1}{2\tau}\|\Phi-\Phi^k\|^2_{H^1},
\end{cases}
\end{equation}
where the quantities $h$, $\tau$ are two small step sizes, and 
\begin{align*}
\|\mathbf{m}-\mathbf{m}^k\|^2_{L^2} & =\sum_{\mathbf{i}=1}^N\sum_{v=1}^d\Big(\mathbf{m}_{\mathbf{i}+\frac{1}{2}e_v}-\mathbf{m}_{\mathbf{i}+\frac{1}{2}e_v}^k \Big)^2 \Delta \x,\\ 
\|\Phi-\Phi^k\|^2_{H^1} & =\sum_{\mathbf{i}=1}^N\Big(\nabla_G\Phi_{\mathbf{i}}-\nabla_G\Phi_{\mathbf{i}}^k\Big)^2 \Delta \x.
\end{align*} 
These steps alternate a gradient ascent in the dual variable $\Phi$, and a gradient descent in the primal variable $\mathbf{m}$. 

It turns out that iteration \eqref{iteration} can be solved by simple explicit formulae. Since the unknown variables $\mathbf{m}$, $\Phi$ are component-wise separable in this problem, each of its components $\mathbf{m}_{\mathbf{i}+\frac{1}{2}}$, $\Phi_{\mathbf{i}}$ can be independently obtained by solving \eqref{iteration}. First, notice that
%\begin{equation*}
%\begin{split}
%$$
\begin{align*}
& \arg\min_{\mathbf{m}}~L(\mathbf{m},\Phi^k)+\frac{1}{2h}\|\mathbf{m}-\mathbf{m}^k\|^2_{L^2} \\
%=&\min_{m} \sum_{i=1}^N \sum_{v=1}^d \sqrt{m_{i+\frac{e_v}{2}}^2}+\frac{1}{\Delta x}\sum_{i=1}^N \sum_{v=1}^d\Phi_i^k (m_{i+\frac{1}{2}e_v}-m_{i-\frac{1}{2}e_v}) +\frac{\|m-m^k\|^2_2}{2\rho}   \\
= &\arg\min_{\mathbf{m}_{\mathbf{i}+\frac{1}{2}}}\sum_{\mathbf{i}=1}^N\Bigg(\|\mathbf{m}_{\mathbf{i}+\frac{1}{2}}\|_{\mathrm{tr}}- \Big(\nabla_G \Phi_{\mathbf{i}+\frac{1}{2}}^k\Big)^\T \mathbf{m}_{\mathbf{i}+\frac{1}{2}}+\frac{1}{2h}\|\mathbf{m}_{\mathbf{i}+\frac{1}{2}}-\mathbf{m}_{\mathbf{i}+\frac{1}{2}}^k\|^2_{L^2}\Bigg),
\end{align*}
%$$
%\end{split}
%\end{equation*}
where $\nabla_G\Phi^k_{\mathbf{i}+\frac{1}{2}}:=\frac{1}{\Delta \x}(\Phi^k_{\mathbf{i}+e_v}-\Phi_{\mathbf{i}}^k)_{v=1}^d$. The first iteration in \eqref{iteration} has an explicit solution, which is: \begin{equation*}
\mathbf{m}^{k+1}_{i+\frac{1}{2}}=\textrm{shrink}_{\mathrm{tr}}(\mathbf{m}_{\mathbf{i}+\frac{1}{2}}^k+h\nabla_G \Phi^k_{\mathbf{i}+\frac{1}{2}}, h),
\end{equation*}
where the {\em shrink} operator is a projection operation to the unit ball with norm $\|\cdot\|_{\mathrm{tr}}$.  Its exact formulation is given further on in Proposition \ref{prop:shrink}.

Second, consider
%\begin{equation*}
%\begin{split}
%$$
\begin{align*}
& \arg\max_{\Phi} L(2\mathbf{m}^{k+1}-\mathbf{m}^k,\Phi)-\frac{1}{2\tau}\|\Phi-\Phi^k\|^2_2 \\
= &\arg\max_{\Phi} \sum_{\mathbf{i}=1}^N \max_{\Phi_\mathbf{i}}\Big( \Phi_{\mathbf{i}} \big(\textrm{div}_G(2\mathbf{m}^{k+1}_{\mathbf{i}}-\mathbf{m}^k_{\mathbf{i}})+q_{\mathbf{i}}^1-q_{\mathbf{i}}^0 \big)-\frac{1}{2\tau}\|\Phi_{\mathbf{i}}-\Phi^k_{\mathbf{i}}\|^2_{H^1}\Big).\\
\end{align*}
%$$
%\end{split}
%\end{equation*}
Thus the second iteration in \eqref{iteration} becomes 
\begin{equation*}
\Phi_{\mathbf{i}}^{k+1}=\Phi_{\mathbf{i}}^k+\tau (-\Delta_G)^{-1} \bigl( \textrm{div}_G(2\mathbf{m}_{\mathbf{i}}^{k+1}-\mathbf{m}^k_{\mathbf{i}})+q_\mathbf{i}^1-q_{\mathbf{i}}^0\bigl).
\end{equation*}
where $\Delta_G=\textrm{div}_G \cdot  \nabla_G$ is the discrete Laplacian operator. 

We are now ready to state our algorithm.
\begin{tabbing}
aaaaa\= aaa \=aaa\=aaa\=aaa\=aaa=aaa\kill  
   \rule{\linewidth}{0.8pt}\\
   \noindent{\large\bf G-Prox Primal-Dual Method for Computing}\\
   \noindent{\large\bf the Tropical Wasserstein-1 Distance}\\
\noindent\textbf{Input}: Discrete probabilities $q^0$, $q^1$; \\
\>~~~Initial guess of $\mathbf{m}^0$, step size $h$, $\tau$, tolerance $\epsilon$.\\
  \noindent \textbf{Output}: $\mathbf{m}$ and $W^\mathrm{tr}_1(\rho^0,\rho^1)$.\\
   \rule{\linewidth}{0.5pt}\\
% 1.  \> for $k=1, 2, \cdots$ \qquad \textrm{Iterates until convergence}\\
1.  \> \textbf{while} \textrm{the relative error of} $\|\mathbf{m}\|_\text{tr} >\epsilon$\\
2.  \>\> $\mathbf{m}^{k+1}_{\mathbf{i}+\frac{1}{2}}=\textrm{shrink}_{\textrm{tr}}(\mathbf{m}_{\mathbf{i}+\frac{1}{2}}^k+h\nabla_G \Phi^k_{\mathbf{i}+\frac{1}{2}}, h)$ ;\\
3.  \>\> $\Phi_\mathbf{i}^{k+1}=\Phi_\mathbf{i}^k+\tau (-\Delta_{G})^{-1} \bigl( \textrm{div}_G(2\mathbf{m}_{\mathbf{i}}^{k+1}-\mathbf{m}^k_{\mathbf{i}})+q_\mathbf{i}^1-q^0_{\mathbf{i}}\bigl)$ ;\\ 
4.  \> \textbf{end}\\
   \rule{\linewidth}{0.8pt}
\end{tabbing}
\begin{remark}
The relative error at iteration $k$ is given by $\displaystyle \frac{|\|\mathbf{m}^k\|_\text{tr}-\|\mathbf{m}^{k-1}\|_\text{tr}|}{\|\mathbf{m}^{k-1}\|_\text{tr}}$.
%\begin{equation*}
    %\text{The relative error} = \frac{|\|m^k\|_\text{tr}-\|m^{k-1}\|_\text{tr}|}{\|m^{k-1}\|_\text{tr}}
%\end{equation*}
\end{remark}

In the algorithm, we require the shrink operator with respect to the tropical metric, $\textrm{shrink}_{\textrm{tr}}$, which is given in the following result.
\begin{proposition}
\label{prop:shrink}
	Let $h>0$ and $b_{1}\ge b_{2}\ge \cdots \ge b_{k}\ge 0 > b_{k+1} \ge \cdots \ge b_{n}$. We denote 
	
	\[u_{i} = b_{i} \: \forall\ 1\le i\le k,~~~u_{k+1} = 0 \]
	and
	\[v_{i} = -b_{n+1-i} \: \forall\ 1\le i\le n-k,~~~v_{n-k+1} = 0. \]
	
	Suppose
	\begin{align*}
	j_{1} & =\begin{cases}
	\displaystyle \max \bigg(1\le j\le k+1 \ \Big| \ \sum_{i=1}^{j}{(u_{i} - u_{j})} < 1 \bigg), & \text{ if } k\ge 1, \\
	0, & \text{ if } k=0,\\
	\end{cases}
	\\
	\ell_{1} & = \max(j_{1},k);
	\end{align*}
	and
	\begin{align*}
	j_{2} & = \begin{cases}
	\displaystyle \max \bigg(1\le j\le n-k+1 \ \Big| \ \sum_{i=1}^{j}{(v_{i} - v_{j})} < 1 \bigg), & \text{ if } k\le n-1, \\
	0, & \text{ if } k=n,
	\end{cases}
	\\
	\ell_{2} & = \max(j_{2},n-k).
	\end{align*}
	We let
	\[t_{1} = 
	\begin{cases}
		\displaystyle \frac{\left(\sum_{i=1}^{j_{1}}{u_{i}}\right) - 1}{j_{1}} &\text{ if } 1\le j_{1}\le k; \\
		0 &\text{ otherwise}.
	\end{cases} \]
	and
	\[t_{2} = 
	\begin{cases}
		\displaystyle \frac{\left(\sum_{i=1}^{j_{2}}{v_{i}}\right) - 1}{j_{2}} &\text{ if } 1\le j_{2}\le n-k; \\
		0 &\text{ otherwise}.
	\end{cases} \]
	Then
	\begin{equation}\label{eq:shrink}
	\mathrm{shrink}_{tr}(\mathbf{b}, h):=\argmin_{\mathbf{a}\in \R^{n}}{\left\{\frac{\sum_{i=1}^{n}{a_{i}^{2}}}{2h} + \|\mathbf{a}\|_{\mathrm{tr}} - \sum_{i=1}^{n}{b_{i}\cdot a_{i}}\right\}}
	\end{equation}
	is the following unique point $\x\in \R^{n}$, where
	
	\[x_{i} =
	\begin{cases}
		h\cdot t_{1}, &\text{ if }i\le \ell_{1}; \\
		h\cdot b_{i}, &\text{ if } \ell_{1} < i < n + 1 - \ell_{2}; \\
		-h\cdot t_{2}, &\text{ if } i\ge n + 1 - \ell_{2}.
	\end{cases} \]	
\end{proposition}

\begin{proof}
	Note that by definition of $t_{1}, t_{2}$, they are bounded by all of $u_{i}$ with $i\le j_{1}$ and all of $v_{i}$ with $i\le j_{2}$, respectively. In addition, we have
	\begin{equation}\label{eq:t1}
		\sum_{i=1}^{\ell_{1}}{\left(u_{i} - t_{1} \right)} \le 1
	\end{equation}
	and
	\begin{equation}\label{eq:t2}
		\sum_{i=1}^{\ell_{2}}{\left(v_{i} - t_{2} \right)} \le 1.
	\end{equation}
	Now we claim that
	\begin{equation}\label{eq:tropnorm-ine}
		\|\mathbf{a}\|_{\mathrm{tr}} \ge \sum_{i=1}^{\ell_{1}}{\left(u_{i} - t_{1} \right)\cdot a_{i}} - \sum_{i=1}^{\ell_{2}}{\left(v_{i} - t_{2} \right)\cdot a_{n+1-i}}.
	\end{equation}
	Notice that (\ref{eq:t1}) implies that
	\[ \sum_{i=1}^{\ell_{1}}{\left(u_{i} - t_{1} \right)\cdot a_{i}} \le \max_{1\le i\le j_{1}}{a_{i}}. \]
	We also have that (\ref{eq:t2}) implies that
	\[ \sum_{i=1}^{\ell_{2}}{\left(v_{i} - t_{2} \right)\cdot a_{n+1-i}} \ge \bigg(\sum_{i=1}^{\ell_{2}}{\left(v_{i} - t_{2} \right)}\bigg)\cdot \min_{1\le i\le j_{2}}{a_{n+1-i}} \ge \min\Big(0, \min_{1\le i\le j_{2}}{a_{n+1-i}} \Big). \]
	Hence, the right-hand side of (\ref{eq:tropnorm-ine})
	\begin{align*}  
	\sum_{i=1}^{\ell_{1}}{\left(u_{i} - t_{1} \right)\cdot a_{i}} - \sum_{i=1}^{\ell_{2}}{\left(v_{i} - t_{2} \right)\cdot a_{n+1-i}} & \le \max_{1\le i\le j_{1}}{a_{i}} - \min\Big(0, \min_{1\le i\le j_{2}}{a_{n+1-i}} \Big)\\
	& = \max_{1\le i_{1} \le j_{1}, 1\le i_{2} \le j_{2}}{\big(a_{i_{1}},\, a_{i_{1}} - a_{i_{2}} \big)}\\ 
	& \le \| \mathbf{a} \|_{\mathrm{tr}}.
	\end{align*}
	So our claim is proved. \\
	
	Since $h>0$ is a constant, we can multiply the objective function in (\ref{eq:shrink}) by $2h$. Now, this new function is greater than or equal to
	\begin{align*}
		& \sum_{i=1}^{n}{a_{i}^{2}} + 2h\left(\sum_{i=1}^{\ell_{1}}{\left(u_{i} - t_{1} \right)\cdot a_{i}} - \sum_{i=1}^{\ell_{2}}{\left(v_{i} - t_{2} \right)\cdot a_{n+1-i}}\right) - 2h\sum_{i=1}^{n}{b_{i}\cdot a_{i}} \\
		= & \sum_{i=1}^{n}{a_{i}^{2}} - \sum_{i=1}^{\ell_{1}}{2ht_{1}\cdot a_{i}} + \sum_{i=1}^{\ell_{2}}{2ht_{2}\cdot a_{n+1-i}} - 2h\sum_{i=\ell_{1}+1}^{n-\ell_{2}}{b_{i}\cdot a_{i}}.
	\end{align*}
	
	The global minimum of the last quadratic polynomial is attained exactly at the point $\x$ in Proposition \ref{prop:shrink}, so we have a lower bound for the new objective function, which is given when $\mathbf{a}=\x$. Finally, we note that the equality of (\ref{eq:tropnorm-ine}) is attained at $\x$, so this value is actually attained by $\mathbf{a}=\x$.
\end{proof}

\begin{example}
    When $n=2$, given $(b_{1},b_{2})\in \R^{2}$, suppose $x_{1} = f_{1}(b_{1},b_{2})$ and $x_{2} = f_{2}(b_{1},b_{2})$, then the shrink operator is given as follows. 
    \begin{table}[H]
        \centering
        \begin{tabular}{|c|c|c|c|c|c|c|c|c|c|}
            \hline
             $\mathbf{b}$ & $k$ & $j_{1}$ & $j_{2}$ & $\ell_{1}$& $\ell_{2}$ & $t_{1}$ & $t_{2}$ & $x_{1}$ & $x_{2}$ \\
             \hline\hline
             $b_{1}\ge b_{2}+1, b_{2}\ge 0$ & $2$ & $1$ & $0$ & $1$ & $0$ & $b_{1} - 1$ & $0$ & $h(b_{1}-1)$ & $b_{2}$ \\
             \hline
             $b_{1}<b_{2}+1, b_{1}\ge 1-b_{2}, b_{1}\ge b_{2}$ & $2$ & $2$ & $0$ & $2$ & $0$ & $\frac{b_{1}+b_{2}-1}{2}$  & $\frac{b_{1}+b_{2}-1}{2}$ & $h\frac{b_{1}+b_{2}-1}{2}$ & $h\frac{b_{1}+b_{2}-1}{2}$ \\
             \hline
             $b_{1}<1-b_{2}, b_{1}\ge b_{2}\ge 0$ & $2$ & $3$ & $0$ & $2$ & $0$ & $0$ & $0$ & $0$ & $0$ \\
             \hline
             $1>b_{1}\ge 0, 0\ge b_{2}>-1$ & $1$ & $2$ & $2$ & $1$ & $1$ & $0$ & $0$ & $0$ & $0$ \\
             \hline
             $1>b_{1}\ge 0, b_{2}\le -1$ & $1$ & $2$ & $1$ & $1$ & $1$ & $0$ & $-b_{2}-1$ & $0$ & $h(b_{2}+1)$ \\
             \hline
             $b_{1}\ge 1, 0\ge b_{2}>-1$ & $1$ & $1$ & $2$ & $1$ & $1$ & $b_{1}-1$ & $0$ & $h(b_{1}-1)$ & $0$ \\
             \hline
             $b_{1}\ge 1, b_{2}\le -1$ & $1$ & $1$ & $1$ & $1$ & $1$ & $b_{1}-1$ & $-b_{2}-1$ & $h(b_{1}-1)$ & $h(b_{2}+1)$ \\
             \hline
             $0\ge b_{1}\ge b_{2}$ & $0$ & $1$ &  & $1$ &  & $0$ & & $-f_{2}(-b_{2},-b_{1})$ & $-f_{1}(-b_{2},-b_{1})$ \\
             \hline
             $b_{1}<b_{2}$ & & & & & & & & $f_{2}(b_{2}, b_{1})$ & $f_{1}(b_{2}, b_{1})$ \\
             \hline
        \end{tabular}
        \caption{The operator $\textrm{shrink}_{\mathrm{tr}}$ when $n=2$}
        \label{tab:shrink_ex2}
    \end{table}
    
\end{example}

\begin{remark}
	Proposition \ref{prop:shrink} provides an algorithm to compute the shrink. Suppose we have $h>0$ and $\mathbf{a}_0, \mathbf{b}\in \R^{n}$ and we would like find
	\begin{equation*}
	\textrm{shrink}_{\mathrm{tr}}(\mathbf{a}_0+h \mathbf{b}, h)=\argmin_{\mathbf{a}\in \R^{n}} \left\{ \frac{|\mathbf{a}-\mathbf{a}_0|_{2}^{2}}{2h} + \|\mathbf{a}\|_{\mathrm{tr}} - \sum_{i=1}^{n}{b_{i}\cdot a_{i}} \right\}.
	\end{equation*}			
	Note that
	\[|\mathbf{a}-\mathbf{a}_0|_{2}^{2} = \sum_{i=1}^{n}{\left(a_{i} - a_{0i} \right)^{2}} = \sum_{i=1}^{n}{a_{i}^{2}} - \sum_{i=1}^{n}{2a_{0i}\cdot a_{i}} + \text{ constant.} \]
	Then we let $\mathbf{b}' = \mathbf{b} + \frac{\mathbf{a}_0}{h}$, the optimization problem becomes the one in Proposition \ref{prop:shrink} for $\mathbf{b}'$ and $h$ after sorting the coordinates of $\mathbf{b}'$.
\end{remark}

\subsection{Computing the Tropical Wasserstein-$2$ Distances}
\label{sec:computation-l2}

We now present an algorithm to compute the tropical Wasserstein-2 distance in the tropical projective torus $\R^3/\R\one$ identified with $\mathbb{R}^2$. Consider the same uniform lattice graph on a domain $\Omega \subset \mathbb{R}^2$ as in the case for the tropical Wasserstein-1 distance. Define the following matrices 
\begin{align*}
\boldsymbol{\rho} &= \big(\boldsymbol{\rho}^n_{\mathbf{i}} \big)^{N_x, N_t}_{\mathbf{i},n=1}\\
%    \rho &= \big(\rho^n_i \big)^{N_x}_{i=1}{}^{N_t}_{n=1}\\
\mathbf{m} &= \Big(\mathbf{m}^n_{\mathbf{i}+\frac{1}{2}e_v} \Big)^{d, N_x, N_t}_{v,\mathbf{i},n=1}
%    m &= \Big(m^n_{i+\frac{1}{2}e_v} \Big)^d_{v=1}{}^{N_x}_{i=1}{}^{N_t}_{n=1}
\end{align*}
where the time interval is discretized uniformly with $N_t$ points, and $N_x$ is the number of vertices from a uniform lattice graph. Here we assume Neumann boundary conditions for $\boldsymbol{\rho}$: $\displaystyle \frac{\partial \rho}{\partial \hat{\mathbf{n}}} = 0$ on $\partial \Omega$, where $\hat{\mathbf{n}}$ is a outward normal vector.
% that is,
% \begin{align*}
%     \frac{\partial \rho(t,x)}{\partial n(t,x)}=0, \indent m(t,x)\cdot n(t,x)=0     \indent \text{ if } x\in \partial \Omega,\indent  t\in[0,1]
% \end{align*}
% where $n(t,x)$ is an outward normal vector.
Given initial densities $\rho_0$ and $\rho_1$, the boundary conditions for $\rho$ at $t=0$ and $t=1$ are
\begin{align*}
    \big(\boldsymbol{\rho}^1_{\mathbf{i}} \big)^{N_x}_{{\mathbf{i}}=1} = \rho_0 \mbox{~~~and~~~}  \big(\boldsymbol{\rho}^{N_t}_{\mathbf{i}} \big)^{N_x}_{\mathbf{i}=1} = \rho_1.
\end{align*}
Define $\Delta t := \frac{1}{N_t}$. We can reformulate the minimization problem (\ref{UOT1}) into a discretization as follows:
\begin{equation}\label{W2new}
\begin{aligned}
& \underset{\mathbf{m}}{\text{minimize}}
& &  \sum^{N_t}_{n=1} \sum^{N_x}_{\mathbf{i}=1} \frac{\|\mathbf{m}^n_{\mathbf{i}+\frac{1}{2}}\|^2_{\mathrm{tr}}}{2 \boldsymbol{\rho}^n_{\mathbf{i}}}\Delta \x \Delta t \\
& \text{subject to}
& & \partial_t \boldsymbol{\rho}^n_\mathbf{i} + \text{div}_G(\mathbf{m}^n_{\mathbf{i}})=0, \quad \mathbf{i}=1,\ldots, N_x;~ n=1,\ldots, N_t\\
& & & \big(\boldsymbol{\rho}^1_{\mathbf{i}}\big)^{N_x}_{{\mathbf{i}}=1}=\rho_0,\\
& & & \big(\boldsymbol{\rho}^{N_t}_{\mathbf{i}} \big)^{N_x}_{{\mathbf{i}}=1}.
\end{aligned}
\end{equation}
where 
\begin{equation*}
    \partial_t \boldsymbol{\rho}^n_{\mathbf{i}}
        = 
    \begin{cases}
        \frac{1}{\Delta t} (\boldsymbol{\rho}^{n+1}_{\mathbf{i}} - \boldsymbol{\rho}^{n}_{\mathbf{i}}) & \text{ for } n = 1\\
        \frac{1}{2\Delta t} (\boldsymbol{\rho}^{n+1}_{\mathbf{i}} - \boldsymbol{\rho}^{n-1}_{\mathbf{i}}) & \text{ for } n = 2,\ldots, N_t-1\\
        \frac{1}{\Delta t} (\boldsymbol{\rho}^{n}_{\mathbf{i}} - \boldsymbol{\rho}^{n-1}_{\mathbf{i}}) & \text{ for } n = N_t
    \end{cases}
\end{equation*}
and
\begin{equation*}
    \text{div}_G(\mathbf{m}^n_{\mathbf{i}}) = \frac{1}{\Delta x} \sum^2_{v=1} \Big(\mathbf{m}^n_{\mathbf{i}+\frac{1}{2}e_v} - \mathbf{m}^n_{\mathbf{i}-\frac{1}{2}e_v} \Big)\indent \text{ for } n=1,\ldots,N_t.
\end{equation*}
In $\mathbb{R}^2$, using (\ref{L}), we can calculate the tropical norm of the flux function $\mathbf{m}$ by considering the six different cases based on $\{\mathbf{m}_{\mathbf{i}+\frac{1}{2}e_v}\}^2_{v=1}$. The tropical norm of $\mathbf{m}$ is given as follows:

\begin{table}[H]
    \centering
    \begin{tabular}{|c|c|}
        \hline
         $\mathbf{m}_{\mathbf{i}+\frac{1}{2}}$ & $\|\mathbf{m}_{\mathbf{i}+\frac{1}{2}}\|_{\mathrm{tr}}$  \\
         \hline\hline
         $\mathbf{m}_{\mathbf{i}+\frac{1}{2}e_1}>\mathbf{m}_{\mathbf{i}+\frac{1}{2}e_2}>0$ & $\mathbf{m}_{\mathbf{i}+\frac{1}{2}e_1}$\\
         \hline
         $\mathbf{m}_{\mathbf{i}+\frac{1}{2}e_2}>\mathbf{m}_{\mathbf{i}+\frac{1}{2}e_1}>0$ & $\mathbf{m}_{\mathbf{i}+\frac{1}{2}e_2}$\\
         \hline
         $0>\mathbf{m}_{\mathbf{i}+\frac{1}{2}e_2}>\mathbf{m}_{\mathbf{i}+\frac{1}{2}e_1}$ & $-\mathbf{m}_{\mathbf{i}+\frac{1}{2}e_1}$\\
         \hline
         $0>\mathbf{m}_{\mathbf{i}+\frac{1}{2}e_1}>\mathbf{m}_{\mathbf{i}+\frac{1}{2}e_2}$ & $-\mathbf{m}_{\mathbf{i}+\frac{1}{2}e_2}$\\
         \hline
         $\mathbf{m}_{\mathbf{i}+\frac{1}{2}e_1}>0>\mathbf{m}_{\mathbf{i}+\frac{1}{2}e_2}$ & $\mathbf{m}_{\mathbf{i}+\frac{1}{2}e_1}-\mathbf{m}_{\mathbf{i}+\frac{1}{2}e_2}$ \\
         \hline
         $\mathbf{m}_{\mathbf{i}+\frac{1}{2}e_2}>0>\mathbf{m}_{\mathbf{i}+\frac{1}{2}e_1}$ & $\mathbf{m}_{\mathbf{i}+\frac{1}{2}e_2}-\mathbf{m}_{\mathbf{i}+\frac{1}{2}e_1}$ \\
         \hline
    \end{tabular}
    \caption{Tropical norm when $n=2$}
    \label{}
\end{table}

% \begin{figure}[H]
%     \centering
%     \begin{tabular}{|c|c|c|}
%         \hline
%          $m_{i+\frac{1}{2}}$ & $\|m_{i+\frac{1}{2}}\|_{tr}$ & $\nabla_m \|m_{i+\frac{1}{2}}\|_{tr}$ \\
%          \hline
%          $m_{i+\frac{1}{2}e_1}>m_{i+\frac{1}{2}e_2}>0$ & $m_{i+\frac{1}{2}e_1}$ & $(1,0)$ \\
%          \hline
%          $m_{i+\frac{1}{2}e_2}>m_{i+\frac{1}{2}e_1}>0$ & $m_{i+\frac{1}{2}e_2}$ & $(0,1)$ \\
%          \hline
%          $0>m_{i+\frac{1}{2}e_2}>m_{i+\frac{1}{2}e_1}$ & $-m_{i+\frac{1}{2}e_1}$ & $(-1,0)$ \\
%          \hline
%          $0>m_{i+\frac{1}{2}e_1}>m_{i+\frac{1}{2}e_2}$ & $-m_{i+\frac{1}{2}e_2}$ &  $(0,-1)$\\
%          \hline
%          $m_{i+\frac{1}{2}e_1}>0>m_{i+\frac{1}{2}e_2}$ & $m_{i+\frac{1}{2}e_1}-m_{i+\frac{1}{2}e_2}$ &  $(1,-1)$\\
%          \hline
%          $m_{i+\frac{1}{2}e_2}>0>m_{i+\frac{1}{2}e_1}$ & $m_{i+\frac{1}{2}e_2}-m_{i+\frac{1}{2}e_1}$ &  $(-1,1)$\\
%          \hline
%     \end{tabular}
%     % \caption{Shrink when $n=2$}
%     \label{}
% \end{figure}
Let $\Phi=(\Phi^n_{\mathbf{i}})_{\mathbf{i}=1}^{N_x}{}_{n=1}^{N_t}$ here be the Lagrange multiplier which satisfies the Neumann boundary condition on the boundary of the domain. The minimization problem (\ref{W2new}) can be reformulated as a saddle point problem. 
\begin{equation}\label{saddle}
\min_{\mathbf{m},\boldsymbol{\rho}}\max_{\Phi} \quad L(\mathbf{m},\boldsymbol{\rho}, \Phi):=\min_{\mathbf{m},\boldsymbol{\rho}}\max_{\Phi} \quad \sum^{N_t}_{n=1} \sum^{N_x}_{\mathbf{i}=1} \frac{\|\mathbf{m}^n_{\mathbf{i}+\frac{1}{2}}\|^2_{\mathrm{tr}}}{2\boldsymbol{\rho}^n_{\mathbf{i}}}+\Phi^n_{\mathbf{i}} \Big(\partial_t\boldsymbol{\rho}^n_{\mathbf{i}} + \textrm{div}_G\big(\mathbf{m}^n_{\mathbf{i}+\frac{1}{2}} \big) \Big).
\end{equation}
Again, we implement G-Prox PDHG to solve the problem as follows:
\begin{equation}\label{iteration-WT2}
%\begin{cases}
\left\{
\begin{array}{ll}
\boldsymbol{\rho}^{k+1} = \text{arg}\min_{\boldsymbol{\rho}} & \quad 
            L(\mathbf{m}^k,\boldsymbol{\rho},\Phi^k) + \frac{1}{2\tau} \|\boldsymbol{\rho}-\boldsymbol{\rho}^k\|^2_{L^2(\Omega\times[0,1])},\\
m^{k+1} = \text{arg}\min_{\mathbf{m}} & \quad
            L(\mathbf{m},\boldsymbol{\rho}^{k+1},\Phi^k) + \frac{1}{2\tau} 
            \|\mathbf{m}-\mathbf{m}^k\|^2_{L^2(\Omega\times[0,1])},\\
\Phi^{k+1} = \text{arg}\max_\Phi & \quad L(2\mathbf{m}^{k+1} - \mathbf{m}^k,2\boldsymbol{\rho}^{k+1}-\boldsymbol{\rho}^k,\Phi) - \frac{1}{2h} \|\Phi-\Phi^k\|^2_{H^1(\Omega\times[0,1])},
\end{array}
\right.
%\end{cases}
\end{equation}

\noindent where $h$, $\tau$ are two small step sizes and
\begin{align*}
\|\boldsymbol{\rho}-\boldsymbol{\rho}^k\|^2_{L^2}&=\sum^{N_t}_{n=1} \sum^{N_x}_{\mathbf{i}=1} \big(\boldsymbol{\rho}^n_{\mathbf{i}}-(\boldsymbol{\rho}^n_{\mathbf{i}})^k\big)^2 \Delta \x \Delta t\\
    \|\Phi-\Phi^k\|^2_{H^1}&=\sum^{N_t}_{n=1} \sum^{N_x}_{\mathbf{i}=1} \left( (\partial_t \Phi^n_{\mathbf{i}} - \partial_t (\Phi^n_{\mathbf{i}})^k)^2 + \|\nabla_G\Phi^n_{\mathbf{i}} - \nabla_G(\Phi^n_{\mathbf{i}})^k\|^2 \right) \Delta \x \Delta t.
\end{align*}
From (\ref{iteration-WT2}), each component $\mathbf{m}^n_{\mathbf{i}+\frac{1}{2}}$, $\boldsymbol{\rho}^n_{\mathbf{i}}$, and $\Phi^n_{\mathbf{i}}$ can be obtained. From the first iteration,
\begin{equation*}
    \begin{aligned}
    \boldsymbol{\rho}^{k+1}=&\text{arg}\min_{\boldsymbol{\rho}} \quad
                L(\mathbf{m}^k,\boldsymbol{\rho},\Phi^k) + \frac{1}{2\tau} \|\boldsymbol{\rho}-\boldsymbol{\rho}^k\|^2_{L^2}
            \\
    =&\text{arg}\min_{\boldsymbol{\rho}} \quad
              \sum^{N_t}_{n=1}\sum^{N_x}_{\mathbf{i}=1} \frac{\|(\mathbf{m}^n_{\mathbf{i}+\frac{1}{2}})^k\|^2_{\mathrm{tr}}}{2\boldsymbol{\rho}^n_{\mathbf{i}}}+(\Phi^n_{\mathbf{i}})^k \partial_t\boldsymbol{\rho}^n_{\mathbf{i}}  + \frac{1}{2\tau} 
               \|\boldsymbol{\rho}-\boldsymbol{\rho}^k\|^2_{L^2}
    \end{aligned}
\end{equation*}
We calculate the minimizer by differentiating the equation with respect to $\boldsymbol{\rho}^n_{\mathbf{i}}$. The minimizer $\boldsymbol{\rho}^{k+1}$ is a positive root of the following cubic polynomial:
\begin{align*} 
    - \frac{\|(\mathbf{m}^n_{\mathbf{i}})^{k}\|^2_{\text{tr}}}{2 ((\boldsymbol{\rho}^n_{\mathbf{i}})^{k+1})^2} - \partial_t (\Phi^n_{\mathbf{i}})^k + \frac{1}{\tau} \big((\boldsymbol{\rho}^n_{\mathbf{i}})^{k+1} - (\boldsymbol{\rho}^n_{\mathbf{i}})^k \big) = 0.
\end{align*}
Thus, we can calculate the root by using a cubic solver.
\begin{align*} 
    (\boldsymbol{\rho}^n_{\mathbf{i}})^{k+1} = \text{root}^+\biggl(-(\boldsymbol{\rho}^n_{\mathbf{i}})^k - \tau \partial_t (\Phi^n_{\mathbf{i}})^k, 0, -\frac{\tau}{2} \|(\mathbf{m}^n_{\mathbf{i}})^k\|^2_{\text{tr}} \biggl),
\end{align*}
where $\text{root}^+(a,b,c)$ is a solution for a cubic polynomial $x^3 + a x^2 + b x + c = 0$.\\

We can reformulate the second iteration as follows:
\begin{equation*}
    \begin{aligned}
    \mathbf{m}^{k+1} &= \text{arg}\min_{\mathbf{m}} \quad
            L(\mathbf{m},\boldsymbol{\rho}^{k+1},\Phi^k) + \frac{1}{2\tau} 
            \|\mathbf{m}-\mathbf{m}^k\|^2_{L^2}
        \\
    &= \text{arg}\min_{\mathbf{m}} \sum^{N_t}_{n=1} \sum^{N_x}_{\mathbf{i}=1} \frac{\|\mathbf{m}^n_{\mathbf{i}+\frac{1}{2}}\|^2_{\mathrm{tr}}}{2(\rho^n_{\mathbf{i}})^{k+1}} + \Phi^n_{\mathbf{i}} \text{div}_G(\mathbf{m}^n_{\mathbf{i}+\frac{1}{2}})+ \frac{1}{2\tau} 
            \|\mathbf{m}-\mathbf{m}^k\|^2_{L^2}\\
    &= \text{arg}\min_{\mathbf{m}} \sum^{N_t}_{n=1} \sum^{N_x}_{\mathbf{i}=1} \frac{\|\mathbf{m}^n_{\mathbf{i}+\frac{1}{2}}\|^2_{\mathrm{tr}}}{2(\boldsymbol{\rho}^n_{\mathbf{i}})^{k+1}} - \mathbf{m}^n_{\mathbf{i}+\frac{1}{2}} \nabla_G \Phi^n_{\mathbf{i}} + \frac{1}{2\tau} 
            \|\mathbf{m}-\mathbf{m}^k\|^2_{L^2}\\            
    &= \text{arg}\min_{\mathbf{m}} \sum^{N_t}_{n=1} \sum^{N_x}_{\mathbf{i}=1} \frac{\|\mathbf{m}^n_{\mathbf{i}+\frac{1}{2}}\|^2_{\mathrm{tr}}}{2(\boldsymbol{\rho}^n_{\mathbf{i}})^{k+1}} + \frac{1}{2\tau} 
            \|\mathbf{m}-\mathbf{m}^k-\tau \nabla_G \Phi\|^2_{L^2}.
    \end{aligned}
\end{equation*}
\noindent Differentiating the equation with respect to $\mathbf{m}^n_{\mathbf{i}+\frac{1}{2}}$, we obtain the following expression:
\begin{equation*}
    \begin{split}
        \frac{\|\mathbf{m}^n_{\mathbf{i}+\frac{1}{2}}\|_\text{tr} \nabla_G \|\mathbf{m}^n_{\mathbf{i}+\frac{1}{2}}\|_\text{tr}}{(\boldsymbol{\rho}^n_{\mathbf{i}})^{k+1}} + \frac{1}{\tau}\big(\mathbf{m}^n_{\mathbf{i}+\frac{1}{2}} - (\mathbf{m}^n_{\mathbf{i}+\frac{1}{2}})^k - \tau \nabla_G \Phi^n_{\mathbf{i}} \big) = 0.
    \end{split}
\end{equation*}
Solving this expression gives an explicit solution for $(\mathbf{m}^n_{\mathbf{i}+\frac{1}{2}})^{k+1}$:
\begin{align}
    (\mathbf{m}^n_{\mathbf{i}+\frac{1}{2}})^{k+1}=\bm{F}\Big( \big(\mathbf{m}^n_{\mathbf{i}+\frac{1}{2}} \big)^k+\tau \nabla_G (\Phi^n_{\mathbf{i}})^k,\, \tau/(\boldsymbol{\rho}^n_{\mathbf{i}})^{k+1} \Big).
\end{align}
\noindent Let $\mu=\tau/(\boldsymbol{\rho}^n_{\mathbf{i}})^{k+1}$ and $c=(c_1,c_2)$ be
\begin{align*}
    c_1=(\mathbf{m}^n_{\mathbf{i}+\frac{1}{2}e_1})^k+\tau \nabla_{x_1} (\Phi^n_{\mathbf{i}+\frac{1}{2}e_1})^k\\
    c_2=(\mathbf{m}^n_{\mathbf{i}+\frac{1}{2}e_2})^k+\tau \nabla_{x_2} (\Phi^n_{\mathbf{i}+\frac{1}{2}e_2})^k.
\end{align*}
The function $\bm{F}(c,\mu)$ is then given as follows:

\begin{table}[H]
        \centering
        \begin{tabular}{|c|c|}
            \hline
              $c_1,c_2,\mu$ & $\bm{F}(c,\mu)$ \\
             \hline\hline
             $c_2>(1+\mu)c_1>0$ or $c_2<(1+\mu)c_1<0$ & $\displaystyle \Big(c_1,\frac{c_2}{1+\mu} \Big)$  \\
             \hline
             $c_1>(1+\mu)c_2>0$ or $c_1<(1+\mu)c_2<0$ & $\displaystyle \Big(\frac{c_1}{1+\mu},c_2 \Big)$ \\
             \hline
             $-\frac{\mu}{1+\mu}c_1>c_2>-\frac{1+\mu}{\mu}c_1$ or $-\frac{\mu}{1+\mu}c_1<c_2<-\frac{1+\mu}{\mu}c_1$ & $\displaystyle \bigg(\frac{(1+\mu)c_1+\mu c_2}{1+2\mu},\frac{(1+\mu)c_2+\mu c_1}{1+2\mu} \bigg)$ \\
             \hline
             $-\frac{\mu}{1+\mu}c_1>c_2>0$ or $-\frac{\mu}{1+\mu}c_1<c_2<0$ & $\displaystyle \Big(\frac{c_1}{1+\mu},0 \Big)$ \\
             \hline
             $c_2>-\frac{1+\mu}{\mu}c_1>0$ or $c_2<-\frac{1+\mu}{\mu}c_1<0$ & $\displaystyle \Big(0,\frac{c_2}{1+\mu} \Big)$ \\
             \hline
             $(1+\mu)c_1>c_2>\frac{1}{1+\mu}c_1$ or $(1+\mu)c_1<c_2<\frac{1}{1+\mu}c_1$ & $\displaystyle \Big(\frac{c_1+c_2}{2+\mu},\frac{c_1+c_2}{2+\mu} \Big)$ \\
             \hline
        \end{tabular}
        \caption{The definition of $\bm{F}(c,\mu)$}
        \label{}
    \end{table}
    
Similarly, we get an explicit formula of $\Phi^{k+1}$ from the third iteration.
\begin{equation*}
    (\Phi^n_{\mathbf{i}})^{k+1} = (\Phi^n_{\mathbf{i}})^k + h (-\Delta_{t,G})^{-1}\Big( \partial_t \big(2(\boldsymbol{\rho}^n_{\mathbf{i}})^{k+1}-(\boldsymbol{\rho}^n_{\mathbf{i}})^k\big) + \text{div}_{t,G} \Big(2\big(\mathbf{m}^n_{\mathbf{i}+\frac{1}{2}} \big)^{k+1}- \big(\mathbf{m}^n_{\mathbf{i}+\frac{1}{2}} \big)^k \Big)  \Big)
\end{equation*}
for $\mathbf{i}=1,\ldots N_x$ and $n=1,\ldots,N_t$. 
%\newpage
Here, $\Delta_{t,G} = \partial_{tt} + \Delta_G$ is the discrete Laplacian operator over time and space. 

%\begin{remark}
Now, define 
$$
E^k :=\sum^{N_t}_{n=1} \sum^{N_x}_{\mathbf{i}=1} \frac{\big\|\big(\mathbf{m}^n_{\mathbf{i}+\frac{1}{2}} \big)^k \big\|^2_{\mathrm{tr}}}{2(\boldsymbol{\rho}^n_{\mathbf{i}})^k}.
$$
Then the relative error at iteration $k$ is calculated as $\displaystyle \frac{|E^k-E^{k-1}|}{|E^{k-1}|}$.
%\end{remark}

We are now ready to present our algorithm to compute the tropical Wasserstein-2 metric.
\begin{tabbing}
aaaaa\= aaa \=aaa\=aaa\=aaa\=aaa=aaa\kill  
   \rule{\linewidth}{0.8pt}\\
   \noindent{\large\bf G-Prox Primal-Dual Method for Computing the}\\
   \noindent{\large\bf Tropical Wasserstein-2 Distance}\\
\noindent\textbf{Input}: Discrete probabilities $\rho^0$, $\rho^1$; \\
\> {Initial} guess of $\boldsymbol{\rho}$, $\mathbf{m}$, $\Phi$, step size $\tau$, $h$,tolerance $\epsilon$\\
  \noindent \textbf{Output}: $\mathbf{m}$ and $W^{\mathrm{tr}}_2(\rho^0,\rho^1)$.\\
   \rule{\linewidth}{0.5pt}\\
1.  \> \textbf{while} \textrm{the relative error of } $\displaystyle \sum^{N_t}_{n=1} \sum^{N_x}_{\mathbf{i}=1} \frac{\big\|\mathbf{m}^n_{\mathbf{i}+\frac{1}{2}} \big\|^2_{\mathrm{tr}}}{2\boldsymbol{\rho}^n_{\mathbf{i}}}  > \epsilon$ \\
2.  \>\> $(\boldsymbol{\rho}^n_{\mathbf{i}})^{k+1} = \text{root}^+\biggl(-(\boldsymbol{\rho}^n_{\mathbf{i}})^k - \tau \partial_t (\Phi^n_{\mathbf{i}})^k,\, 0,\, -\frac{\tau}{2} \|(\mathbf{m}^n_{\mathbf{i}})^k\|^2_{\text{tr}} \biggl)$ ;\\
3.  \>\> $(\mathbf{m}^n_{\mathbf{i}+\frac{1}{2}})^{k+1} = \bm{F}\big((\mathbf{m}^n_{i+\frac{1}{2}})^k+\tau \nabla_G (\Phi^n_{\mathbf{i}})^k,\, \tau/(\boldsymbol{\rho}^n_{\mathbf{i}})^{k+1} \big)$;\\ 
4.  \>\> $\begin{aligned}[t] (\Phi^n_{\mathbf{i}})^{k+1} = (\Phi^n_{\mathbf{i}})^k & + h (-\Delta_{t,G})^{-1}\Big( \partial_t \big(2(\boldsymbol{\rho}^n_{\mathbf{i}})^{k+1}-(\boldsymbol{\rho}^n_i)^k \big)\\ & + \text{div}_G \big(2(\mathbf{m}^n_{\mathbf{i}})^{k+1}-(\mathbf{m}^n_{\mathbf{i}})^k \big) \Big);\end{aligned}$\\         
5.  \> \textbf{end}\\
   \rule{\linewidth}{0.8pt}
\end{tabbing}

\subsection{Convergence}

Our proposed primal-dual algorithms for the tropical Wasserstein-1 and tropical Wasserstein-2 distances converge to their respective minimizers as given by Propositions \ref{prop:min1} and \ref{prop:min2}. 

%%\begin{theorem}
%%Assume $\tau\mu <1/\lambda_\mathrm{max}(\Delta_G)$,
%%where $\lambda_\mathrm{max}(\Delta_G)$ denotes the largest eigenvalue
%%of the discrete Laplacian operator $\Delta_G$.
%%Then with iterations \eqref{iteration}
%%\[
%%(m^k,\Phi^k) \rightarrow (m^\star,\Phi^\star),
%%\]
%%where $(m^\star, \Phi^\star)$ is a saddle point of $L$ in \eqref{saddle}.
%%\end{theorem}
%%\begin{proof}
%%The proof follows the ones in \cite{pock1,pock2}. Let us rewrite the Lagrangian $L$ by 
%%\begin{equation*}
%%L(m, \Phi)=G(m)+\Phi^TKm-F(\Phi)\ ,
%%\end{equation*}
%%where $G(m)=\|m\|_{tr}$, $K=\textrm{div}_G$, and $F(\Phi)=\sum_{i}\Phi_{i}(p_i^0-p_j^1)$.
%%Observe that {$G$, $F$ are convex functions} and $K$ is a linear operator. Since $\Delta_G=KK^T$, the algorithm
%%converges for $\mu\tau\|\Delta_G\|_2^2<1$.
%%\end{proof}

\begin{theorem}
\begin{enumerate}[(i)]
\item Consider the G-Prox PDHG algorithm to compute the tropical Wasserstein-1 distance. Let 
$$
\sqrt{\tau\mu}\|(-\Delta_G)^{-\frac{1}{2}}\mathrm{div}_G\|_2<1.
$$ 
Then $(\mathbf{m}^{k}, \Phi^{k})$ defined by \eqref{iteration} converges weakly to $(\mathbf{m}^*,\Phi^*)$.\\

\item Consider the G-Prox PDHG algorithm to compute the tropical Wasserstein-2 distance. %Let $K:=\partial_t + \textrm{div}_G$ be a linear operator and $\Sigma := -\Delta_{t,G}$ be a symmetric positive definite map.  
%and assume there exists a saddle point $(m^*,\rho^*,\Phi^*)$ in \eqref{saddle}. 
Let
$$
\sqrt{\tau\mu}\| (-\Delta_{t,G})^{-\frac{1}{2}}\mathrm{div}_{t,G}\|_2<1.
$$
Then $(\mathbf{m}^{k}, \boldsymbol{\rho}^k, \Phi^{k})$ defined by \eqref{iteration-WT2} converges weakly to $(\mathbf{m}^*,\boldsymbol{\rho}^*,\Phi^*)$.
\end{enumerate}
\end{theorem}

\begin{proof}
The proof follows that of Theorem 1 in \cite{pock2}. We justify the conditions in \cite{pock2}. In the case of (i), we write the Lagrangian $L$ as 
\begin{equation*}
L(\mathbf{m}, \Phi)=g(\mathbf{m})+\Phi^\T K \mathbf{m}-f^*(\Phi),
\end{equation*}
where $g(\mathbf{m})=\|\mathbf{m}\|_{\mathrm{tr}}$, $K=\textrm{div}_G$, and $f^*(\Phi)=\sum_{\mathbf{i}}\Phi_{\mathbf{i}}(q_{\mathbf{i}}^0-q_{\mathbf{i}}^1)$.
Observe that $g$, $f^*$ are convex functions and $K$ is a linear operator. Then there exists a saddle point $(\mathbf{m}^*,\Phi^*)$.
Notice that the preconditioning norm for $\Phi$ is $\Sigma:=\mu (-\Delta_G)^{-1}$ and the preconditioning norm for $\mathbf{m}$ is $T:=\tau \cdot \mathrm{Id}$ where $\mathrm{Id}$ is an identity operator.
Thus, the algorithm converges when $\| \Sigma^{\frac{1}{2}}KT^{\frac{1}{2}}\|_2^2<1$. This is our condition $\sqrt{\tau\mu}\|(-\Delta_G)^{-\frac{1}{2}}\mathrm{div}_G\|_2<1$, which finishes the proof. A similar argument holds for (ii). 
\end{proof}

%\begin{theorem}[\cite{pock1,pock2}]
%(ii) Consider G-Prox PDHG for tropical Wasserstein-2 metric. Let $K:=\partial_t + \textrm{div}_G$ be a linear operator and $\Sigma := -\Delta_{t,G}$ which is a symmetric positive definite map,  and assume there exists a saddle point $(m^*,\rho^*,\Phi^*)$ in \eqref{saddle}. Then, ($m^{k}$, $\rho^k$,$\Phi^{k}$) defined by \eqref{iteration-WT2} converges weakly to $(m^*,\rho^*,\Phi^*)$.
%\end{theorem}
%\begin{proof}
%The proof follows the ones in \cite{pock2}.
%\end{proof}

%%%%%%%%%%%%%%%%%%%%%%%%%%%%%%%%%%%%%%%%%%%%%%%%%%%

\section{Numerical Experiments}
\label{sec:num_exp}
In this section, we present the results of numerical experiments solving the tropical optimal transport problem for three different sets of initial densities using our proposed G-Prox primal-dual methods for $L^1$ and $L^2$. In particular, we give the minimizers of $L^1$ and $L^2$ tropical optimal transport problems from each experiment.

\paragraph{Experiment 1.}
We consider a two-dimensional problem on $\Omega =  [0,1]\times[0,1]$. The initial densities $\rho_0$ and $\rho_1$ are same sizes of squares centered at $(\frac{1}{3},\frac{1}{3})$ and $(\frac{2}{3},\frac{2}{3})$, respectively. In this experiment, the parameters are
\begin{align*}
N_x &=128\times128,\\
N_t &=15.
\end{align*}
Figure \ref{fig:exp-l1-dp} shows the minimizer $m(x)$ of the tropical Wasserstein-1 distance and Figure \ref{fig:exp-l2-dp} shows the minimizer $\rho(t,x)$ of the tropical Wasserstein-2 distance.

\paragraph{Experiment 2.}
Similar to Experiment 1, we consider a two dimensional problem on $\Omega =  [0,1]\times[0,1]$. The initial densities $\rho_0$ and $\rho_1$ are same sizes of squares centered at $(\frac{1}{3},\frac{2}{3})$ and $(\frac{2}{3},\frac{1}{3})$ respectively. The same parameters are set as in Experiment 1. Together with Experiment 1, Experiment 2 shows that the minimizers of tropical optimal transport show different geodesics depending on the positions of initial densities. See Figure \ref{fig:exp-l1-dn} for $L^1$ result and Figure \ref{fig:exp-l2-dn} for $L^2$ result.

\paragraph{Experiment 3.}
We again consider a two dimensional problem on $\Omega =  [0,1]\times[0,1]$. The initial density $\rho_0$ at time 0 is a square centered at $(0.5,0.5)$ with width $0.2$. The initial density $\rho_1$ at time 1 is four squares of the same size centered at $(0.2,0.2)$, $(0.2,0.8)$, $(0.8,0.2)$ and $(0.8,0.8)$ with width $0.1$. The same parameters are set as in Experiment 1. See Figure \ref{fig:exp-l1-c4} for the $L^1$ result and Figure \ref{fig:exp-l2-c4} for the $L^2$ result; notice that the geodesics of minimizers from both results depend on the direction in which the densities travel. We see that Experiment 3 coincides with Experiments 1 and 2.

\paragraph{Software.}

Software to implement the numerical experiments presented in this paper is publicly available and located on the TropicalOT GitHub repository at \url{https://github.com/antheamonod/TropicalOT}.

%%%%%%%%%%%%%%%%%%%%%%%%%%%%%%%%%%%%%%%%%%%%%%%%%%%

\section{Discussion}
\label{sec:discussion}

In this paper, we connected optimal transport theory---specifically, dynamic optimal transport---with tropical geometry. In particular, we explicitly formulated geodesics for the tropical Wasserstein-$p$ distances over the tropical projective torus.  The tropical projective torus is the ambient space of the polyhedral Gr\"{o}bner complex of a homogeneous ideal in a polynomial ring $K[x_0, x_1, \ldots, x_n]$ over a field $K$---a foundational object in tropical geometry.  It is also the ambient space of the space of phylogenetic trees. 
%When $p=2$, we also derive the metric tensor for tropical Wassserstein-2 geometry. 
%The tropical heat flow is derived as the gradient flow of relative entropy.

We constructed and implemented primal-dual algorithms to compute tropical Wasserstein-1 and 2 geodesics on the tropical projective torus.  These results provide a framework to identifying all infinitely-many geodesic paths between points in this space, which leads to a better understanding of paths on the ambient space containing important structures in tropical geometry theory as well as in practice and applications.  In addition, the Wasserstein-2 distance possesses an important structure for statistical inference, since it provides the form for Fr\'{e}chet means on the tropical projective torus, as well as a general inner product structure.

Our research lays the foundation for further connections between optimal transport and tropical geometry.  Our work provides powerful tools to study important aspects such as geometry and statistics on the tropical projective torus.  A current work in progress is to characterize and solve the optimal transport problem on the subset of the tropical projective torus corresponding to phylogenetic tree space with 5 leaves, $\mathcal{T}_5$.  This space is made up of a union of 5!! = 15 polyhedral cones in the tropical projective torus, each with dimension 2.  In this study, the main challenge involves the polyhedral structure of the tree space (as discussed in Section \ref{subsec:trop_proj}), and in particular, how to handle the intersections of the cones; a weaker form of the divergence and gradient operators are required to traverse the cones.  The present work solves the problem within a single cone, which defines a shrink operator with already six cases, see Table \ref{tab:shrink_ex2}; we also expect the characterization of the shrink operator to be combinatorially more complicated on all 15 cones of $\mathcal{T}_5$.

From the perspective of optimal transport, we observe that the combinatorial structure of the tropical metric poses several interesting challenges in optimal transport.  For example, the partial differential equations derived in Section \ref{sec:opt} are defined in a piecewise manner: in two-dimensional sample space, there are six corresponding equations characterizing geodesics in optimal transport.  In the general case, there are interesting regularity issues to be further studied.  The theory of optimal transport and the study of associated density manifolds provide a natural base to construct heat equations with respect to the tropical metric.  This provides an important potential to defining non-uniform probability distributions on the tropical projective torus: classically, the solution to the heat equation gives rise to the Gaussian distribution, thus, a solution to the tropical heat equation is a candidate for a tropical Gaussian distribution on the tropical projective torus \citep{tran2018tropical, maazouz2019statistics}.  The dynamic setting of optimal transport with the tropical ground metric introduced in this paper also provides a foundation to studying the displacement convexity and Ricci curvature tensor on the tropical projective torus.  In forthcoming work, we further study such questions by applying the relevant work of \cite{Li2018_geometrya, Li2019_diffusion}, which also studies geometric and probabilistic questions in the context of optimal transport theory.  

\vfill\eject

%%%%%%%%%%%%%%%%%%%%%%%%%%%%%%%%%%%%%%%%%%%%%%%%%%%

\section*{Acknowledgments}

The authors wish to thank Marzieh Eidi, Th\'eo Lacombe, Victor Panaretos, Ronen Talmon, and Yoav Zemel for helpful discussions, with special thanks extended to Emil Saucan.  A.M.~wishes to acknowledge the Max Planck Institute for Mathematics in the Sciences for hosting her visit in Leipzig in July 2018, which inspired this work.  

%%%%%%%%%%%%%%%%%%%%%%%%%%%%%%%%%%%%%%%%%%%%%%%%%%%

%\appendix
%\renewcommand{\thesection}{Appendix}
%\renewcommand{\thesubsection}{A\arabic{subsection}}
%\renewcommand{\theAppDefinition}{A\arabic{AppDefinition}}
%\renewcommand{\theAppClaim}{A\arabic{AppClaim}}
%\section*{Appendix}

%%%%%%%%%%%%%%%%%%%%%%%%%%%%%%%%%%%%%%%%%%%%%%%%%%%

\newpage
\section*{Figures: Numerical Experiments}

\begin{figure}[ht]
\begin{minipage}{0.32\linewidth}
\includegraphics[width=1\linewidth]{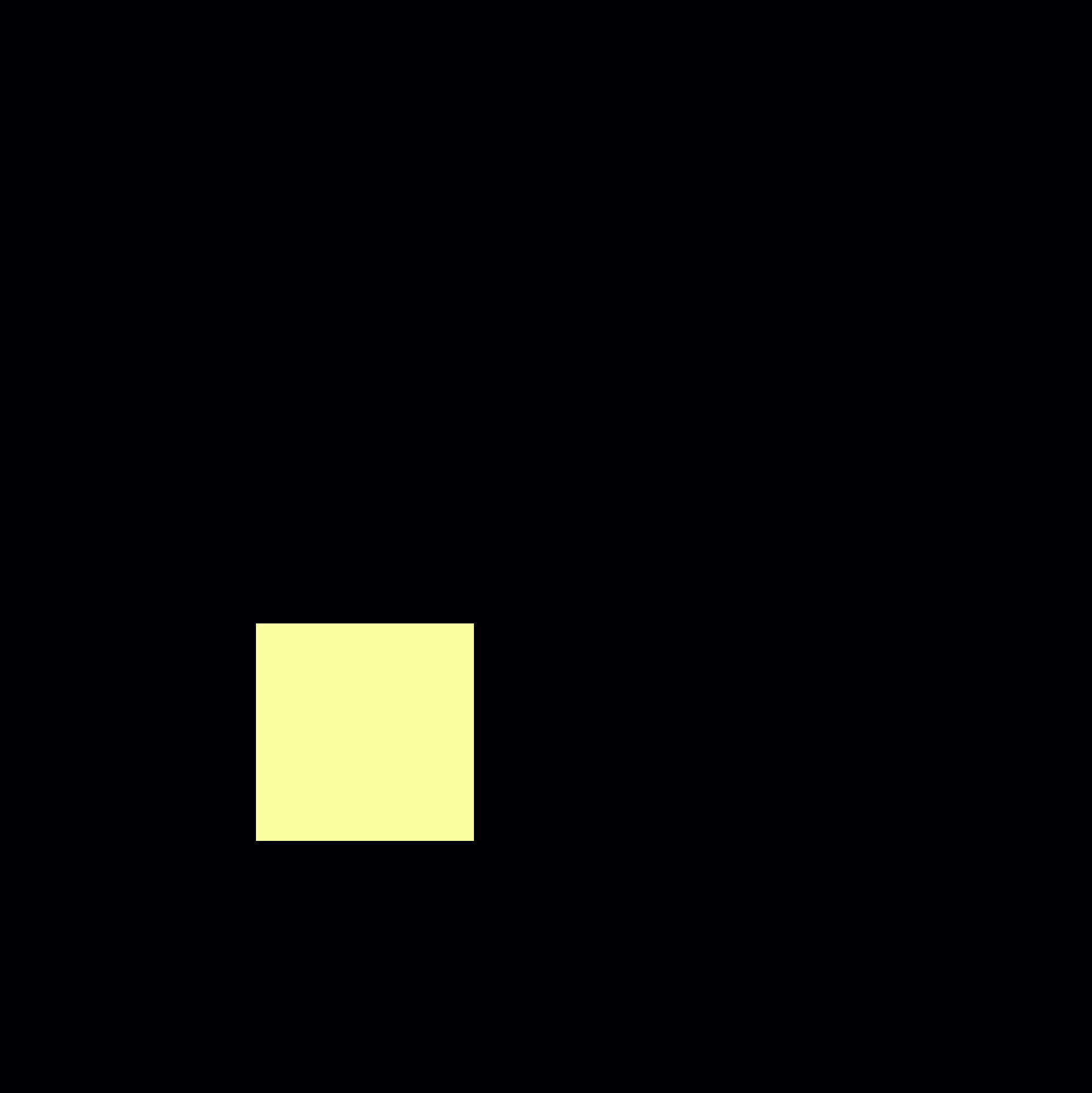}
\caption*{(a) $\rho_0$}
\end{minipage}\hfill
\begin{minipage}{0.32\linewidth}
\includegraphics[width=1\linewidth]{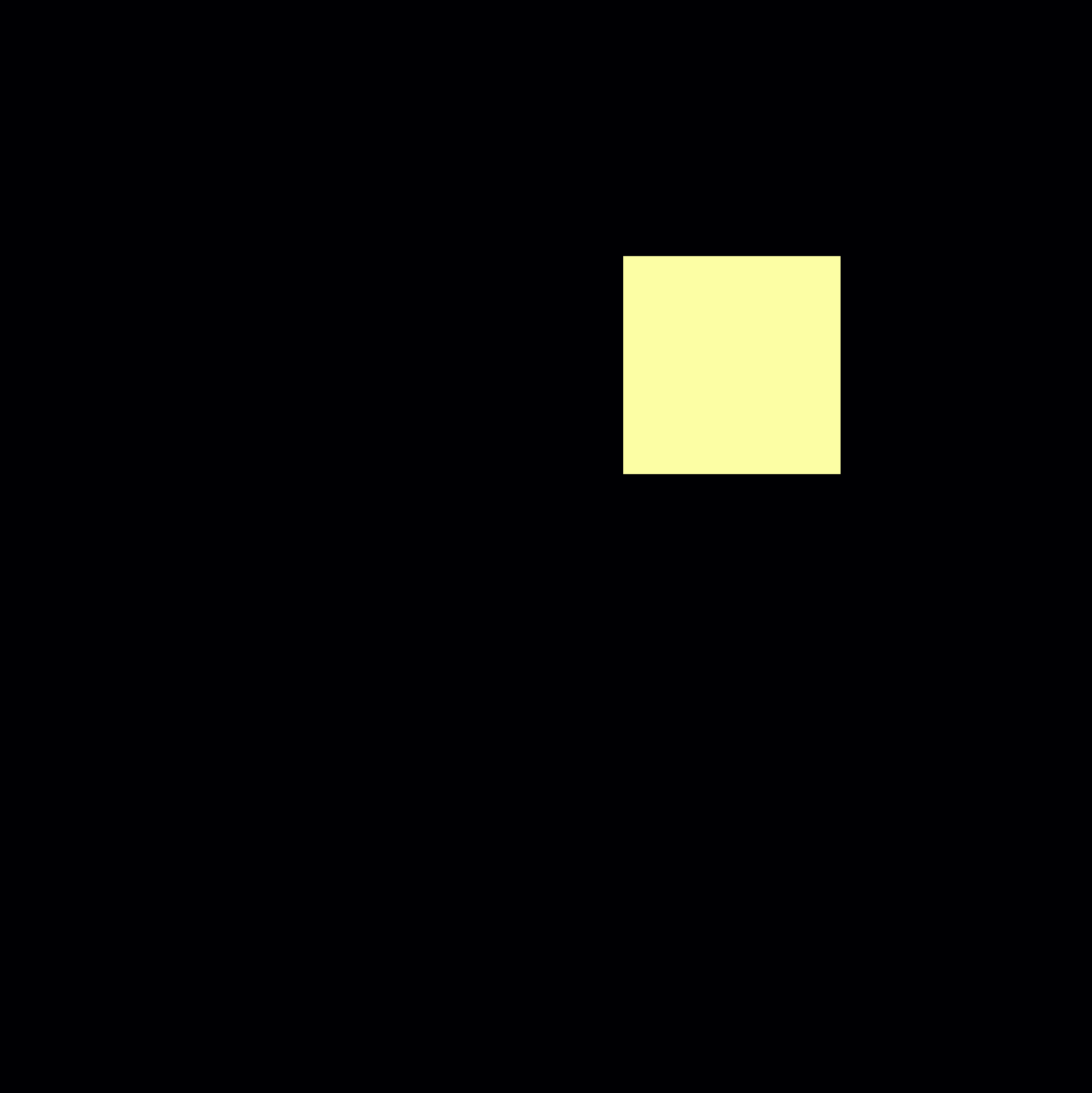}
\caption*{(b) $\rho_1$}
\end{minipage}\hfill
\begin{minipage}{0.32\linewidth}
\includegraphics[width=1\linewidth]{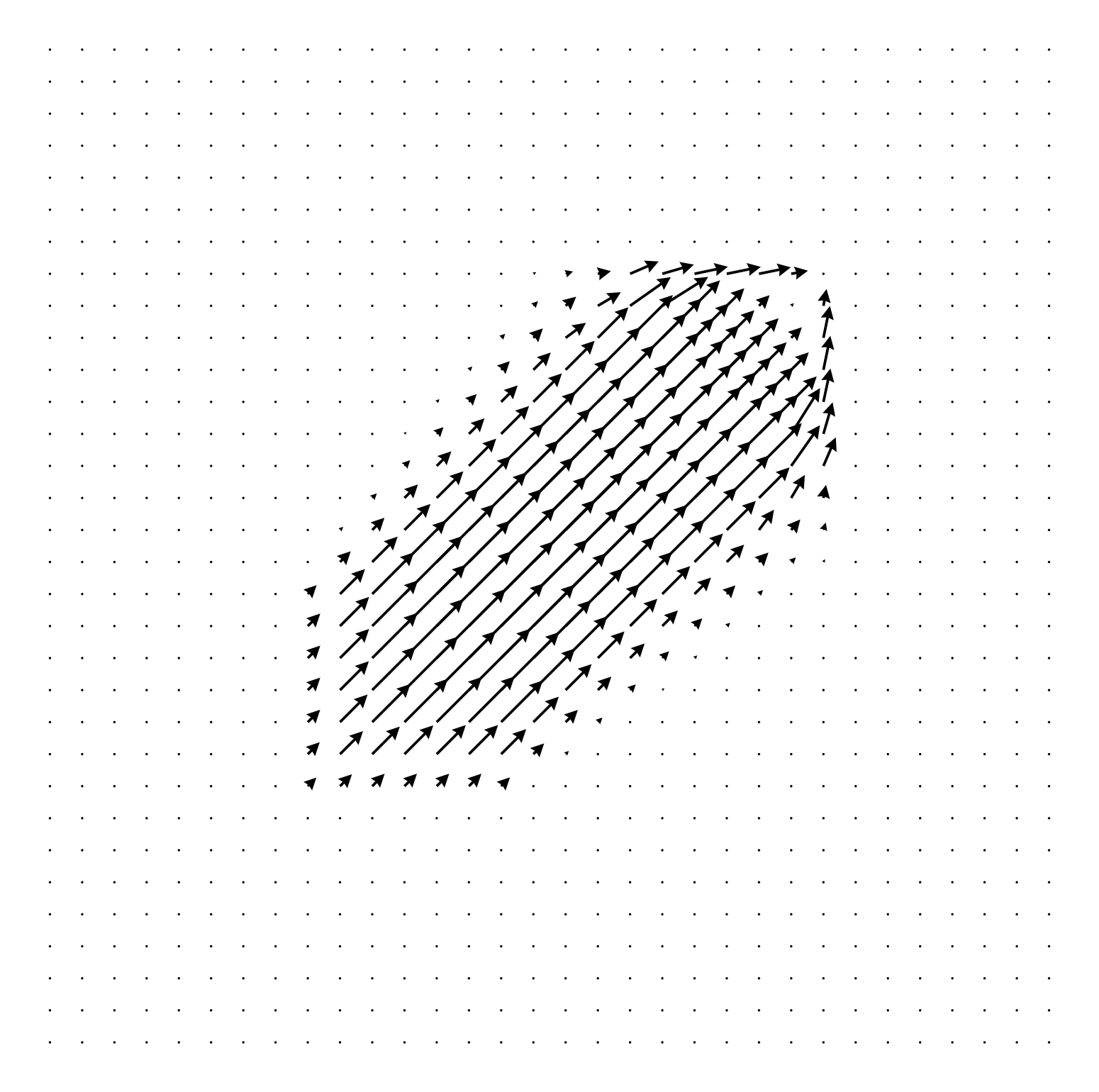}
\caption*{(c) $\mathbf{m}$}
\end{minipage}\hfill
\caption{Experiment 1: $L^1$ tropical optimal transport. (a) and (b) show the initial densities $\rho_0$ and $\rho_1$, while (c) shows the geodesics of the $L^1$ tropical optimal transport between $\rho_0$ and $\rho_1$. }
\label{fig:exp-l1-dp}
\end{figure}

\begin{figure}[ht]
\begin{minipage}{0.32\linewidth}
\includegraphics[width=1\linewidth]{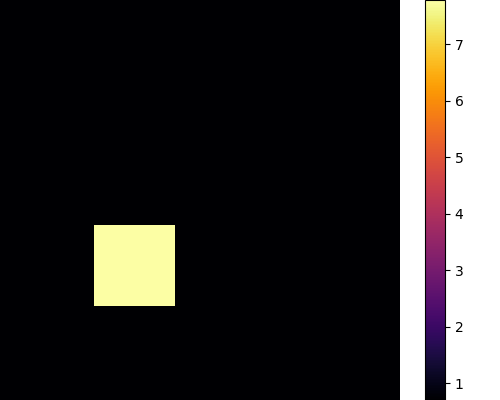}
\vspace*{-0.6cm}
\caption*{$t=0$}
\end{minipage}\hfill
\begin{minipage}{0.32\linewidth}
\includegraphics[width=1\linewidth]{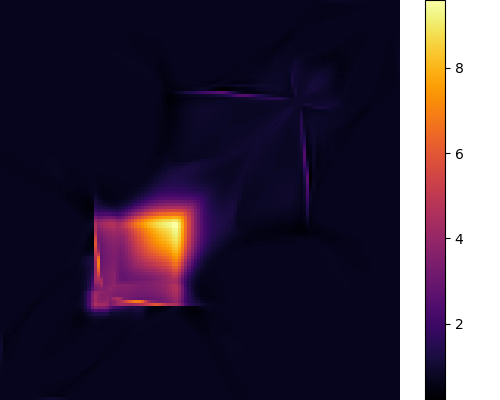}
\vspace*{-0.6cm}
\caption*{$t=0.21$}
\end{minipage}\hfill
\begin{minipage}{0.32\linewidth}
\includegraphics[width=1\linewidth]{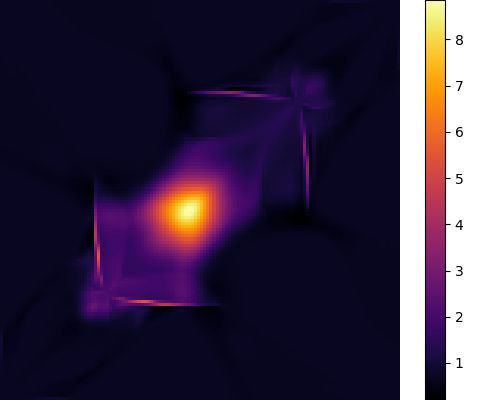}
\vspace*{-0.6cm}
\caption*{$t=0.42$}
\end{minipage}\hfill
\begin{minipage}{0.32\linewidth}
\includegraphics[width=1\linewidth]{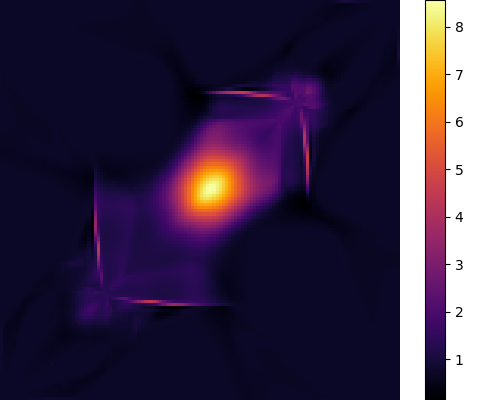}
\vspace*{-0.6cm}
\caption*{$t=0.64$}
\end{minipage}\hfill
\begin{minipage}{0.32\linewidth}
\includegraphics[width=1\linewidth]{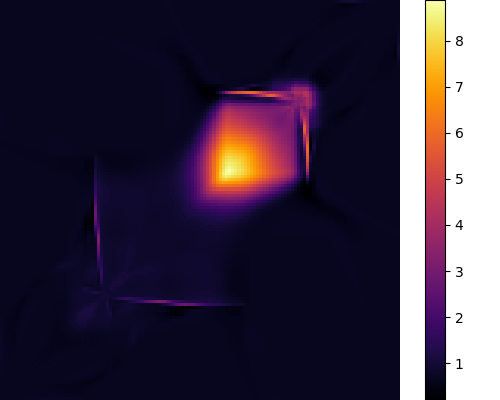}
\vspace*{-0.6cm}
\caption*{$t=0.86$}
\end{minipage}\hfill
\begin{minipage}{0.32\linewidth}
\includegraphics[width=1\linewidth]{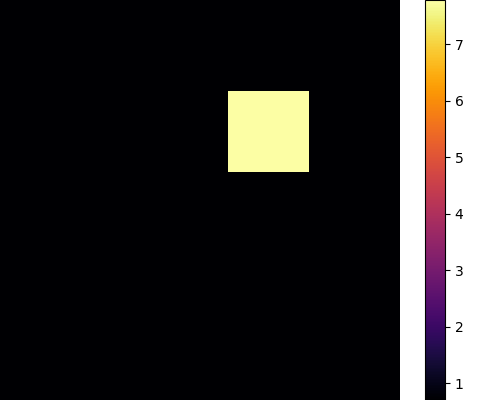}
\vspace*{-0.6cm}
\caption*{$t=1$}
\end{minipage}\hfill
\caption{Experiment 1: $L^2$ tropical optimal transport. The six figures show the geodesics of $L^2$ tropical optimal transport from $t=0$ to $t=1$. The initial densities are same as in Figure \ref{fig:exp-l1-dp}.}
\label{fig:exp-l2-dp}
\end{figure}

\begin{figure}[ht]
\begin{minipage}{0.32\linewidth}
\includegraphics[width=1\linewidth]{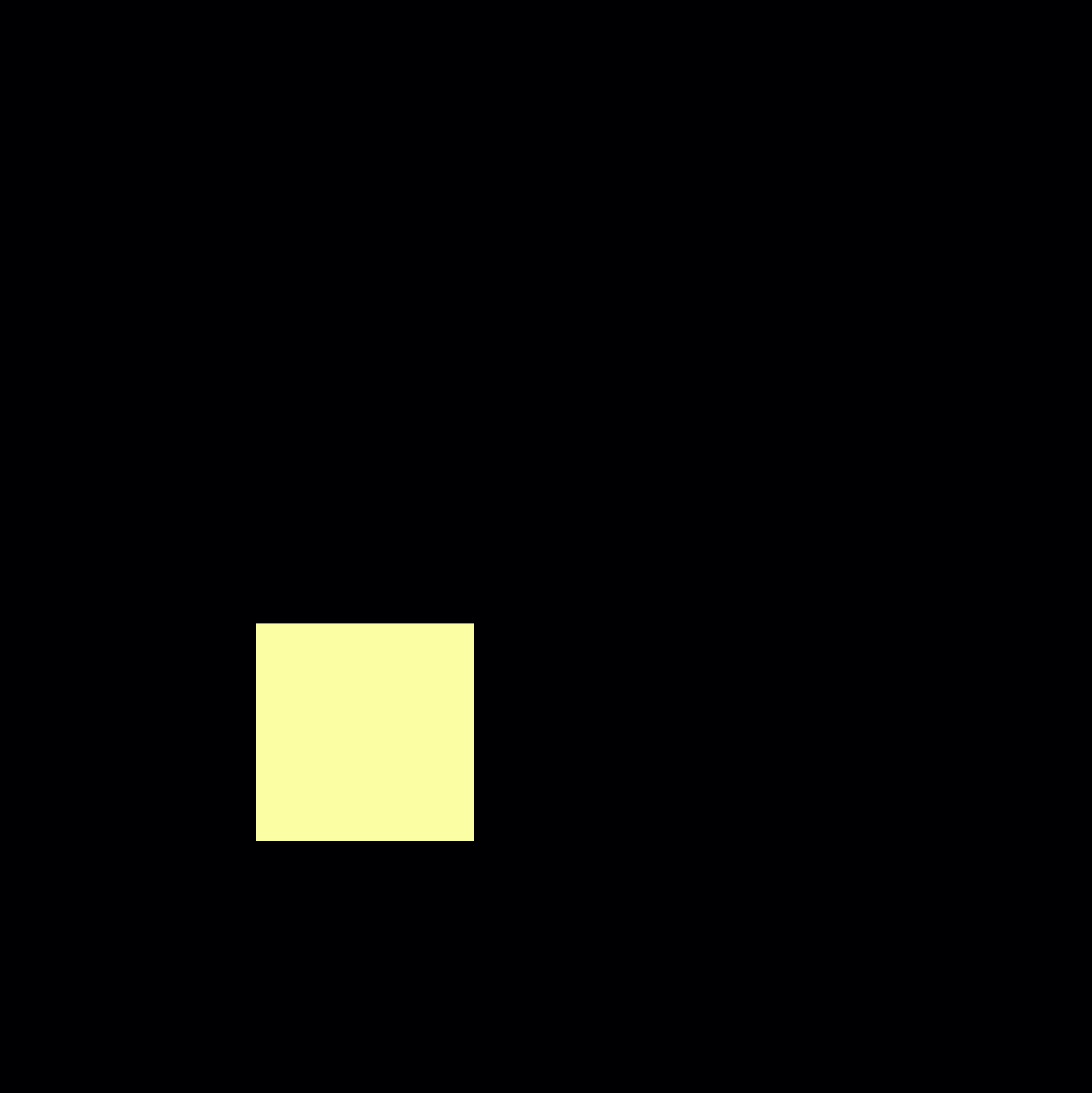}
\caption*{(a) $\rho_0$}
\end{minipage}\hfill
\begin{minipage}{0.32\linewidth}
\includegraphics[width=1\linewidth]{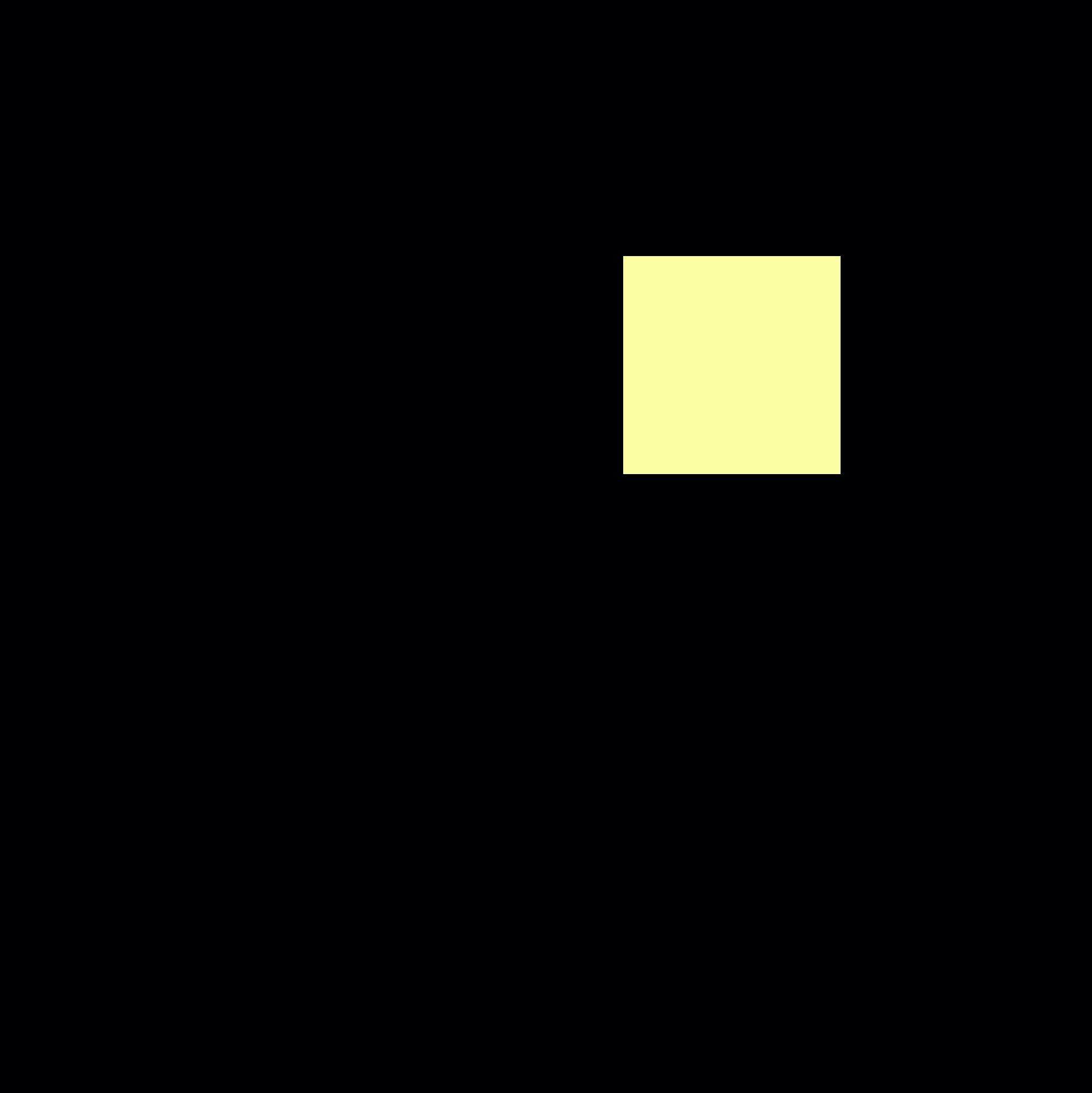}
\caption*{(b) $\rho_1$}
\end{minipage}\hfill
\begin{minipage}{0.32\linewidth}
\includegraphics[width=1\linewidth]{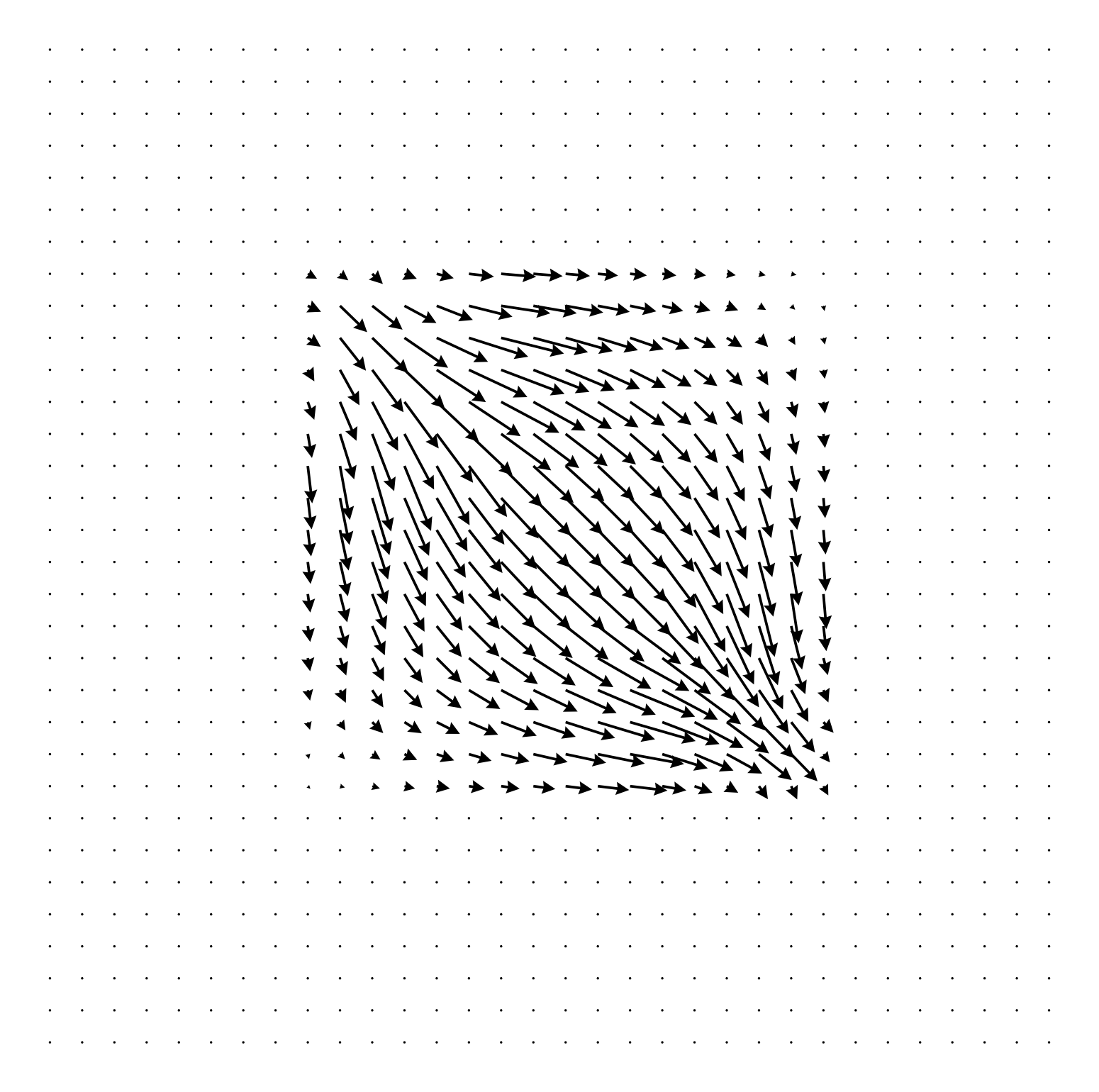}
\caption*{(c) $\mathbf{m}$}
\end{minipage}\hfill
\caption{Experiment 2: $L^1$ tropical optimal transportation. (a) and (b) show the initial densities $\rho_0$ and $\rho_1$. (c) shows the geodesics of the $L^1$ tropical optimal transportation between $\rho_0$ and $\rho_1$. }
\label{fig:exp-l1-dn}
\end{figure}

\begin{figure}[ht]
\begin{minipage}{0.32\linewidth}
\includegraphics[width=1\linewidth]{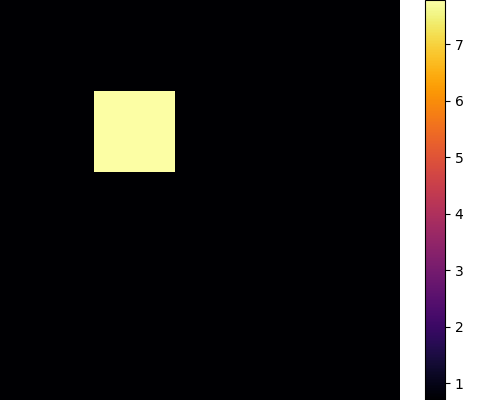}
\vspace*{-0.6cm}
\caption*{$t=0$}
\end{minipage}\hfill
\begin{minipage}{0.32\linewidth}
\includegraphics[width=1\linewidth]{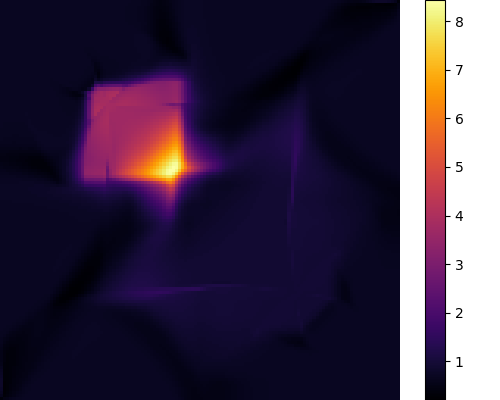}
\vspace*{-0.6cm}
\caption*{$t=0.21$}
\end{minipage}\hfill
\begin{minipage}{0.32\linewidth}
\includegraphics[width=1\linewidth]{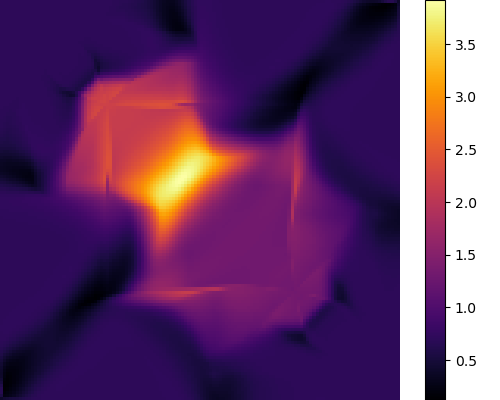}
\vspace*{-0.6cm}
\caption*{$t=0.42$}
\end{minipage}\hfill
\begin{minipage}{0.32\linewidth}
\includegraphics[width=1\linewidth]{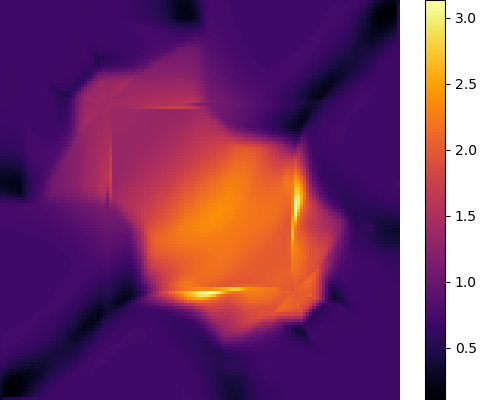}
\vspace*{-0.6cm}
\caption*{$t=0.64$}
\end{minipage}\hfill
\begin{minipage}{0.32\linewidth}
\includegraphics[width=1\linewidth]{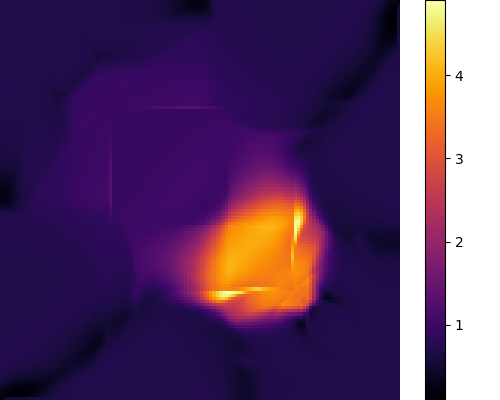}
\vspace*{-0.6cm}
\caption*{$t=0.86$}
\end{minipage}\hfill
\begin{minipage}{0.32\linewidth}
\includegraphics[width=1\linewidth]{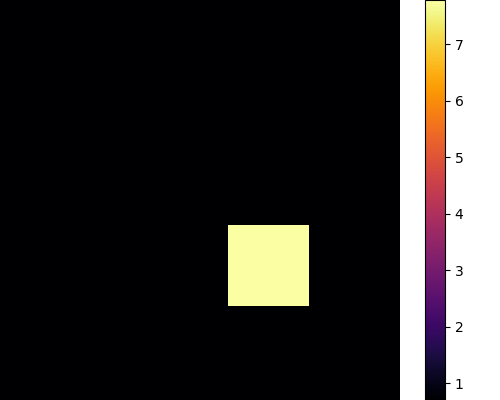}
\vspace*{-0.6cm}
\caption*{$t=1$}
\end{minipage}\hfill
\caption{Experiment 2: $L^2$ tropical optimal transport. The figures show the geodesics of $L^2$ tropical optimal transport between two initial densities from $t=0$ to $t=1$. The initial densities are same as in Figure \ref{fig:exp-l1-dn}.}
\label{fig:exp-l2-dn}
\end{figure}

\begin{figure}[ht]
\begin{minipage}{0.32\linewidth}
\includegraphics[width=1\linewidth]{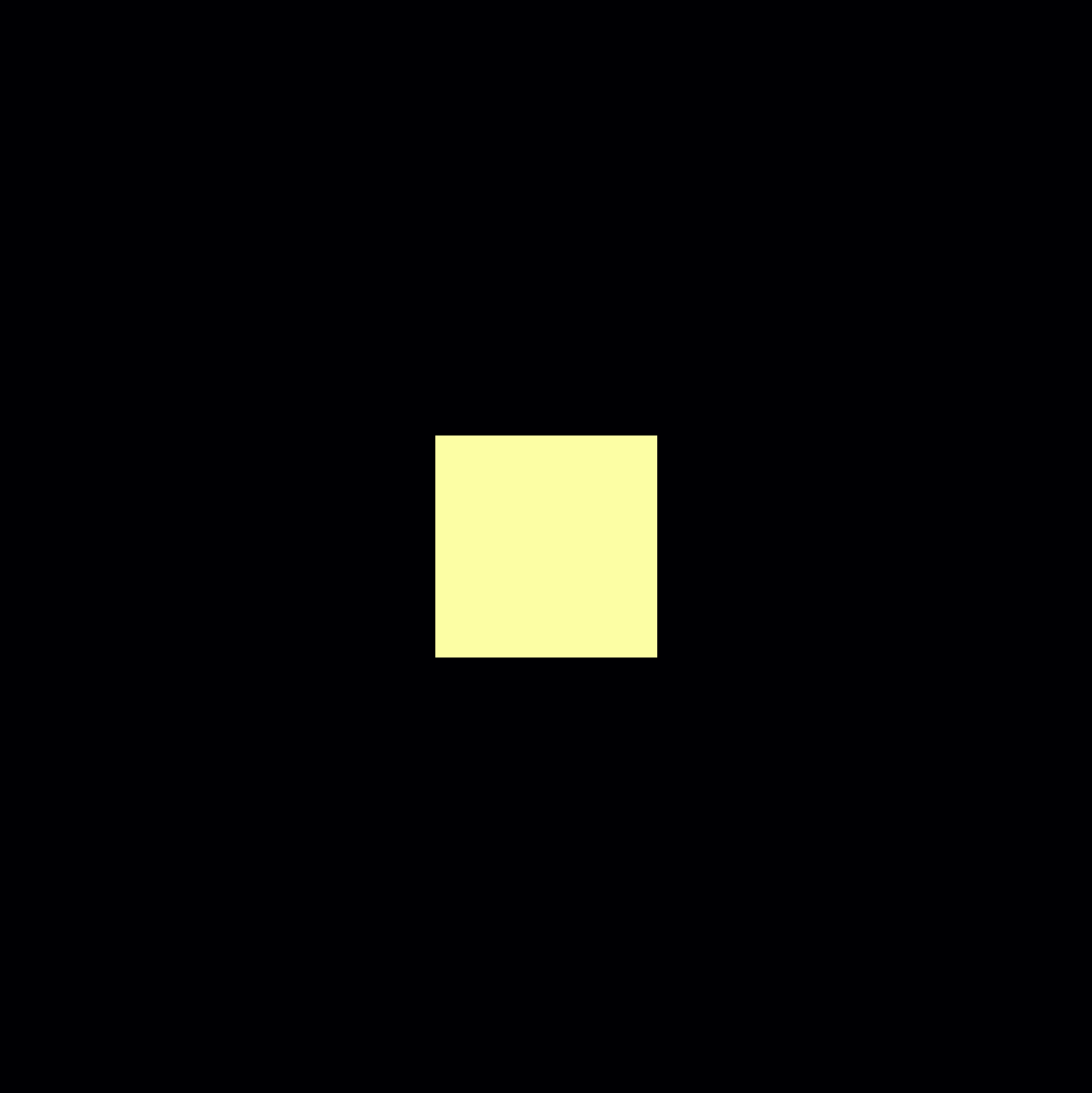}
\caption*{(a) $\rho_0$}
\end{minipage}\hfill
\begin{minipage}{0.32\linewidth}
\includegraphics[width=1\linewidth]{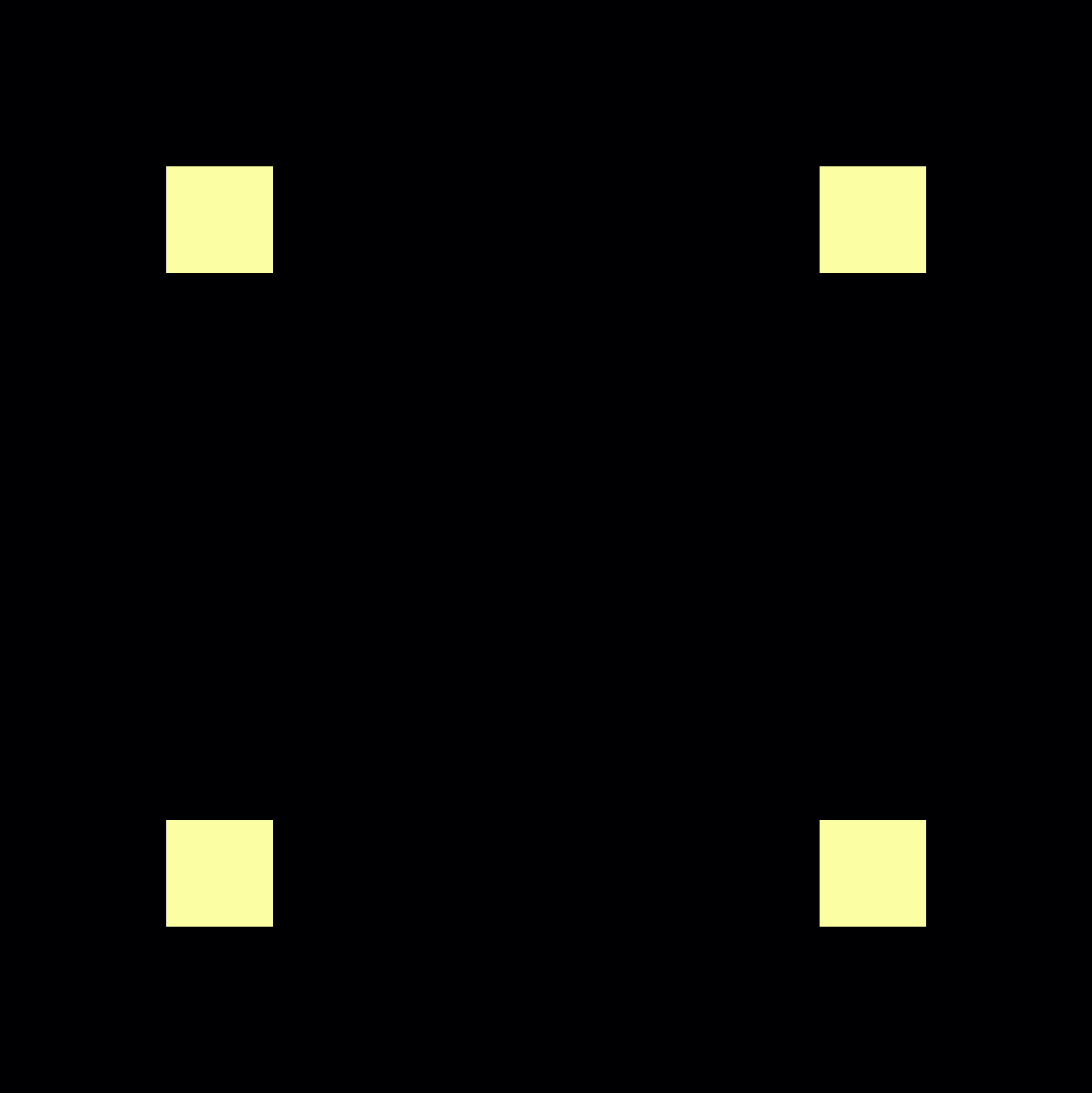}
\caption*{(b) $\rho_1$}
\end{minipage}\hfill
\begin{minipage}{0.32\linewidth}
\includegraphics[width=1\linewidth]{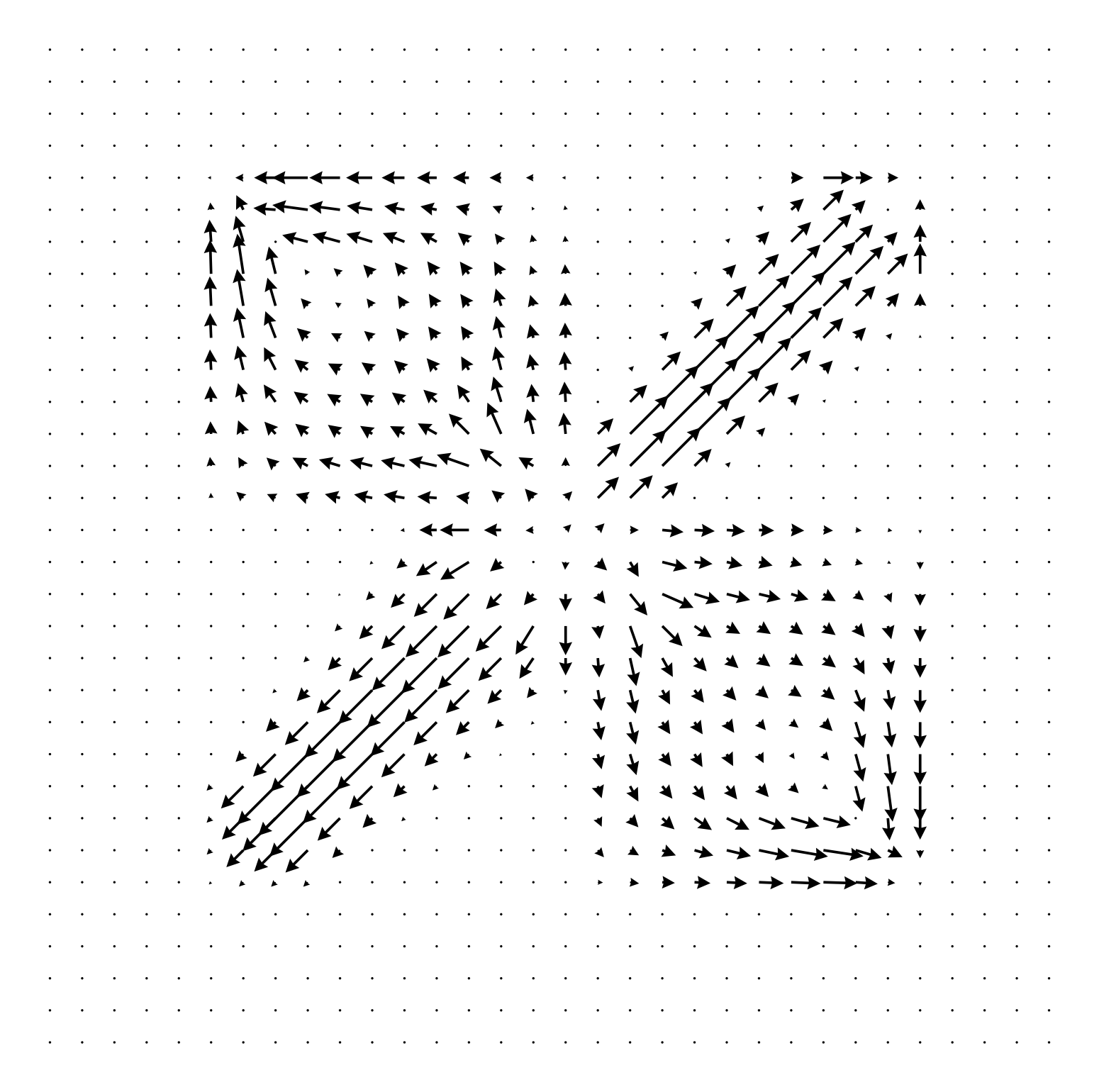}
\caption*{(c) $\mathbf{m}$}
\end{minipage}\hfill
\caption{Experiment 3: $L^1$ tropical optimal transport. (a) and (b) show the initial densities $\rho_0$ and $\rho_1$, while (c) shows the geodesics of the $L^1$ tropical optimal transport between $\rho_0$ and $\rho_1$. This experiment shows similar patterns of geodesics from Experiment 1 and Experiment 2.}
\label{fig:exp-l1-c4}
\end{figure}

\begin{figure}[ht]
\begin{minipage}{0.32\linewidth}
\includegraphics[width=1\linewidth]{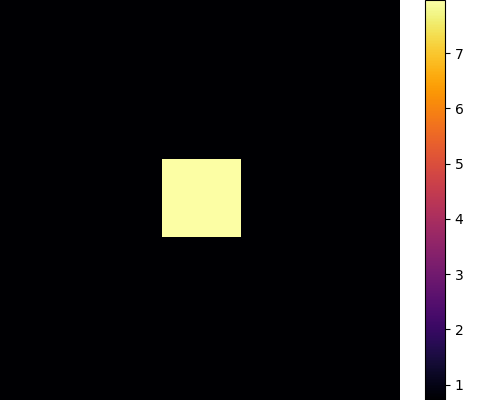}
\vspace*{-0.6cm}
\caption*{$t=0$}
\end{minipage}\hfill
\begin{minipage}{0.32\linewidth}
\includegraphics[width=1\linewidth]{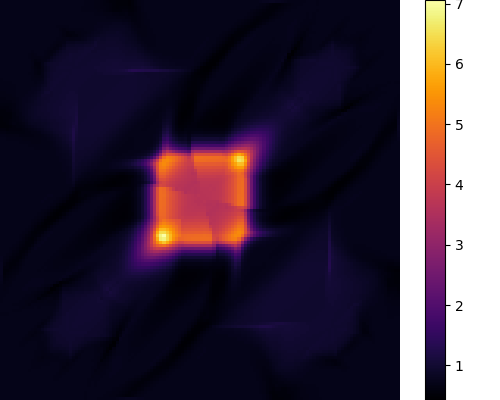}
\vspace*{-0.6cm}
\caption*{$t=0.21$}
\end{minipage}\hfill
\begin{minipage}{0.32\linewidth}
\includegraphics[width=1\linewidth]{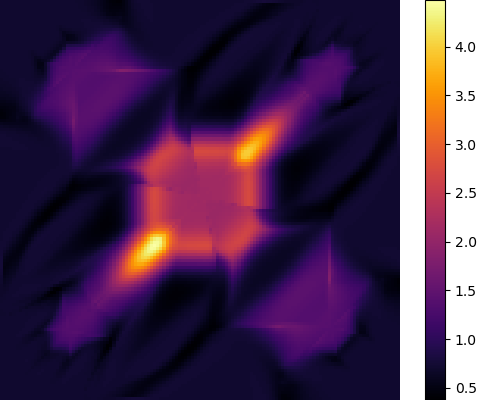}
\vspace*{-0.6cm}
\caption*{$t=0.42$}
\end{minipage}\hfill
\begin{minipage}{0.32\linewidth}
\includegraphics[width=1\linewidth]{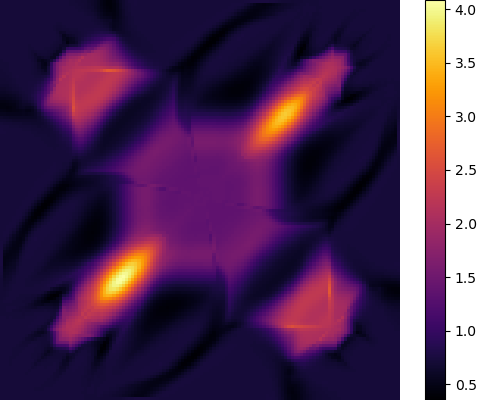}
\vspace*{-0.6cm}
\caption*{$t=0.64$}
\end{minipage}\hfill
\begin{minipage}{0.32\linewidth}
\includegraphics[width=1\linewidth]{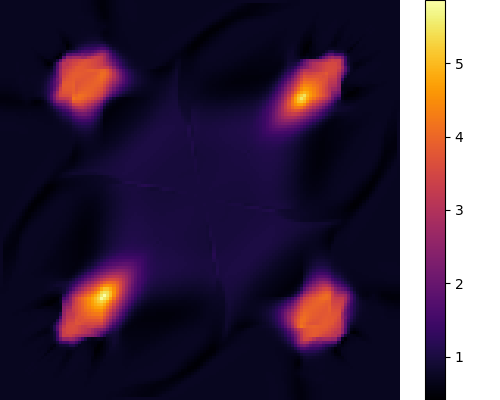}
\vspace*{-0.6cm}
\caption*{$t=0.86$}
\end{minipage}\hfill
\begin{minipage}{0.32\linewidth}
\includegraphics[width=1\linewidth]{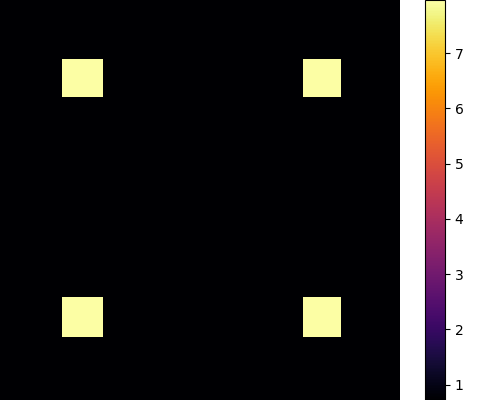}
\vspace*{-0.6cm}
\caption*{$t=1$}
\end{minipage}\hfill
\caption{Experiment 3: $L^2$ tropical optimal transport. The six figures show the geodesics of $L^2$ tropical optimal transportation from $t=0$ to $t=1$. The initial densities are same as in Figure \ref{fig:exp-l1-c4}.}
\label{fig:exp-l2-c4}
\end{figure}

%%%%%%%%%%%%%%%%%%%%%%%%%%%%%%%%%%%%%%%%%%%%%%%%%%%
%%%%%%%%%%%%%%%%%%%%%%%%%%%%%%%%%%%%%%%%%%%%%%%%%%%
%%%%%%%%%%%%%%%%%%%%%%%%%%%%%%%%%%%%%%%%%%%%%%%%%%%

% \section{Referencing}

% The bibliography file (a standard {\em .bib} formatted file) is % pulled in at the end, and articles, papers, books,
% web sites etc listed there are cited in various ways, such as \cite{carvalho2008bfrm} and \cite{aguilar2000jbes},
%as well as \citep{carvalho2008bfrm} and \citep{aguilar2000jbes}, and %also
%\citep{carvalho2008bfrm,aguilar2000jbes,west2003bayes7}. You can also use (\citet{aguilar2000jbes}) or
%(\citet{carvalho2008bfrm,aguilar2000jbes,west2003bayes7})  at a more hands-on level.

\clearpage
\newpage
\bibliography{Tropical_OT_ref} % edit bibexample.bib file ...

\begin{thebibliography}{}

\bibitem[\protect\citeauthoryear{Akian, Gaubert, Ni\c{t}ic\u{a}, and
  Singer}{Akian et~al.}{2011}]{AKIAN20113261}
Akian, M., S.~Gaubert, V.~Ni\c{t}ic\u{a}, and I.~Singer (2011).
\newblock Best approximation in max-plus semimodules.
\newblock {\em Linear Algebra and its Applications\/}~{\em 435\/}(12),
  3261--3296.

\bibitem[\protect\citeauthoryear{Ambrosio and Gigli}{Ambrosio and
  Gigli}{2013}]{ambrosio2013user}
Ambrosio, L. and N.~Gigli (2013).
\newblock {\em A User's Guide to Optimal Transport}, pp.\  1--155.
\newblock Berlin, Heidelberg: Springer Berlin Heidelberg.

\bibitem[\protect\citeauthoryear{Ambrosio, Gigli, and Savar{\'e}}{Ambrosio
  et~al.}{2008}]{ambrosio2008gradient}
Ambrosio, L., N.~Gigli, and G.~Savar{\'e} (2008).
\newblock {\em Gradient {F}lows: {I}n {M}etric {S}paces and in the {S}pace of
  {P}robability {M}easures}.
\newblock Springer Science \& Business Media.

\bibitem[\protect\citeauthoryear{Benamou and Brenier}{Benamou and
  Brenier}{2000}]{BB}
Benamou, J.-D. and Y.~Brenier (2000, January).
\newblock A computational fluid mechanics solution to the
  {M}onge--{K}antorovich mass transfer problem.
\newblock {\em Numerische Mathematik\/}~{\em 84\/}(3), 375--393.

\bibitem[\protect\citeauthoryear{{Benamou, Jean-David}, {Carlier, Guillaume},
  and {Hatchi, Rom\'eo}}{{Benamou, Jean-David}
  et~al.}{2018}]{benamou2016numerical}
{Benamou, Jean-David}, {Carlier, Guillaume}, and {Hatchi, Rom\'eo} (2018).
\newblock A numerical solution to {M}onge's problem with a {F}insler distance
  as cost.
\newblock {\em ESAIM: M2AN\/}~{\em 52\/}(6), 2133--2148.

\bibitem[\protect\citeauthoryear{Brunn}{Brunn}{1887}]{brunn1887ueber}
Brunn, H. (1887).
\newblock {\em Ueber Ovale und Eifl{\"a}chen, Inaugural}.
\newblock Ph.\ D. thesis, Dissertation, Munich, F. Straub.

\bibitem[\protect\citeauthoryear{{\c{C}}elik, Jamneshan, Mont{\'u}far,
  Sturmfels, and Venturello}{{\c{C}}elik et~al.}{2019}]{otvariety}
{\c{C}}elik, T.~{\"O}., A.~Jamneshan, G.~Mont{\'u}far, B.~Sturmfels, and
  L.~Venturello (2019).
\newblock Optimal {T}ransport to a {V}ariety.
\newblock In {\em International Conference on Mathematical Aspects of Computer
  and Information Sciences}, pp.\  364--381. Springer.

\bibitem[\protect\citeauthoryear{Chambolle and Pock}{Chambolle and
  Pock}{2011}]{pock1}
Chambolle, A. and T.~Pock (2011).
\newblock A {{First}}-{{Order Primal}}-{{Dual Algorithm}} for {{Convex
  Problems}} with {{Applications}} to {{Imaging}}.
\newblock {\em Journal of Mathematical Imaging and Vision\/}~{\em 40\/}(1),
  120--145.

\bibitem[\protect\citeauthoryear{Cohen, Gaubert, and Quadrat}{Cohen
  et~al.}{2004}]{COHEN2004395}
Cohen, G., S.~Gaubert, and J.-P. Quadrat (2004).
\newblock Duality and {S}eparation {T}heorems in {I}dempotent {S}emimodules.
\newblock {\em Linear Algebra and its Applications\/}~{\em 379}, 395--422.
\newblock Special Issue on the Tenth ILAS Conference (Auburn, 2002).

\bibitem[\protect\citeauthoryear{Divol and Lacombe}{Divol and
  Lacombe}{2019}]{divol2019understanding}
Divol, V. and T.~Lacombe (2019).
\newblock Understanding the topology and the geometry of the persistence
  diagram space via optimal partial transport.
\newblock {\em arXiv:1901.03048\/}.

\bibitem[\protect\citeauthoryear{El~Maazouz and Tran}{El~Maazouz and
  Tran}{2019}]{maazouz2019statistics}
El~Maazouz, Y. and N.~M. Tran (2019).
\newblock Statistics of {G}aussians on local fields and their tropicalizations.
\newblock {\em arXiv:1909.00559\/}.

\bibitem[\protect\citeauthoryear{Evans and Matsen}{Evans and
  Matsen}{2012}]{doi:10.1111/j.1467-9868.2011.01018.x}
Evans, S.~N. and F.~A. Matsen (2012).
\newblock The phylogenetic {K}antorovich\textemdash{R}ubinstein metric for
  environmental sequence samples.
\newblock {\em Journal of the Royal Statistical Society: Series B (Statistical
  Methodology)\/}~{\em 74\/}(3), 569--592.

\bibitem[\protect\citeauthoryear{Garba, Nye, Lueg, and Huckemann}{Garba
  et~al.}{2020}]{garba2020information}
Garba, M.~K., T.~M. Nye, J.~Lueg, and S.~F. Huckemann (2020).
\newblock Information geometry for phylogenetic trees.
\newblock {\em arXiv:2003.13004\/}.

\bibitem[\protect\citeauthoryear{Garba, Nye, and Boys}{Garba
  et~al.}{2017}]{10.1093/sysbio/syx080}
Garba, M.~K., T.~M.~W. Nye, and R.~J. Boys (2017, 10).
\newblock {Probabilistic {D}istances {B}etween {T}rees}.
\newblock {\em Systematic Biology\/}~{\em 67\/}(2), 320--327.

\bibitem[\protect\citeauthoryear{Jacobs, L\'{e}ger, Li, and Osher}{Jacobs
  et~al.}{2019}]{JacobsLegerLiOsher2018_solvinga}
Jacobs, M., F.~L\'{e}ger, W.~Li, and S.~Osher (2019).
\newblock Solving {L}arge-{S}cale {O}ptimization {P}roblems with a
  {C}onvergence {R}ate {I}ndependent of {G}rid {S}ize.
\newblock {\em SIAM Journal on Numerical Analysis\/}~{\em 57\/}(3), 1100--1123.

\bibitem[\protect\citeauthoryear{Kantorovich}{Kantorovich}{2006}]{kantorovich1942translocation}
Kantorovich, L.~V. (2006).
\newblock On the translocation of masses.
\newblock {\em Journal of Mathematical Sciences\/}~{\em 133\/}(4), 1381--1382.

\bibitem[\protect\citeauthoryear{Kloeckner}{Kloeckner}{2015}]{kloeckner_2015}
Kloeckner, B.~R. (2015).
\newblock A geometric study of {W}asserstein spaces: {U}ltrametrics.
\newblock {\em Mathematika\/}~{\em 61\/}(1), 162--178.

\bibitem[\protect\citeauthoryear{Kullback and Leibler}{Kullback and
  Leibler}{1951}]{Kullback:1951aa}
Kullback, S. and R.~A. Leibler (1951).
\newblock On information and sufficiency.
\newblock {\em The Annals of Mathematical Statistics\/}~{\em 22\/}(1), 79--86.

\bibitem[\protect\citeauthoryear{Lacombe, Cuturi, and Oudot}{Lacombe
  et~al.}{2018}]{lacombe2018large}
Lacombe, T., M.~Cuturi, and S.~Oudot (2018).
\newblock Large scale computation of means and clusters for persistence
  diagrams using optimal transport.
\newblock In {\em Advances in Neural Information Processing Systems}, pp.\
  9770--9780.

\bibitem[\protect\citeauthoryear{Lafferty}{Lafferty}{1988}]{Lafferty}
Lafferty, J.~D. (1988).
\newblock The {D}ensity {M}anifold and {C}onfiguration {S}pace {Q}uantization.
\newblock {\em Transactions of the American Mathematical Society\/}~{\em
  305\/}(2), 699--741.

\bibitem[\protect\citeauthoryear{Le, Yamada, Fukumizu, and Cuturi}{Le
  et~al.}{2019}]{NIPS2019_9396}
Le, T., M.~Yamada, K.~Fukumizu, and M.~Cuturi (2019).
\newblock Tree-{S}liced {V}ariants of {W}asserstein {D}istances.
\newblock In H.~Wallach, H.~Larochelle, A.~Beygelzimer, F.~d\textquotesingle
  Alch\'{e}-Buc, E.~Fox, and R.~Garnett (Eds.), {\em Advances in Neural
  Information Processing Systems 32}, pp.\  12304--12315. Curran Associates,
  Inc.

\bibitem[\protect\citeauthoryear{Li}{Li}{2018}]{Li2018_geometrya}
Li, W. (2018).
\newblock Transport {I}nformation {G}eometry {I}: {R}iemannian {C}alculus on
  {P}robability {S}implex.
\newblock {\em arXiv:1803.06360\/}.

\bibitem[\protect\citeauthoryear{Li}{Li}{2019}]{Li2019_diffusion}
Li, W. (2019).
\newblock Diffusion hypercontractivity via generalized density manifold.
\newblock {\em arXiv:1907.12546\/}.

\bibitem[\protect\citeauthoryear{Li, Ryu, Osher, Yin, and Gangbo}{Li
  et~al.}{2018}]{Li2018}
Li, W., E.~K. Ryu, S.~Osher, W.~Yin, and W.~Gangbo (2018).
\newblock A {P}arallel {M}ethod for {E}arth {M}over's {D}istance.
\newblock {\em Journal of Scientific Computing\/}~{\em 75\/}(1), 182--197.

\bibitem[\protect\citeauthoryear{Lin, Sturmfels, Tang, and Yoshida}{Lin
  et~al.}{2017}]{doi:10.1137/16M1079841}
Lin, B., B.~Sturmfels, X.~Tang, and R.~Yoshida (2017).
\newblock Convexity in {T}ree {S}paces.
\newblock {\em SIAM Journal on Discrete Mathematics\/}~{\em 31\/}(3),
  2015--2038.

\bibitem[\protect\citeauthoryear{Lin and Tran}{Lin and Tran}{2019}]{lin2017two}
Lin, B. and N.~M. Tran (2019).
\newblock Two-player incentive compatible outcome functions are affine
  maximizers.
\newblock {\em Linear Algebra and its Applications\/}~{\em 578}, 133--152.

\bibitem[\protect\citeauthoryear{Lin and Yoshida}{Lin and
  Yoshida}{2018}]{LinYoshida}
Lin, B. and R.~Yoshida (2018).
\newblock Tropical {F}ermat--{W}eber {P}oints.
\newblock {\em SIAM Journal on Discrete Mathematics\/}~{\em 32\/}(2),
  1229--1245.

\bibitem[\protect\citeauthoryear{Lyusternik}{Lyusternik}{1935}]{lyusternik1935brunn}
Lyusternik, L. (1935).
\newblock Die brunn-minkowskische ungleichung f{\"u}r beliebige messbare
  mengen, cr (dokl.) acad.
\newblock {\em Sci. URSS, n. Ser\/}~{\em 3}, 55--58.

\bibitem[\protect\citeauthoryear{Maclagan and Sturmfels}{Maclagan and
  Sturmfels}{2015}]{maclagan2015introduction}
Maclagan, D. and B.~Sturmfels (2015).
\newblock {\em Introduction to Tropical Geometry (Graduate Studies in
  Mathematics)}.
\newblock American Mathematical Society.

\bibitem[\protect\citeauthoryear{Minkowski}{Minkowski}{1896}]{minkowski1896geometrie}
Minkowski, H. (1896).
\newblock Geometrie der zahlen (2 vol.).
\newblock {\em Teubner, Leipzig\/}~{\em 1910\/}(1896).

\bibitem[\protect\citeauthoryear{Monge}{Monge}{1781}]{monge1781memoire}
Monge, G. (1781).
\newblock M{\'e}moire sur la th{\'e}orie des d{\'e}blais et des remblais.
\newblock {\em Histoire de l'Acad{\'e}mie royale des sciences de Paris\/}.

\bibitem[\protect\citeauthoryear{Monod, Lin, Kang, and Yoshida}{Monod
  et~al.}{2018}]{lin2018tropical}
Monod, A., B.~Lin, Q.~Kang, and R.~Yoshida (2018).
\newblock Tropical {G}eometry of {P}hylogenetic {T}ree {S}pace: {A}
  {S}tatistical {P}erspective.
\newblock {\em arXiv:1805.12400\/}.

\bibitem[\protect\citeauthoryear{Otto}{Otto}{2001}]{otto}
Otto, F. (2001).
\newblock The {G}eometry of {D}issipative {E}volution {E}quations: {T}he
  {P}orous {M}edium {E}quation.
\newblock {\em Communications in Partial Differential Equations\/}~{\em
  26\/}(1-2), 101--174.

\bibitem[\protect\citeauthoryear{Otto and Villani}{Otto and Villani}{2000}]{OV}
Otto, F. and C.~Villani (2000).
\newblock Generalization of an {I}nequality by {T}alagrand and {L}inks with the
  {L}ogarithmic {S}obolev {I}nequality.
\newblock {\em Journal of Functional Analysis\/}~{\em 173\/}(2), 361 -- 400.

\bibitem[\protect\citeauthoryear{Page, Yoshida, and Zhang}{Page
  et~al.}{2020}]{10.1093/bioinformatics/btaa564}
Page, R., R.~Yoshida, and L.~Zhang (2020, 06).
\newblock {Tropical principal component analysis on the space of phylogenetic
  trees}.
\newblock {\em Bioinformatics\/}.
\newblock btaa564.

\bibitem[\protect\citeauthoryear{Panaretos and Zemel}{Panaretos and
  Zemel}{2019}]{doi:10.1146/annurev-statistics-030718-104938}
Panaretos, V.~M. and Y.~Zemel (2019).
\newblock Statistical {A}spects of {W}asserstein {D}istances.
\newblock {\em Annual Review of Statistics and Its Application\/}~{\em 6\/}(1),
  405--431.

\bibitem[\protect\citeauthoryear{Panaretos and Zemel}{Panaretos and
  Zemel}{2020}]{panaretos2020invitation}
Panaretos, V.~M. and Y.~Zemel (2020).
\newblock {\em An Invitation to Statistics in Wasserstein Space}.
\newblock Springer Nature.

\bibitem[\protect\citeauthoryear{Pock and Chambolle}{Pock and
  Chambolle}{2011}]{pock2}
Pock, T. and A.~Chambolle (2011).
\newblock Diagonal {P}reconditioning for {F}irst {O}rder {P}rimal-{D}ual
  {A}lgorithms in {C}onvex {O}ptimization.
\newblock In {\em 2011 {{International Conference}} on {{Computer Vision}}},
  pp.\  1762--1769.

\bibitem[\protect\citeauthoryear{Richter-Gebert, Sturmfels, and
  Theobald}{Richter-Gebert et~al.}{2005}]{richter2005first}
Richter-Gebert, J., B.~Sturmfels, and T.~Theobald (2005).
\newblock First steps in tropical geometry.
\newblock {\em Contemporary Mathematics\/}~{\em 377}, 289--318.

\bibitem[\protect\citeauthoryear{Sato, Yamada, and Kashima}{Sato
  et~al.}{2020}]{sato2020fast}
Sato, R., M.~Yamada, and H.~Kashima (2020).
\newblock Fast {U}nbalanced {O}ptimal {T}ransport on {T}ree.
\newblock {\em arXiv:2006.02703\/}.

\bibitem[\protect\citeauthoryear{Sommerfeld and Munk}{Sommerfeld and
  Munk}{2018}]{doi:10.1111/rssb.12236}
Sommerfeld, M. and A.~Munk (2018).
\newblock Inference for empirical {W}asserstein distances on finite spaces.
\newblock {\em Journal of the Royal Statistical Society: Series B (Statistical
  Methodology)\/}~{\em 80\/}(1), 219--238.

\bibitem[\protect\citeauthoryear{Speyer and Sturmfels}{Speyer and
  Sturmfels}{2004}]{Speyer2004}
Speyer, D. and B.~Sturmfels (2004).
\newblock The {T}ropical {G}rassmannian.
\newblock {\em Advances in Geometry\/}~{\em 4\/}(3).

\bibitem[\protect\citeauthoryear{Tang, Wang, and Yoshida}{Tang
  et~al.}{2020}]{tang2020tropical}
Tang, X., H.~Wang, and R.~Yoshida (2020).
\newblock Tropical {S}upport {V}ector {M}achine and its {A}pplications to
  {P}hylogenomics.
\newblock {\em arXiv:2003.00677\/}.

\bibitem[\protect\citeauthoryear{Tran}{Tran}{2018}]{tran2018tropical}
Tran, N.~M. (2018).
\newblock Tropical {G}aussians: {A} {B}rief {S}urvey.
\newblock {\em arXiv:1808.10843\/}.

\bibitem[\protect\citeauthoryear{Villani}{Villani}{2003}]{villani2003topics}
Villani, C. (2003).
\newblock {\em Topics in {O}ptimal {T}ransportation}.
\newblock Number~58. American Mathematical Soc.

\bibitem[\protect\citeauthoryear{Villani}{Villani}{2008}]{villani2008optimal}
Villani, C. (2008).
\newblock {\em Optimal {T}ransport: {O}ld and {N}ew}, Volume 338.
\newblock Springer Science \& Business Media.

\bibitem[\protect\citeauthoryear{Wasserman}{Wasserman}{2019}]{wasserman_opt}
Wasserman, L. (2019, April).
\newblock Lecture notes on {S}tatistical {M}ethods for {M}achine {L}earning.

\bibitem[\protect\citeauthoryear{Yoshida, Zhang, and Zhang}{Yoshida
  et~al.}{2019}]{yoshida2019tropical}
Yoshida, R., L.~Zhang, and X.~Zhang (2019).
\newblock Tropical principal component analysis and its application to
  phylogenetics.
\newblock {\em Bulletin of Mathematical Biology\/}~{\em 81\/}(2), 568--597.

\end{thebibliography}
\bibliographystyle{chicago}  % or choose another bib list style

%%%%%%%%%%%%%%%%%%%%%%%%%%%%%%%%%%%%%%%%%%%%%%%%%%%
%%%%%%%%%%%%%%%%%%%%%%%%%%%%%%%%%%%%%%%%%%%%%%%%%%%
%%%%%%%%%%%%%%%%%%%%%%%%%%%%%%%%%%%%%%%%%%%%%%%%%%%

\end{document}